\documentclass[10pt,reqno]{amsart}
\usepackage{fullpage}
\usepackage[margin=1in]{geometry}
\usepackage{amsmath}
\usepackage{amsfonts}
\usepackage{amssymb}
\usepackage{amsthm,mathtools,mathabx,accents,}
\usepackage{newlfont}
\usepackage{graphicx}
\usepackage{xcolor,enumerate,mathrsfs, scalerel}
\usepackage{esint}
\usepackage{hyperref}
\newtheorem{thm}{Theorem}[section]
\newtheorem{cor}[thm]{Corollary}
\newtheorem{lem}[thm]{Lemma}
\newtheorem{prop}[thm]{Proposition}

\theoremstyle{definition}

\newtheorem{rem}[thm]{Remark}
\numberwithin{equation}{section}

\newcommand{\norm}[1]{\Vert#1\Vert}

\newcommand{\na}{\nabla}
\newcommand{\pa}{\partial}
\newcommand{\lec}{\lesssim}
\newcommand{\td}{\tilde}

\renewcommand{\div}{\operatorname{div}}

\newcommand\al{\alpha}

\newcommand\de{\delta}
\newcommand\De{\Delta}

\newcommand\Ga{\Gamma}

\newcommand{\la}{\lambda}

\newcommand{\eps}{\varepsilon}

\newcommand{\T}{\mathbb{T}}
\newcommand{\R}{\mathbb{R}}
\newcommand{\Z}{\mathbb{Z}}
\newcommand{\N}{\mathbb{N}}

\newcommand{\cR}{\mathcal{R}}

\newcommand{\supp}{\operatorname{supp}}

\newcommand{\I}{\operatorname{Id}}

\usepackage[usestackEOL]{stackengine}
\def\dint{\,\ThisStyle{\ensurestackMath{%
  \stackinset{c}{.2\LMpt}{c}{.5\LMpt}{\SavedStyle-}{\SavedStyle\phantom{\int}}}%
  \setbox0=\hbox{$\SavedStyle\int\,$}\kern-\wd0}\int}
  
\begin{document}

\title{An Onsager type theorem for the Euler-Boussinesq \\equations in two spatial dimensions}

\author{Ujjwal Koley}
\date{\today}
\address{\parbox{\linewidth}{
Ujjwal Koley \\
Centre for Applicable Mathematics, Tata Institute of Fundamental Research\\
P.O. Box 6503, GKVK Post Office, Bangalore 560065, India\\
E-mail address: ujjwal@math.tifrbng.res.in
}
} 
 
%\thanks{The author acknowledges the support of the Department of Atomic Energy,  Government of India, under project no.$12$-R$\&$D-TFR-$5.01$-$0520$, and DST-SERB SJF grant DST/SJF/MS/$2021$/$44$.}

%\subjclass[2010]{Primary: . Secondary: }
% \keywords{} 
\begin{abstract} 
In this article, we construct non-trivial weak solutions $(v, \theta)$ to the inviscid Euler-Boussinesq system in two spatial dimensions. These solutions exhibit compact temporal support, thereby violating the conservation of the temperature’s
$L^p$-norm. Furthermore, the pair $(v, \theta)$ resides in the H\"older space $C^\gamma(\mathbb R \times \mathbb T^2) \times C^\gamma (\mathbb R \times \mathbb T^2)$ for any exponent $\gamma<1/3$. The methodology integrates a Nash iteration scheme with a linear decoupling technique, as first established in \cite{vikram}, to achieve these results.
\end{abstract}

\maketitle

\tableofcontents

\section{Introduction} 
In this manuscript, we consider the following two-dimensional Euler-Boussinesq system of equations
\begin{equation} \label{ie}
    \begin{cases}
        \partial_t v + \div(v \otimes v) + \nabla p = \theta e_2, \\ 
        \div v = 0, \\
        \partial_t \theta + \div (v \theta)=0,
    \end{cases}
\end{equation}
in the periodic setting $x \in \T^2 := \mathbb R^2 / \mathbb Z^2$, where $v: \R\times \T^2 \to \R^2$ is a velocity field, $\theta: \R\times \T^2 \to \R$ denotes the temperature which is a scalar function, $p: \R\times \T^2 \to \R$ is the scalar pressure, and $e_2$ is the standard unit vector i.e., $e_2= (0,1)^T$. The Euler-Boussinesq system of equations describe the convection phenomena in the ocean or atmosphere, and considered to be an excellent model for many geophysical flows, such as ocean circulations and atmospheric fronts (see \cite{Majda,Pedlosky}). 

Note that when $\theta \equiv 0$, the Euler-Boussinesq system \eqref{ie} reduces to the two-dimensional incompressible Euler equations. The existence and uniqueness of classical (local-in-time) solutions for the incompressible Euler equations are well-established in the literature (see \cite{Majda01}). Over the past two decades, significant advances have been made in the study of global-in-time weak solutions for these equations. Among the numerous developments in this field, one of the most prominent problems is Onsager's conjecture for the Euler equations. In his seminal 1949 work \cite{Onsager49}, Lars Onsager predicted that the threshold H\"older regularity required for weak solutions of the incompressible Euler equations to conserve energy is the exponent $1/3$. This conjecture has since motivated extensive research into the interplay between regularity and energy conservation in hydrodynamic models. More precisely, Onsager conjectured 
\begin{itemize}
\item [(a)] $C^{0,\alpha}$ solutions must conserve kinetic energy for $\alpha>1/3$.
\item [(b)] For any $\alpha<1/3$, there exists non-conservative (dissipative) weak solutions in the H\"older class $C^{0,\alpha}$.
\end{itemize}
Note that the first part of the aforementioned conjecture was established by Constantin et. al. in \cite{ConstantinETiti94} (see also \cite{CCFS08}), while the flexible component of the problem, which is notably intricate, has been addressed by multiple researchers. The foundational work in this direction was initiated by V. Scheffer (\cite{Scheffer93}), followed by contributions from A. Shnirelman (\cite{Shnirelman00}). A pivotal breakthrough in developing H\"older continuous solutions was achieved through the work of Camillo De Lellis and László Szekelyhidi. In their foundational paper \cite{DLS13}, they devised an iterative scheme that leveraged Beltrami flows on the three-dimensional torus $\T^3$ and the Geometric Lemma, successfully constructing continuous periodic solutions with prescribed kinetic energy. This framework was later refined to produce \cite{DLS14} H\"older continuous periodic solutions with exponents $\alpha < 1/10$, while maintaining the specified energy profile. Since then, a significant body of work has been devoted to this area (\cite{BDLIS15, BDLS16, cltv, BMNV21, DaneriSzekelyhidi17, IsetRegulTime}; see also survey articles \cite{001,002}), with Isett \cite{Isett18} ultimately achieving a key breakthrough by constructing weak solutions in the space $C_t(C^{\alpha})$, with $\alpha<1/3$, that possess compact support in time and do not conserve the total kinetic energy.

We note that, unlike the classical Euler equations, there are relatively few results in the literature for the system \eqref{ie}. In this context, we first highlight the works of Chae et al. \cite{Chae01, Chae02}, where the authors established the local existence and uniqueness of smooth solutions to \eqref{ie} and provided a blow-up criterion for these solutions.
Moreover, any (sufficiently) smooth solution $(v, \theta)$ of \eqref{ie} with initial data $(v_0, \theta_0)$ satisfies following two identities for all $t \in \R^{+}$ and $p\ge 1$:
\begin{align}
\label{ONSA}
\|v(t)\|^2_{L^2(\T^2)} = \|v_0\|^2_{L^2(\T^2)} + \int_0^t \int_{\T^2} \theta(x,s) v_2(x,s)\,dx, \quad \|\theta(t)\|_{L^p(\T^2)} = \|\theta_0\|_{L^p(\T^2)}.
\end{align}

In this work, we focus on weak solutions of \eqref{ie}, and for the existence of such solutions, we refer to the work by Tao and Zhang \cite{Tao}. To present their result precisely, we first recall the concept of weak solutions to \eqref{ie}.
A weak solution is defined as a tuple $(v,p, \theta) \in (L_{\mathrm{loc}}^{\infty}(\R; L^2(\T^2)))^3$ that satisfies \eqref{ie} in the distributional sense, meaning the following:
\begin{itemize}
\item [(a)] For all $\varphi \in C_c^{\infty} ( \R\times \T^2;\R^2)$, with $\div \varphi =0$
\begin{align*}
\int_{\R} \int_{\T^2} \big( v \cdot \partial_t \varphi + v \otimes v : \nabla \varphi + \theta e_2 \cdot \varphi \big)\,dx\,dt=0.
\end{align*}
\item [(b)] For all $\phi \in C_c^{\infty} (\R \times \T^2; \R)$
\begin{align*}
\int_{\R} \int_{\T^2} \big( \theta \partial_t \phi + v \theta \cdot \nabla \phi \big)\,dx\,dt=0.
\end{align*}
\item [(c)] For any $\psi \in C^{\infty}(\T^2; \R)$, and any time $t \in \R$,
\begin{align*}
\int_{\T^2} v(t,x) \cdot \nabla \psi(x) \,dx=0.
\end{align*}
\end{itemize}
The authors in \cite{Tao} studied the system \eqref{ie} and constructed H\"older continuous weak solutions
$(v, \theta) \in C_{t,x}^{\frac{1}{28} -}(\R \times \R^2) \times C_{t,x}^{\frac{1}{25} -}(\R \times \R^2)$ with compact support in space-time. Their approach relied on a multi-step iteration scheme with one-dimensional oscillation, incorporating plane waves of varying frequencies along the same direction to eliminate one stress error component per step. To fully remove stress errors, the process was iterated over multiple steps. Due to the coupling of velocity and temperature, two stress errors were eliminated simultaneously.

On the other hand, drawing inspiration from Onsager's conjecture for the incompressible Euler equations, there has been considerable interest in exploring Onsager-type theorems for a variety of fluid flow models. Within this framework, a natural question arises for the Euler-Boussinesq system: what is the critical H\"older exponent at which a weak solution satisfies \eqref{ONSA}? Indeed, a recent study by Miao et al. \cite[Theorem 1.3]{Miao} established the following Onsager-type theorem on the 3-dimensional torus $\T^3$:
\begin{thm}
\label{thm01}
Consider a weak solution of the Euler-Boussinesq system in the H\"older space $C_t C^{\gamma}(\T^3)$. Then
\begin{itemize}
\item [(a)] If $\gamma>1/3$, then $\|\theta(t)\|_{L^p(\T^3)}$ is conserved for all $1 \le p \le \infty$.
\item [(a)] If $\gamma<1/3$, there are solutions violating conservation of $L^p$-norm of the temperature.
\end{itemize}
\end{thm}
Their approach relies on intermittent cuboid flows, which are supported within disjoint cuboids aligned in multiple directions and essentially function as Mikado flows, similar to those described in \cite{DaneriSzekelyhidi17}. To fully leverage this advantage and prevent the interaction of waves oscillating in different directions, the method employs the gluing approximation technique introduced in \cite{Isett18}. 
We mention that the work of Tao and Zhang \cite{Tao} also establishes the existence of H\"older continuous solutions in two spatial dimensions that violate \eqref{ONSA}. However, unlike the results in \cite{Miao}, the threshold exponent for \eqref{ONSA} provided by \cite{Tao} deviates significantly from the optimal threshold of $1/3$. 

Note that the arguments employed by Miao et al. \cite{Miao} to establish the rigidity part of Theorem~\ref{thm01} are independent of the dimension. Consequently, these arguments also suffice to prove the rigidity part in the two-dimensional case. The primary objective of this article is to address the flexibility gap for \eqref{ie} in two spatial dimensions. Inspired by the earlier work of Giri and Radu \cite{vikram} on the incompressible Euler equations, we successfully construct Onsager-critical regular solutions for \eqref{ie} on the 2-dimensional torus that violate \eqref{ONSA}. We emphasize that our findings significantly diverge from all prior results on the Euler-Boussinesq equations \eqref{ie}. Notably, any solution to the two-dimensional Euler-Boussinesq equations can be effortlessly extended to a $k$-dimensional solution for $k\ge2$.

More precisely, our main result is encapsulated in the following theorem:

\begin{thm}[\textbf{Main Theorem}]\label{main_thm}
Given any $0\leq \gamma<1/3$, there exist compactly supported (in time), non-trivial, weak solutions $(v, \theta)$ to \eqref{ie} in the H\"older space $C^\gamma(\mathbb R \times \mathbb T^2) \times C^\gamma(\mathbb R \times \mathbb T^2)$.
\end{thm}

The proof of the above theorem relies on a variant of classical Nash iteration technique \cite{Nash}. In short, this technique is inductive in nature and at each inductive step highly oscillatory perturbations are added to reduce the ``error'' associated to the inductive scheme. In particular, errors are decomposed into finitely many elementary errors corresponding to a finite set of directions. However, to achieve the critical regularity, one needs to guarantee that perturbations corresponding to different directions do not interact with each other. This non-intersection property can be enforced by choosing the so-called Mikado building blocks in dimensions $d \ge 3$, but not in dimension $d=2$ (since any two non-parallel lines must intersect in two dimensions). To overcome this fundamental problem, the authors in \cite{vikram} (see also Cheskidov-Luo \cite{cluo}) introduced a novel technique which make use of time oscillations to achieve the desired non-intersection property. Loosely speaking, they first introduced an additional Newton iteration step by solving a linearized Euler equations with an oscillatory force term. This step, combined with the well-known Nash iteration step, helped them to achieve Onsager-critical regularity threshold also in two space dimensions.

Following the approach outlined in \cite{vikram}, we have also established the essential non-intersection property of the underlying building blocks. However, the primary challenge in this work, compared to that of Giri and Radu \cite{vikram}, arises from estimating the additional terms introduced by the genuine interactions between the velocity field and the temperature. Specifically, due to this complexity, the error cancellation mechanism required for both the Newton and Nash iteration steps will differ significantly from the previous case. To be more specific:

First, for the Newton iteration step, it is necessary to define the Newtonian linearization of the Euler-Boussinesq system of equations, incorporating suitable temporal oscillations, and to establish a well-posedness theory for smooth solutions of the linearized Euler-Boussinesq equations. Additionally, to obtain optimal estimates for the Newton perturbations, appropriate stream functions must be defined for both the velocity and temperature perturbations. These solutions, corresponding to two distinct stream functions, are then iterated jointly. In fact, this approach not only refines the error cancellation mechanisms but also addresses the unique challenges posed by the interplay between velocity and temperature fields in the Euler-Boussinesq system. For further details, refer to Section~\ref{section_Newton}.

Second, for the Nash iteration step, the velocity perturbations must be carefully designed to consist of two components: one addressing the error term $R_q$, which arises from the self-interaction of the velocity field $v$, and the other responsible for reducing the error term $T_q$, which stems from the interaction between $v$ and the temperature field 
$\theta$. Building on the approach introduced in \cite{GK}, this is achieved by utilizing two geometric lemmas linked to different sets of directions. For a more detailed explanation, see Lemma~\ref{geom} and Lemma~\ref{lem:geo2}. Additionally, precise estimates of various new terms present in \eqref{ie} are required to achieve optimal Nash errors. For a comprehensive discussion, refer to Section~\ref{section_Nash}.

Finally, we highlight a few related Onsager-type results for fluid flow equations. A sharp Onsager-type result for SQG was established by Dai-Giri-Radu \cite{Dai} and Looi-Isett \cite{Isett}. An intermittent Onsager theorem for the Euler equations was proven by Novack-Vicol \cite{Novack}, while a wavelet-inspired $L^3$-based strong Onsager theorem was reported by Giri-Kwon-Novack \cite{vikram_01} (see also \cite{vikram_02}).
 
A brief description of the organization of the rest of this paper is as follows: In Section~\ref{pre}, we present some of the mathematical framework used in this paper. In Section~\ref{Section_Main_Prop}, we have first stated the induction scheme for constructing Euler-Boussinesq-Reynolds system, and then stated the main Proposition~\ref{Main_prop} which plays a crucial role in the proof of Main Theorem~\ref{main_thm}. We collect all preliminary results, required for the Newton step, in Section~\ref{Prelim}, while we give details of Newton perturbations, and estimates for the total Newton perturbations in Section~\ref{section_Newton}. All the details related to Nash iterations, in particular, the construction of the density and velocity perturbations, estimates on the Nash perturbations, the mollification process are reported in Section~\ref{section_Nash}. Finally, we present a well-posedness result for the linearized Euler-Boussinesq system of equations in Appendix~\ref{wl-psd} .

\section{Preliminaries and Mathematical Framework}
\label{pre}
In this section, we recapitulate some of the relevant mathematical tools that are used in the subsequent analysis. To begin, we use the letter $C$ to denote various generic constants, independent of approximation parameters, that may change from line to line along the proofs. For the sake of readability, we do not track explicit constants. However, when necessary, we will explicitly indicate the dependence of $C$ on the relevant parameters in our arguments. Furthermore, throughout this manuscript, we adopt the following convention for any two parameter-dependent quantities $A_q$ and $B_q$: we write $A\lec B $ to mean that $A\leq CB$ for some constant $C$ that is independent of the parameter $q$. We denote by $\mathbb P$, the  standard Leray projection operator onto divergence-free vector fields, the standard inner product in $\mathbb R^2$ will be denoted by $\langle \,, \rangle$, while $\otimes$ denotes the usual tensor product and $\mathring \otimes$ denotes the trace-less part of it. Furthermore, by $\supp_t (g)$, we denote the temporal support of the function $g$ on $[0, T]\times \T^3$, i.e., $\supp_t (g) := \{t: \exists \ x \;\mbox{with}\; g (t,x) \neq 0\}$.

We set $\N_0 = \N \cup \{0\}$. For $\alpha \in [0,1)$, $M \in \N_0$, we denote supremum norm, H\"{o}lder semi-norm and norm of a vector or a tensor field $g:\mathbb T^2 \rightarrow \mathbb R$ by
\begin{align*}
&  \|g\|_0 := \sup_{x \in \mathbb T^2} |g(x)|, \qquad  [g]_M = \sup_{|\theta|=M} \|D^\theta g\|_0, \quad 
[g]_{M+\alpha} = \sup_{|\theta|= M} \sup_{x \neq y} \frac{|D^\theta g(x) - D^\theta g(y)|}{|x-y|^\alpha}, \\
& \|g\|_M = \sum_{j = 0}^M [g]_j, \hspace{1.15cm}  \|g\|_{M+\alpha} = \|g\|_M + [g]_{M+\alpha}.
\end{align*}
By an abuse of notation, for a time-dependent function $g$, we denote 
\begin{equation*}
    \|g\|_M = \sup_t \|g(\cdot, t)\|_M, \quad \|g\|_{M+\alpha} = \sup_t \|g(\cdot, t)\|_{M+\alpha}.
\end{equation*}
Moreover, to denote spatial H\"older norms at a specific time slice, we will use the same notation as in the preceding line.
We now recall the following classical interpolation inequality 
\begin{equation*}
    \|g\|_{M+\alpha} \lesssim  \|g\|_{M_1 + \alpha_1}^\lambda \|g\|_{M_2 + \alpha_2}^{1-\lambda}, \text{with}\,\,
    M + \alpha = \lambda(M_1 + \alpha_1) + (1-\lambda)(M_2 + \alpha_2),
\end{equation*}
where the implicit bound is dependent on $M$, $M_1$, $M_2$, $\alpha$, $\alpha_1$ and $\alpha_2$. Recall also the classical estimate for product of functions
\begin{equation*}
    \|fg\|_{M+\alpha} \lesssim (\|f\|_{M+\alpha}\|g\|_0 + \|f\|_0\|g\|_{M+\alpha}).
\end{equation*}

\subsection{Convex integration}

Here we recall some of the well-known auxiliary lemmas which are used in the subsequent analysis. We start with the lemma regarding H\"older estimates for composition of functions. For a proof, we refer to \cite{DLS14}.

\begin{prop} \label{comp_estim}
For any two given smooth functions $\Psi:\Omega \rightarrow \mathbb R$ and $u: \mathbb R^n \rightarrow \Omega$ for some $\Omega \subset \mathbb R^N$, and for any $M \in \mathbb N \setminus \{0\}$, we have the following estimates:
    \begin{align*}
        [\Psi \circ u]_M &\lesssim \big([\Psi]_1 \|Du\|_{M-1} + \|D\Psi\|_{M-1} \|u\|_0^{M-1} \|u\|_M \big), \\
           [\Psi \circ u]_M &\lesssim \big([\Psi]_1 \|Du\|_{M-1} + \|D\Psi\|_{M-1} [u]_1^M \big).
    \end{align*}
Here all the implicit constants depend on $M, N, n$.
\end{prop}
Next, we recall two lemmas regarding mollification estimates. For a detailed proof, consult \cite{CDLS12}. 

\begin{prop} \label{moli}
For any standard symmetric mollifier $\Psi$ and any smooth function $f \in C^\infty(\mathbb T^2)$, we have the following estimate:
\begin{equation*}
    \|f - f*\Psi_\ell\|_M \lesssim \ell^2 \|f\|_{M+2}, \,\, \text{for any}\,\, M \geq 0.
\end{equation*}
Here the implicit bound in only dependent on $M$.
\end{prop}

\begin{prop} \label{CET_comm}
For any standard symmetric mollifier $\Psi$ and smooth functions $f, g \in C^\infty(\mathbb T^2)$, we have the following estimate:
\begin{equation*}
    \|(fg)*\Psi_\ell - (f*\Psi_\ell)(g*\Psi_\ell)\|_{N} \lesssim \ell^{2 - N + M}\big( \|f\|_{M+1}\|g\|_1 + \|f\|_{1}\|g\|_{M+1} \big),
\end{equation*}
for $N \geq M \geq 0$. Here the implicit bound is dependent on $N$ and $M$. 
\end{prop}
Note that the case $M=0$ of Proposition~\ref{CET_comm} is proved in \cite{CDLS12}, while the case $M>0$ is proved in \cite{vikram}.
Next, we recall two geometric lemmas which play pivotal roles in the construction of velocity and density perturbations in the subsequent analysis. For a proof of these lemmas, we refer to \cite{S12,Kwon1}. Notice that, Lemma~\ref{lem:geo2} is stated in \cite{Kwon1} on $\mathbb T^3$, but a closer look reveals that the same argument can be used for $\mathbb T^2$. 

\begin{lem} \label{geom}
Let $B_{1/2}(\I)$ represent the metric ball centered at the identity matrix $\I$ in the space $\mathcal{S}^{2 \times 2}$ of symmetric $2 \times 2$ matrices. Then, there is a finite set $\Lambda_R \subset \mathbb Z^2$ and smooth functions $\gamma_\xi: B_{1/2}(\I) \rightarrow \mathbb{R}$, for each $\xi \in \Lambda_R$, such that
\begin{equation*}
    R = \sum_{\xi \in \Lambda_R} \gamma_\xi^2 (R) \xi \otimes \xi,
\end{equation*}
whenever $R \in B_{1/2}(\I)$. 
\end{lem}
 
\begin{lem}
\label{lem:geo2}
Suppose $\Lambda_T := \{ \xi_1, \xi_2, \xi_3 \} \subset \Z^2\setminus\{0\}$ be such that $\{\xi_1, \xi_2\}$ is an orthogonal frame and $\xi_3=-(\xi_1+ \xi_2)$.
Then for any $N_0>0$, there are affine functions $\{\Ga_{\xi_i}\}_{1\leq i \leq 3} \subset C^\infty(\cal{W}_{N_0}; [N_0, \infty))$ with domain $\cal{W}_{N_0} := \{v\in \R^2: |v| \leq N_0\}$ such that
\[
v= \sum_{i=1}^3 \Ga_{\xi_i} (v) \xi_i,  \quad\forall v \in \cal{W}_{N_0}\, .
\] 
\end{lem}
In order to define new Reynolds stresses, we need to ``invert the divergence'' of various vector fields. To that context, we shall use the following inverse-divergence operator (first introduced in \cite{ChDLS12}):
\begin{equation} \label{invdiv}
    (\mathcal{R} u)^{ij} =  \Delta^{-1} (\partial_i u^j + \partial_j  u^i - \div u \,\delta_{ij}),
\end{equation}
which maps smooth, mean-zero, vector fields to smooth, trace-free, and symmetric $2$-tensors. We now state the following known result.
\begin{prop}
If $u$ is a smooth, mean-zero vector field, then the $2$-tensor field $\mathcal{R} u$ defined by \eqref{invdiv} is symmetric and satisfies 
\begin{equation*}
    \div \mathcal{R} u = u.  
\end{equation*}
\end{prop} 
We also need an inverse divergence operator which maps a mean-zero scalar function to a mean-zero vector-valued one. To this end, we abuse the notation and define 
\[
(\cR f)_i = \De^{-1} \pa_i f. 
\]
Indeed, for a mean-zero scalar function, we have $\div \cR f =f $. 
We also recall the following version of the stationary phase lemma. For a proof, we refer the reader to \cite{DaneriSzekelyhidi17}. 
\begin{prop}
\label{prop.inv.div}
    Let $\alpha \in (0,1)$ and $M \geq 1$. Let $a \in C^\infty(\mathbb{T}^2)$, $\Psi \in C^\infty(\mathbb{T}^2;\mathbb{R}^2)$ be smooth functions and assume that
    \[C^{-1} \leq |\na \Psi| \leq C\]
    holds on $\mathbb{T}^2$. Then
    \begin{equation}
        \left|\int_{\T^2} a(x)e^{ik\cdot\Psi}\,dx\right| \lesssim \frac{\|a\|_M + \|a\|_0\|\nabla \Psi\|_M}{|k|^M}\,,
    \end{equation}
    and for the operator $\mathcal{R}$ defined in \eqref{invdiv}, we have
    \begin{equation}
        \left\|\mathcal{R}\left(a(x)e^{ik\cdot\Psi}\right)\right\|_\alpha \lesssim \frac{\|a\|_0}{|k|^{1-\alpha}} + \frac{\|a\|_{M+\alpha} + \|a\|_0\|\nabla \Psi\|_{M+\alpha}}{|k|^{M-\alpha}}
    \end{equation}
    where the implicit constants depend on $C$, $\alpha$ and $M$, but not on $k$.
\end{prop}

Finally, we recall the following proposition on commutator estimates. For a proof of this lemma, we refer to \cite[Prop. H.1]{BDLIS15}.
\begin{prop}
\label{prop.comm}
Let $k$ be a fixed element in $\Z^2\setminus \{0\}$, $\alpha \in (0,1)$, and set $G= b(x) e^{i \lambda k \cdot x}$ for a smooth vector field 
$b \in C^{\infty}(\T^2; \R^2)$. Then for any smooth function $a$, we have
\begin{align*}
\left\| [a, \mathcal{R}] (G) \right\|_{\alpha} \le \lambda^{\alpha-2} \|b\|_0 \|a\|_1 + C \lambda^{\alpha-m} \left(\|b\|_{m-1+\alpha} \|a\|_{1 +\alpha}  + \|b\|_\alpha \|a\|_{m+\alpha}\right),
\end{align*}
where the constant $C$ depends on $\alpha$ and $m$.
\end{prop}

\subsubsection{Transport equations}

We recall some of the useful estimates related to the classical transport equation
\begin{equation} 
\label{transport}
\partial_t f + w \cdot \nabla f = g, \quad f \big | _{t = t_0} = f_0. 
\end{equation}
For a proof of the following proposition, we refer to \cite{BDLIS15}. 

\begin{prop} \label{transport_estim}
Assume that $|t-t_0| \|w\|_1 \leq 1$. Then every solution $f$ to the transport equation \eqref{transport} satisfies following estimates:
\begin{equation*}
    \|f(\cdot, t)\|_0 \leq \|f_0\|_0 + \int_{t_0}^t \|g(\cdot, \tau)\|_0 d \tau, \quad 
    \|f(\cdot, t)\|_\alpha \leq 2 \big( \|f_0\|_\alpha + \int_{t_0}^t \|g(\cdot, \tau)\|_\alpha d\tau\big),
\end{equation*}
for $\alpha \in [0,1]$.
Moreover, for any $M \geq 1$ and $\alpha \in [0,1)$,
\begin{equation*}
[f(\cdot, t)]_{M+\alpha} \lesssim [f_0]_{M+\alpha} + |t-t_0| [u]_{M+\alpha} [f_0]_1 + \int_{t_0}^t \big( [g(\cdot, \tau)]_{M+\alpha} + (t-\tau) [u]_{M+\alpha} [g(\cdot, \tau)]_1\big) d\tau.
\end{equation*}
Here the implicit constant depends on $M$ and $\alpha$. As a by product, the backwards flow $\Psi$, defined in \eqref{Flow_t}, of the velocity field $w$ starting at $t=t_0$ satisfies 
\begin{equation*}
    \|\nabla \Psi(\cdot, t) - \I\|_0 \lesssim |t-t_0|[w]_1, \quad 
    [\Psi(\cdot, t)]_M \lesssim  |t - t_0| [w]_M, \,\,\,  \forall M \geq 2. 
\end{equation*}
\end{prop}

\subsubsection{Singular integral operators}
We recall a classical result on Calder\'on-Zygmund operators. In what follows, we consider the following Calder\'on-Zygmund operator ($\mathbb T^2$-periodic) acting on mean-zero periodic functions $f$: 
\begin{equation*}
    S_K f(x) = \text{PV} \int_{\mathbb T^2} K_{\mathbb T^2}(x-y)\,f(y)\, dy.
\end{equation*}
Here $K_{\mathbb T^2}$ denotes the periodized version of a smooth (away from the origin), zero mean (on circles centred at the origin), homogeneous kernel $K:\mathbb R^2 \rightarrow \mathbb R$: 
\begin{equation*}
    K_{\mathbb T^2}(z) = K(z) + \sum_{n \in \mathbb Z^2 \setminus \{0\}} (K(z+n) - K(n)).
\end{equation*}
We can now state the following classical result, for a proof see \cite{CZ}. 

\begin{prop}
The $\mathbb T^2$-periodic Calder\'on-Zygmund operators are bounded on the space of periodic, zero mean $C^\alpha$ functions, for any $\alpha \in (0,1)$.
\end{prop}
We also state the following useful estimate for the commutator. We refer to \cite{constantin} for a proof. 
\begin{prop} 
 \label{CZ_comm}
Let $S_K$ be a Calder\'on-Zygmund operator as mentioned above, and $e \in C^{M+\al}(\T^2)$ be a given vector field. Then, for any $f \in C^{M+\al}(\T^2)$, $\al \in (0,1)$ and $M\geq 0$, we have
     \begin{align*}
         \|[S_K, e \cdot \na] f \|_{M+\al} \lec \|e\|_{1+\al} \|f\|_{M+\al} + \|e\|_{M+1+\al} \|f\|_{\al}.
     \end{align*}
Here the implicit constant depends on $\al, M$, and $K$.
\end{prop}

\section{Induction Scheme} 
\label{Section_Main_Prop}

The proof of the Main theorem \ref{main_thm} is based on an iterative procedure, and, in this section, we first describe the associated iteration scheme. To be more precise, we first assume that $q \in \mathbb N$ denotes different levels of the iteration, and the smooth tuple $(v_q, p_q, \theta_q, R_q, T_q)$ solves following the Euler-Boussinesq-Reynolds system at the level $q$
\begin{equation} \label{ER}
    \begin{cases}
        \partial_t v_q + \div(v_q \otimes v_q) + \nabla p_q - \theta_q e_2 = \div R_q, \\ 
        \div v_q = 0, \\
        \partial_t \theta_q + \div(v_q \theta_q) = \div T_q.
    \end{cases}
\end{equation}
Here $R_q$ is a symmetric $2$-tensor field, and transport error $T_q$ is a vector field. Assuming the ``error $(R_q, T_q)$'' is suitably small, our main job is find a new Euler-Boussinesq-Reynolds system at the level $q+1$ with new ``error $(R_{q+1}, T_{q+1})$'' which is substantially small compared to that of level $q$. This is typically done by adding suitable perturbations to velocity $v_q$, pressure $p_q$ and temperature $\theta_q$. Indeed, goal is to construct the sequence $\{(v_q, p_q, \theta_q)\}$ so that the errors $R_q$ and $T_q$ converge to zero, and $(v_q, p_q, \theta_q)$ converges in the required H\"older space. Hence, in the limit, we will recover a weak solution to the Euler-Boussinesq system of equations. 

\subsection{Main proposition} 

Here we state the key iterative proposition which in turn gives the proof of the Main theorem \ref{main_thm}. Let us first begin by fixing parameters $a>1$, $b>1$, to be chosen later, and define frequency and amplitude parameters as
\begin{equation*}
    \lambda_q = 2\pi \lceil a^{b^q} \rceil, \quad \delta_q = \lambda_q^{-2\beta}.
\end{equation*}
Here the parameter $\beta \in (0,1/3)$ is related to the H\"older regularity of the solution.  

Now, utilizing the above parameters, we make the following inductive assumptions:
\begin{align} 
& \|v_q\|_0 \leq  M(1 - \delta_q^{1/2}), \hspace{1.7cm}  \|\theta_q\|_0 \leq  M(1 - \delta_q^{1/2}), \label{ind_est_1} \\
&\|v_q\|_N \leq M \delta_q^{1/2} \lambda_q^N, \hspace{2cm} \|\theta_q\|_N \leq M \delta_q^{1/2} \lambda_q^N, \hspace{1.9cm}  \, \forall N \in \{1,2,..., L\}, \label{ind_est_2} \\
&\|p_q\|_N \leq M^2 \delta_q \lambda_q^N, \hspace{6.9cm}\, \,\forall N \in \{1,2,..., L\}, \label{ind_est_3} \\
&\|R_q\|_N \leq \delta_{q+1} \lambda_q^{N-2\alpha}, \qquad \qquad \,\,\,\,\, \|T_q\|_N \leq \delta_{q+1} \lambda_q^{N-2\alpha}, \hspace{1.5cm}\,\,\, \forall N \in \{0, 1,..., L\}, \label{ind_est_4} \\
&\|D_t R_q\|_{N} \leq \delta_{q+1} \delta_q^{1/2}\lambda_q^{N+1 - 2\alpha}, \quad \|D_t T_q\|_{N} \leq \delta_{q+1} \delta_q^{1/2}\lambda_q^{N+1 - 2\alpha}, \,\,\,\,\, \forall N \in \{0, 1,..., L-1\}, \label{ind_est_5}
\end{align}
where $L \in \mathbb N \setminus \{0\}$, $M > 0$ and $0< \alpha < 1$. Throughout the paper, we reserve the notation $D_t$ for the material derivative associated to the velocity field $v_q$, i.e., $D_t = \partial_t + v_q \cdot \nabla$. Moreover, regarding the temporal support of both Reynolds stresses, we assume: 
\begin{equation} 
\label{ind_est_6}
     \supp_t R_q, \,\, \supp_t T_q \subset [-2 + (\delta_q^{1/2} \lambda_q)^{-1}, -1 - (\delta_q^{1/2} \lambda_q)^{-1}] \cup [1+(\de_q^{1/2}\la_q)^{-1}, 2-(\de_q^{1/2}\la_q)^{-1}],
\end{equation}
Here we also assume that $a$ is sufficiently large so that $(\delta_0^{1/2} \lambda_0)^{-1} < \frac{1}{4}$.

The central iterative proposition of this paper is as follows:
\begin{prop} 
\label{Main_prop}
Assume that $L \geq 4$, $0< \beta < 1/3$, and $1 < b < \frac{1 + 3\beta}{6\beta}$. There exist constants $M_0>1$, depending on $\beta$ and $L$, and $\exists~ 0< \alpha_0 < 1$, depending on $\beta$ and $b$, such that the following holds: For any $0<\alpha < \alpha_0$ and $M>M_0$, there exists a constant $a_0 > 1$, depending on $\beta$, $b$, $\alpha$, $M_0$, $M$ and $L$, such that for any $a > a_0$, if $(v_q, p_q, \theta_q, R_q, T_q)$ is a smooth solution of \eqref{ER} satisfying \eqref{ind_est_1} - \eqref{ind_est_6}, then we can construct a smooth solution $(v_{q+1}, p_{q+1}, \theta_{q+1}, R_{q+1}, T_{q+1})$ of \eqref{ER} satisfying the same estimates \eqref{ind_est_1} - \eqref{ind_est_6} with $q$ replaced by $q+1$. Moreover, we also have
\begin{equation} 
\label{Main_prop_eqn}
    \|v_{q+1} - v_{q}\|_0 + \frac{1}{\lambda_{q+1}} \|v_{q+1} - v_{q}\|_1 + \|\theta_{q+1} - \theta_{q}\|_0 + \frac{1}{\lambda_{q+1}} \|\theta_{q+1} - \theta_{q}\|_1 \leq 2M \delta_{q+1}^{1/2}, 
\end{equation}
and 
\begin{equation} \label{Main_prop_eqn_2}
   \supp_t (v_{q+1} - v_q) \subset (-2, -1) \cup (1, 2), \quad  \supp_t (\theta_{q+1} - \theta_q) \subset (-2, -1) \cup (1, 2).
\end{equation}
\end{prop}
Next, we prove the main theorem \ref{main_thm}, assuming the proposition \ref{Main_prop}.

\subsection{Proof of the Main theorem~\ref{main_thm}, assuming Proposition~\ref{Main_prop}}

First of all, we fix the parameter $L = 4$, and any H\"older exponent $\beta < 1/3$ such that $\gamma < \beta$, and the parameter $b$ as in the proposition~\ref{Main_prop}. Moreover, let $\alpha_0$ and $a_0$ be the constants given by the proposition~\ref{Main_prop}. Finally, we fix $\alpha < \min \{ \alpha_0, 1/4\}$, and the parameter $a > a_0$ will be fixed at the end of the proof.

Let us now define smooth functions $f, g : \R \to [0,1]$, supported in $[-7/4, 7/4]$, such that $f = 1$ on $[-5/4, 5/4]$ and $g=0$ on $[-5/4, 5/4]$. Then consider
\begin{align*}
 v_0 (x,t) &= f(t) \delta_0^{1/2}\cos ( \la_0 x_1) e_2, \,\,\,\,\, R_0(x,t) = \big(f'(t)-g(t) \big) \frac{\delta_0^{1/2}}{\la_0}  \begin{pmatrix}
        0 & \sin ( \la_0 x_1)\\ \sin(\la_0 x_1) & 0 
    \end{pmatrix}, \\
   \theta_0(x,t) &=g(t) \delta_0^{1/2}\cos ( \la_0 x_1), \quad p_0(x,t) = 0, \quad T_0(x,t) =\Big(\frac{1}{\la_0} g'(t) \delta_0^{1/2}\sin( \la_0 x_1), 0 \Big),  
\end{align*}
where $(x_1, x_2) \in \T^2$. In view of the above choices, it is easy to verify that $(v_0,p_0,\theta_0,R_0,T_0)$ solves the Euler-Boussinesq-Reynolds system~\eqref{ER}. Note that, for sufficiently large choice of $a$, we have 
\begin{equation*}
    \|v_0\|_0 \leq M\delta_0^{1/2} \leq M(1 - \delta_0^{1/2}), \quad \|\theta_0\|_0 \leq M\delta_0^{1/2} \leq M(1 - \delta_0^{1/2}),
\end{equation*}
Thus, the estimates \eqref{ind_est_1} are satisfied. Moreover, for any $N \geq 1$, we have
\begin{equation*}
    \|v_0\|_N \leq M\delta_0^{1/2} \lambda_0^N, \quad  \|\theta_0\|_N \leq M\delta_0^{1/2} \lambda_0^N,
\end{equation*}
so that \eqref{ind_est_2} and \eqref{ind_est_3} also satisfy. For the Reynolds stresses, note that, for any $N \geq 0$,
\begin{equation*}
    \|R_0\|_N \leq 2 \sup_t \big( |f'(t)| + |g(t)| \big) \frac{\delta_0^{1/2}}{\lambda_0} \lambda_0^N, \quad \|R_0\|_N \leq  \sup_t |g'(t)| \frac{\delta_0^{1/2}}{\lambda_0} \lambda_0^N.
\end{equation*}
Since by our choice $(2b-1)\beta <1/3$, we can always choose $a$ large enough to make sure that
\begin{equation*}
    2 \sup_t \big( |f'(t)| + |g(t)| \big) + \sup_t |g'(t)| < \delta_1 \delta_0^{-1/2} \lambda_0^{1/2},
\end{equation*}
so that, thanks to the choice of $\alpha < 1/4$, the estimate \eqref{ind_est_4} holds. Indeed, we have
\begin{equation*}
    \|R_0\|_N \leq \delta_1 \lambda_0^{-1/2} \lambda_0^N, \quad \|T_0\|_N \leq \delta_1 \lambda_0^{-1/2} \lambda_0^N,
\end{equation*}
For the material derivative estimates, we first see that
\begin{align*}
    \partial_t R_0 + u_0 \cdot \nabla R_0 &= \big( f''(t) -g'(t) \big) \frac{\delta_0^{1/2}}{\lambda_0} \begin{pmatrix}
        0 & \sin ( \la_0 x_1)\\ \sin(\la_0 x_1) & 0
    \end{pmatrix}, \\
     \partial_t T_0 + u_0 \cdot \nabla T_0 & = g''(t) \frac{\delta_0^{1/2}}{\lambda_0} \big(  \sin ( \la_0 x_1), 0\big).
\end{align*}
Again as before, we may choose $a$ large enough so that \eqref{ind_est_5} is satisfied.
Finally, notice that $\supp_t R_0 \subset [-7/4, 7/4] \setminus (-5/4, 5/4)$, and $\supp_t T_0 \subset [-7/4, 7/4] \setminus (-5/4, 5/4)$. Hence, the condition \eqref{ind_est_6} is satisfied, thanks to our choice $(\delta_0^{1/2} \lambda_0)^{-1} < \frac{1}{4}$.

At this point, we fix $a$ large enough so that all the above estimates are satisfied. This confirms that the tuple $(v_0, p_0, \theta_0, R_0, T_0)$ satisfies all the inductive estimates in Proposition~\ref{Main_prop}. Next, let us assume that $\theta_q$ has mean-zero, and $\{(v_q, p_q, \theta_q, R_q, T_q)\}$ be the sequence of solutions given by the Proposition~\ref{Main_prop}. Then a simple calculation reveals that both $\{v_q\}$ and $\{\theta_q\}$ are Cauchy sequences in the H\"older space $C_tC^\gamma_x$, and hence they converge to a velocity field $v$ and a temperature $\theta$, respectively. Indeed, we have
\begin{align*}
& \|v_{q+1} - v_q\|_\gamma \lesssim \|v_{q+1} - v_q\|_0^{1-\gamma} \|v_{q+1} - v_q\|_1^\gamma \lesssim \delta_{q+1}^{1/2} \lambda_{q+1}^\gamma \lesssim \lambda_{q+1}^{\gamma - \beta}, \\
&\|\theta_{q+1} - \theta_q\|_\gamma \lesssim \|\theta_{q+1} - \theta_q\|_0^{1-\gamma} \|\theta_{q+1} - \theta_q\|_1^\gamma \lesssim \delta_{q+1}^{1/2} \lambda_{q+1}^\gamma \lesssim \lambda_{q+1}^{\gamma - \beta}
\end{align*}
Moreover, regarding the Reynolds stresses, we have
\begin{align*}
\|R_q\|_\gamma \lesssim \|R_q\|_0^{1-\gamma}\|R_q\|_1^{\gamma} \lesssim \delta_{q+1}\lambda_q^{\gamma} \lesssim \lambda_{q+1}^{\gamma - 2 \beta}, \quad 
\|T_q\|_\gamma \lesssim \|T_q\|_0^{1-\gamma}\|T_q\|_1^{\gamma} \lesssim \delta_{q+1}\lambda_q^{\gamma} \lesssim \lambda_{q+1}^{\gamma - 2 \beta}.
\end{align*}
Therefore, we conclude that both $R_q$ and $T_q$ converge to zero in $C_tC^{\gamma}_x$. Finally, since $p_q$ satisfies 
\begin{equation*}
    \Delta p_q = \div \div(-v_q \otimes v_q + R_q + \mathcal{R}(\theta_q e_2)),
\end{equation*}
we conclude that $p_q - \fint p_q$ converges, in $C_tC_x^\gamma$, to some $p$, and $\nabla p_q \rightarrow \nabla p$ in the sense of distributions. Therefore $(v,p, \theta)$ is a weak solution to the Euler-Boussinesq system of equations with $(v, \theta) \in C_tC^\gamma_x \times C_tC^\gamma_x$. Finally, regarding the temporal support, we see that $\supp_t v \subset [-2, 2]$ and $\supp_t \theta \subset [-2, 2]$. 

To prove the time regularity of both velocity and the temperature, we follow the argument given in \cite{cltv}. To proceed, we first fix a smooth standard mollifier $\zeta$ in space variable only, and consider $\bar v_q := v*\zeta_{2^{-q}}$, $\bar \theta_q := \theta*\zeta_{2^{-q}}$ for $q \in \N$, where $\zeta_{\ell}(x) = \ell^{-2} \zeta(x \ell^{-1})$. Let us recall the well-known mollification estimates 
\begin{align}
\norm{\bar v_q-v}_0\lesssim \norm{v}_{\gamma} 2^{-q \gamma}, \quad \norm{\bar \theta_q- \theta}_0\lesssim \norm{\theta}_{\gamma} 2^{-q \gamma}.
\label{e:trivial:1}
\end{align}
 Above estimates clearly imply that $\bar v_q - v \to 0$, and $\bar \theta_q - \theta \to 0$ uniformly as $q \to \infty$. Next, we mollify the equations \eqref{ie} to obtain
 \begin{equation} \label{ie_new}
    \begin{cases}
        \partial_t \bar v_{q}+\div\left(v  \otimes v \right)*\psi_{2^{-q}}  +\nabla p * \psi_{2^{-q}} = \theta* \psi_{2^{-q}} e_2, \\ 
        \div \bar v_{q} = 0, \\
        \partial_t \bar \theta_q + \div (v \theta)* \psi_{2^{-q}}=0,
    \end{cases}
\end{equation}
A classical Schauder's estimate, for any fixed $\varepsilon>0$, reveals that
\[
\|\nabla p * \psi_{2^{-q}}\|_0 \leq \|\nabla p * \psi_{2^{-q}}\|_\varepsilon 
\lesssim \|v\|^2_{\gamma} \,2^{q(1 + \eps - \gamma)} + \|\theta\|_{\gamma} \,2^{q(\eps - \gamma)} \lesssim \big(\|v\|^2_{\gamma} + \|\theta\|_{\gamma} \big)2^{q(1 + \eps - \gamma)}\, ,
\]
Moreover, we have
\[
\norm{ \left(v  \otimes v \right)*\psi_{2^{-q}}}_1 \lesssim \norm{v}_{\gamma}^2  2^{q(1-\gamma)}, \quad  
\norm{ \left(v \theta \right)*\psi_{2^{-q}}}_1 \lesssim \norm{v}_{\gamma} \norm{\theta}_{\gamma}  2^{q(1-\gamma)}\,.
\]
Hence, the equations in \eqref{ie_new} imply that
\begin{align}
\norm{\partial_t \bar  v_{q}}_0 \lesssim \big(\|v\|^2_{\gamma} + \|\theta\|_{\gamma} \big)2^{q(1 + \eps - \gamma)}, \quad 
\norm{\partial_t \bar  \theta_{q}}_0 \lesssim \|v\|_{\gamma} \|\theta\|_{\gamma} 2^{q(1 - \gamma)} \,.
\label{e:trivial:2}
\end{align}
Notice that, all the implicit constants in the above estimates depend on $\varepsilon$ but not on $q$. We can now conclude, thanks to \eqref{e:trivial:1} and \eqref{e:trivial:2} and $\gamma' < \gamma$
\begin{align*}
\norm{\bar v_q - \bar v_{q+1}}_{C^0_x C^{\gamma'}_t}  
&\lesssim \left( \norm{\bar v_q - v}_0 + \norm{\bar v_{q+1} - v}_0 \right)^{1-\gamma'}
\left(\norm{\partial_t \bar v_q}_0  + \norm{\partial_t \bar v_{q+1}}_0\right)^{\gamma'}
\notag\\
&\lesssim \big(\|v\|^2_{\gamma'} + \|\theta\|_{\gamma'} + \|v\|_{\gamma'} \big) 2^{-q \eps} \\
\norm{\bar \theta_q - \bar \theta_{q+1}}_{C^0_x C^{\gamma'}_t}  
&\lesssim \left( \norm{\bar \theta_q - \theta}_0 + \norm{\bar \theta_{q+1} - \theta}_0 \right)^{1-\gamma'}
\left(\norm{\partial_t \bar \theta_q}_0  + \norm{\partial_t \bar \theta_{q+1}}_0\right)^{\gamma'}
\notag\\
&\lesssim \big(\|\theta\|_{\gamma'} + \|v\|_{\gamma'} \|\theta\|_{\gamma'} \big) 2^{-q (\gamma-\gamma')}
\end{align*}
Therefore, both series 
\[
v = \bar v_0 + \sum_{q\geq0} (\bar v_{q+1} - \bar v_q), \,\, \text{and} \,\,
\theta = \bar \theta_0 + \sum_{q\geq0} (\bar \theta_{q+1} - \bar \theta_q)
\]
converge in $C^0_x C^{\gamma'}_t$. Since we have already established that $v, \theta \in C^0_t C^{\gamma}_x$, so we obtain $v, \theta \in C^{\gamma'}(\R \times \T^2)$ with $\gamma' < \gamma < \beta < 1/3$ arbitrary. This finishes the proof of the theorem.

\section{Set up for the Newton Steps} 
\label{Prelim}

It is well-known that any typical Nash iteration scheme (see, for example, \cite{DLS13} and \cite{DLS14}) suffers from the loss of derivative problem, and to overcome this issue efficiently we need to mollify the velocity, temperature and Reynolds stresses. To that context, let us begin by defining $\zeta$ as a symmetric spatial mollifier with vanishing first moments. We then fix the spatial mollification scale as $\ell_q = (\lambda_q \lambda_{q+1})^{-1/2}$. Moreover, we denote
\begin{equation*}
    \bar v_q = v_q * \zeta_{\ell_q}, \quad  \bar \theta_q = \theta_q * \zeta_{\ell_q}, \quad 
    R_{q, 0} = R_q * \zeta_{\ell_q}, \quad T_{q, 0} = T_q * \zeta_{\ell_q}.
\end{equation*} 
As before, we shall reserve the notation $\bar D_t = (\partial_t + \bar v_q \cdot \nabla)$ for the material derivative associated to the velocity field $\bar v_q$.

Next, to construct the iterative Newton perturbations, we need to make use of various time-dependent functions. In what follows, we first fix the temporal parameter $\tau_q = \big(\delta_q^{1/2} \lambda_q \lambda_{q+1}^\alpha \big)^{-1}$. Note that with this choice of $\tau_q$, we have $ \|\bar v_q\|_1 \tau_q \leq  C \lambda_{q+1}^{-\alpha} \leq 1$ for sufficiently large $a_0$. Also notice that the following holds:
\begin{equation*}
    \|\bar v_q\|_{1+\alpha} \tau_q \lesssim \bigg(\frac{\lambda_q}{\lambda_{q+1}}\bigg)^\alpha \lesssim 1. 
\end{equation*}

Following \cite{vikram}, we also introduce two sets of partition of unity in time. For that purpose, let us first denote by $t_p = p \tau_q$, for $p \in \mathbb{Z}$. We then define cut-off functions $\{\eta_p\}_{p \in \mathbb Z}$, and $\{\tilde \eta_p\}_{p \in \mathbb Z}$ such that they satisfy the following properties:
\begin{itemize}
    \item [(a)] The squared cut-offs form a partition of unity, i.e, $\sum_{p \in \mathbb Z} \eta_p^2(t) = 1$.
    \item [(b)] $\supp \eta_p \subset (t_p - \frac{2}{3} \tau_q, t_p + \frac{2}{3} \tau_q)$. In particular, $\supp \chi_{p-1} \cap \supp \chi_{p+1} = \varnothing, \forall p \in \mathbb Z$.
    \item [(c)] For any $N \geq 0$ and $p \in \mathbb Z$, $|\partial_t^N \eta_p| \lesssim \tau_q^{-N}$. Here the implicit constant depends on $N$ only.
    \item [(d)] $\supp \tilde \eta_p \subset (t_p - \tau_q, t_p + \tau_q)$ and $\tilde \eta_p = 1$ on $(t_p - \frac{2}{3} \tau_q, t_p + \frac{2}{3} \tau_q)$. Notice that $\eta_p \Tilde \eta_p = \eta_p$, for all $p \in \mathbb Z$.
    \item [(e)] For any $N \geq 0$ and $p \in \mathbb Z$, $ |\partial_t^N \tilde \eta_p| \lesssim \tau_q^{-N}$. Here also the implicit constant depends on $N$ only.
\end{itemize}

We now define time-periodic temporal building blocks which play a pivotal role in the upcoming analysis. Note that the number of such temporally oscillatory profiles directly linked to the number of implemented Newton steps. Let the number of such Newton steps be given by
\begin{equation}\label{def:ga}
    \Gamma = \bigg\lceil \frac{1}{1/3 - \beta} \bigg \rceil.
\end{equation}
We remark that $\Gamma$ only depends on $\beta$, and independent of the iteration steps. 
We also define the temporal frequency parameter $\mu_{q+1} > \tau_q^{-1}$ by 
\begin{equation*}
    \mu_{q+1} = \delta_{q+1}^{1/2} \lambda_q^{2/3} \lambda_{q+1}^{1/3} \lambda_{q+1}^{4\alpha}.
\end{equation*}
Next, we recall the following lemma which essentially guarantees the disjointness of different building blocks. In fact, this is a crucial requirement for our analysis. For a proof of the lemma, modulo cosmetic changes, we refer to \cite[Lemma 3.3]{vikram}.
\begin{lem} 
\label{osc_prof}
Let $\Lambda_R, \Lambda_T \subset \mathbb Z^2$ be two disjoint sets given by the Lemma \ref{geom} and Lemma~\ref{lem:geo2} respectively, and $\Gamma \in \mathbb N$. Given any $\xi \in \Lambda_R$, there exist $2\Gamma$ numbers of smooth $1$-periodic functions $g_{\xi, e, n}, g_{\xi, o, n} :\mathbb R \rightarrow \mathbb R$ with $n \in \{1, 2,..., \Gamma\}$ such that
\begin{equation*}
\int_0^1 g_{\xi, p, n}^2 = 1,\ \ \forall \xi \in \Lambda_R \text{, } p \in \{e, o\} \text{, and } n \in \{1,2,..., \Gamma\}.
\end{equation*}
Moreover, for any $\zeta \in \Lambda_T$, there exist $2\Gamma$ numbers of smooth $1$-periodic functions $h_{\zeta, e, n}, h_{\zeta, o, n}:\mathbb R \rightarrow \mathbb R$ with $n \in \{1, 2,..., \Gamma\}$ such that
\begin{equation*}
 \int_0^1 h^2_{\zeta, p, n} = 1,\ \ \forall \zeta \in \Lambda_T \text{, } p \in \{e, o\} \text{, and } n \in \{1,2,..., \Gamma\}.
\end{equation*}
Furthermore, by denoting $\Lambda = \Lambda_R \cup \Lambda_T$, we have
\begin{equation*}
    \supp g_{\xi, p, n} \cap \supp g_{\zeta, q, m} \cap  \supp h_{\xi, p, n} \cap \supp h_{\zeta, q, m} = \varnothing, 
\end{equation*}
whenever $(\xi, p, n) \neq (\zeta, q, m) \in \Lambda \times \{e, o\} \times \{1, 2,..., \Gamma\}$.
\end{lem}
Before concluding this section, we collect some standard estimates related to mollified velocity, mollified temperature and Reynolds stresses in the following lemma. For a proof, we refer to \cite[Lemma 3.1]{vikram}. 
\begin{lem} \label{smoli_estim}
We have the following estimates: 
\begin{align} 
&\|\bar v_q\|_N \lesssim \delta_q^{1/2} \lambda_q^N, \hspace{2.8cm} \|\bar \theta_q\|_N \lesssim \delta_q^{1/2} \lambda_q^N, \hspace{2.6cm} \forall N \in \{1,2,...,L\}, \label{smoli_1}\\
&\|R_{q,0}\|_N \lesssim \delta_{q+1} \lambda_q^{N-2\alpha}, \hspace{1.85cm} \|T_{q,0}\|_N \lesssim \delta_{q+1} \lambda_q^{N-2\alpha}, \hspace{1.66cm} \, \forall N \in \{0, 1,..., L\}, \label{smoli_2} \\
&\|\bar D_t R_{q, 0}\|_N \lesssim \delta_{q+1} \delta_q^{1/2} \lambda_q^{N+1-2\alpha}, \quad \,\, \|\bar D_t T_{q, 0}\|_N \lesssim \delta_{q+1} \delta_q^{1/2} \lambda_q^{N+1-2\alpha},  \quad \, \forall N \in \{0, 1,..., L-1\},  \label{smoli_3} \\
&\|\bar v_q\|_{N + L} \lesssim \delta_q^{1/2} \lambda_q^{L} \ell_q^{-N}, \hspace{1.82cm}  \|\bar \theta_q\|_{N + L} \lesssim \delta_q^{1/2} \lambda_q^{L} \ell_q^{-N}, \,\hspace{1.6cm}  \forall N \geq 0, \label{smoli_4} \\
&\|R_{q,0}\|_{N+L} \lesssim \delta_{q+1} \lambda_q^{L-2\alpha} \ell_q^{-N}, \qquad  \,\,\,\|T_{q,0}\|_{N+L} \lesssim \delta_{q+1} \lambda_q^{L-2\alpha} \ell_q^{-N}, \qquad \,\forall N \geq 0, \label{smoli_5} \\
&\|\bar D_t R_{q,0}\|_{N+L-1} \lesssim \delta_{q+1} \delta_q^{1/2}\lambda_q^{L-2\alpha} \ell_q^{-N}, \quad   \|\bar D_t T_{q,0}\|_{N+L-1} \lesssim \delta_{q+1} \delta_q^{1/2}\lambda_q^{L-2\alpha} \ell_q^{-N}, \quad \forall N \geq 0. \label{smoli_6}
\end{align}
Here all the implicit constants depend on the parameters $M$, $N$, and $L$.
\end{lem}

We also collect some standard estimates on the forward and backward flows of $\bar v_q$. 
To do that, let us first introduce the backwards flow $\Psi_t : \T^2 \times \R \to \T^2$, starting at $t$, defined by
\begin{equation} \label{Flow_t}
    \begin{cases}
        \partial_s \Psi_t(x,s) + \bar v_q(x,s) \cdot \nabla \Psi_t(x,s) = 0, \\ 
        \Psi_t \big|_{s = t}(x) = x.
    \end{cases}
\end{equation}
Moreover, we introduce the forward flow $X_t: \T^2 \times \R \to \T^2$, starting at $t$, defined by 
\begin{equation} \label{Lagr_t}
    \begin{cases}
        \frac{d}{ds}X_t(\gamma, s) = \bar v_q(X_t(\gamma, s), s) \\ 
        X_t(\gamma, t) = \gamma.
    \end{cases}
\end{equation}

\begin{lem} \label{Flow_estim}
Fix any $t \in \mathbb R$. Then for any $|s - t| < \tau_q$, we have the following estimates:
    \begin{align} 
        &\|(\nabla \Psi_t)^{-1}(\cdot, s)\|_N + \|\nabla \Psi_t (\cdot, s)\|_N \lesssim \lambda_q^N, \hspace{4.4cm}\, \forall N \in\{0,1,..., L-1\}, \label{Flow_estim_1}\\ 
        &\|\bar D_t (\nabla \Psi_t)^{-1}(\cdot, s)\|_N + \|\bar D_t \nabla \Psi_t (\cdot, s)\|_N \lesssim \delta_q^{1/2} \lambda_q^{N+1}, \hspace{2.72cm} \forall N \in\{0,1,..., L-1\},  \label{Flow_estim_2} \\
         &\|D X_t(\cdot, s)\|_N  \lesssim \lambda_q^N, \hspace{7.5cm} \forall N \in \{0,1,..., L-1\},  \label{Lagr_estim_1}\\
         &\|(\nabla \Psi_t)^{-1}(\cdot, s)\|_{N+L-1} + \|\nabla \Psi_t(\cdot, s)\|_{N+L-1}  \lesssim \lambda_q^{L-1} \ell_q^{-N},\hspace{2.08cm} \forall N \geq 0, \label{Flow_estim_3}\\ 
        &\|\bar D_t (\nabla \Psi_t)^{-1}(\cdot, s)\|_{N+L-1} + \|\bar D_t \nabla \Psi_t(\cdot, s)\|_{N+L-1}  \lesssim \delta_q^{1/2} \lambda_q^L \ell_q^{-N}, \qquad \,\,\,\,\, \forall N \geq 0, \label{Flow_estim_4}\\
        &\|D X_t(\cdot, s)\|_{N+L-1} \lesssim \lambda_q^{L-1} \ell_q^{-N}, \hspace{5.87cm} \forall N \geq 0.  \label{Lagr_estim_2}
    \end{align}
Here all the implicit constants depend on the parameters $M$, $N$ and $L$.
\end{lem}

\begin{proof}
For a proof of this lemma, we refer to \cite[Lemma 3.2]{vikram}.
\end{proof}

\section{Newton Iterations} 
\label{section_Newton}

In this section, we present the construction of Newton perturbations. In what follows, we first write down the Euler-Boussinesq-Reynolds system we will have obtained after $n$ perturbation steps, for $n \in \{0, 1, ..., \Gamma - 1\}$. Indeed, it has the following form 
\begin{equation} \label{steps}
    \begin{cases}
        \partial_t v_{q, n} + \div(v_{q, n} \otimes v_{q, n}) + \nabla p_{q,n} - \theta_{q, n} e_2 = \div R_{q, n} + \div S_{q, n} + \div P_{q + 1, n} , \\
        \div v_{q, n} = 0, \\
        \partial_t \theta_{q, n} + \div(v_{q, n} \theta_{q, n}) = - \div T_{q, n} + \div X_{q, n} + \div Y_{q + 1, n},
    \end{cases}
\end{equation}
where 
\begin{itemize}
    \item [(a)] $v_{q, n}$ is the velocity to be defined inductively starting from $v_{q, 0} = v_q$ by adding $n$ perturbations, and similarly $\theta_{q, n}$ is the temperature to be defined inductively starting from $\theta_{q, 0} = \theta_q$ by adding $n$ perturbations; 
    \item [(b)] $p_{q, n}$ is the pressure to be defined starting from $p_{q, 0} = p_q$; 
    \item [(c)] $R_{q, n}$, and $T_{q, n}$ are gluing errors of the $n$-th perturbation, for $n \ge 1$, while $R_{q, 0}$ is and $T_{q, 0}$ are already defined mollified stresses;
    \item [(d)] $S_{q, n}$ and $X_{q, n}$ are inductively defined errors to be erased 
    by highly-oscillatory Nash perturbations, starting from $S_{q, 0} = 0$ and $X_{q, 0} = 0$;
    \item [(e)]$ P_{q + 1, n}$ and $Y_{q + 1, n}$ are inductively defined (sufficiently) small residue errors, starting from $P_{q+1, 0} = R_q - R_{q, 0}$ and $Y_{q+1, 0} = T_q - T_{q, 0}$.
\end{itemize}
Notice that for $n = 0$, \eqref{steps} gives back the Euler-Boussinesq-Reynolds system \eqref{ER}.

We are now ready to describe the construction of Newton perturbations $w_{q+1, n+1}^{(t)}$ and $\theta_{q+1, n+1}^{(t)}$ for the system \eqref{steps}. In what follows, we first decompose the stresses $R_{q, n}$ and $T_{q,n}$ by using the geometric lemmas, Lemma~\ref{geom} and Lemma~\ref{lem:geo2}, respectively. To fix the idea, let $n \in \{0,1,...,\Gamma-1\}$, $k \in \mathbb Z$ and $\zeta \in \Lambda_T$, and $\Psi_k$ be the backward flow characterized by 
$ \partial_t \Psi_k + \bar v_q \cdot \nabla \Psi_k = 0, \,\Psi_k (t_k,x) = x$. Then we define
\begin{equation}\label{def.b.coeff}
b_{\zeta, k, n} = \lambda_q^{-\alpha/2}\delta_{q+1, n}^{1/2} \eta_k \Gamma^{1/2}_\zeta \bigg(\lambda_q^{\alpha}\delta_{q+1, n}^{-1} \nabla \Psi_k T_{q, n} \bigg),
\end{equation}
where the smooth functions $\Gamma_\zeta$ are given by Lemma \ref{lem:geo2}, and $\delta_{q+1, n}$ are defined by 
\begin{equation*}
    \delta_{q+1, n} = \delta_{q+1} \big(\lambda_q/\lambda_{q+1}\big)^{n(1/3 - \beta)}.
\end{equation*}
\noindent Let us also denote by 
\begin{equation*}
    \mathbb Z_{q,n} = \big\{ k \in \mathbb Z \mid k \tau_q \in \mathcal{N}_{\tau_q}(\supp_t R_{q,n} \cap (\bigcup_{k' \in \Z} \supp_t b_{\zeta, k', n}))\big\},
\end{equation*}
where $\mathcal N_{\tau}(B)$ represents the neighbourhood of a set $B$ of size $\tau$. Then for any $t \in \supp_t R_{q,n}$, and $t \in \bigcup_{k' \in \Z} \supp_t b_{\zeta, k', n}$, it holds that 
$\displaystyle \sum_{k \in \mathbb Z_{q,n}} \eta_k^2(t) = 1$.

\noindent Moreover, for $n \in \{0,1,...,\Gamma-1\}$, $k \in \mathbb Z$ and $\xi \in \Lambda_R$, we define
\begin{align}
\label{def.a.coeff}
a_{\xi, k, n} =  \delta_{q+1, n}^{1/2} \eta_k \gamma_\xi \bigg(\nabla \Psi_k \nabla \Psi_k^T & - \nabla \Psi_k \frac{R_{q, n}}{\delta_{q+1,n}} \nabla \Psi_k^T \\
& - \nabla \Psi_k \Big(\sum_{\zeta \in \Lambda_T} \sum_{k' \in \mathbb Z_{q,n}} \delta_{q+1, n}^{-1}  b^2_{\zeta, k', n} (\nabla \Psi_{k'} )^{-1}(\zeta \otimes \zeta) (\nabla \Psi_{k'} )^{-T}\Big) \nabla \Psi_k^T \bigg), \nonumber
\end{align}
where the smooth functions $\gamma_\xi$ are given by Lemma \ref{geom}. To obtain the above expressions for $a_{\xi, k, n}$, and $b_{\zeta, k, n}$, we have employed similar strategies as in \cite{GK}.

\noindent In view of the expression for $\{b_{\zeta, k, n}: \zeta \in \Lambda_T\}$ in \eqref{def.b.coeff}, Proposition \ref{NewIter}, and Lemma \ref{Flow_estim}, it is clear that
$$
\| \nabla \Psi_k \Big(\sum_{\zeta \in \Lambda_T} \sum_{k' \in \mathbb Z_{q,n}} \delta_{q+1, n}^{-1}  b^2_{\zeta, k', n} (\nabla \Psi_{k'} )^{-1}(\zeta \otimes \zeta) (\nabla \Psi_{k'} )^{-T}\Big) \nabla \Psi_k^T \| \lesssim \lambda_q^{-\alpha}
$$
Therefore, arguing as in \cite{vikram}, for sufficiently large $a_0$ and any $\alpha>0$, we conclude that
$$\nabla \Psi_k \nabla \Psi_k^T - \nabla \Psi_k \frac{R_{q, n}}{\delta_{q+1,n}} \nabla \Psi_k^T - \nabla \Psi_k \Big(\sum_{\zeta \in \Lambda_T} \sum_{k' \in \mathbb Z_{q,n}} \delta_{q+1, n}^{-1}  b^2_{\zeta, k', n} (\nabla \Psi_{k'} )^{-1}(\zeta \otimes \zeta) (\nabla \Psi_{k'} )^{-T}\Big) \nabla \Psi_k^T \in B_{1/2}(\I).
$$
In other words, $a_{\xi, k, n}$ is indeed well-defined for $\xi \in \Lambda_R$. Similarly, we can show that $b_{\zeta, k, n}$ is also well-defined for $\zeta \in \Lambda_T$. Indeed, Proposition \ref{NewIter} stated below will ensure that
$\|T_{q, n}\|_0 \leq \delta_{q+1, n} \lambda_q^{-\alpha}$.
Therefore, on $\supp \eta_k$, making use of Lemma \ref{Flow_estim}, we conclude that
\begin{equation*}
    \|\lambda_q^{\alpha}\delta_{q+1, n}^{-1} \nabla \Psi_k T_{q, n} \|_0 \lesssim C_0, \,\, \text{for some constant}\,\, C_0.
\end{equation*}
The above bound allows us to use the Lemma~\ref{lem:geo2} with $N_0 = C_0$.

In view of Lemma~\ref{geom}, it follows that 
\begin{align*}
\sum_{k \in \mathbb Z_{q,n}}  \sum_{\xi \in \Lambda_R} & a^2_{\xi, k, n} (\nabla \Psi_k)^{-1}  \xi \otimes \xi (\nabla \Psi_k)^{-T} 
= \sum_{k \in \mathbb Z_{q,n}} \eta_k^2  (\delta_{q+1, n} \I - R_{q,n}) \\
& \qquad - \sum_{k \in \mathbb Z_{q,n}} \eta_k^2  \sum_{\zeta \in \Lambda_T} \sum_{k' \in \mathbb Z_{q,n}}  b^2_{\zeta, k', n} (\nabla \Psi_{k'} )^{-1}(\zeta \otimes \zeta) (\nabla \Psi_{k'} )^{-T} \\
&= \sum_{k \in \mathbb Z_{q,n}} \eta_k^2  (\delta_{q+1, n} \I - R_{q,n}) - \sum_{\zeta \in \Lambda_T} \sum_{k' \in \mathbb Z_{q,n}}  b^2_{\zeta, k', n} (\nabla \Psi_{k'} )^{-1}(\zeta \otimes \zeta) (\nabla \Psi_{k'} )^{-T},
\end{align*}
therefore, we conclude
\begin{align} \label{decomp}
   & \div \Big[ \sum_{k \in \mathbb Z_{q,n}} \bigg( \sum_{\xi \in \Lambda_R} a^2_{\xi, k, n} (\nabla \Psi_k)^{-1} \xi \otimes \xi (\nabla \Psi_k)^{-T} +  \sum_{\zeta \in \Lambda_T} b^2_{\zeta, k, n} (\nabla \Psi_k)^{-1} \zeta \otimes \zeta (\nabla \Psi_k)^{-T} \bigg)\Big]  \\ \notag
    &\qquad = \div \Big(\sum_{k \in \mathbb Z_{q,n}} \eta_k^2 (\delta_{q+1, n} \I - R_{q,n}) \Big) = - \div R_{q,n}. 
\end{align}
Moreover, in view of Lemma~\ref{lem:geo2}, it follows that 
\begin{equation} 
\label{decomp1}
\div \Big[ \sum_{k \in \mathbb Z_{q,n}} \sum_{\zeta \in \Lambda_T} b^2_{\zeta, k, n}  (\nabla \Psi_k)^{-1} \zeta \Big]  = \div \Big(\sum_{k \in \mathbb Z_{q,n}} \eta_k^2 T_{q,n}\Big) =  \div T_{q,n}. 
\end{equation} 
For ease of notations, we denote
\begin{align*}
 A_{\sigma, k, n} &= a_{\xi, k, n}^2  (\nabla \Psi_k)^{-1} \xi \otimes \xi (\nabla \Psi_k)^{-T} + b_{\zeta, k, n}^2  (\nabla \Psi_k)^{-1} \zeta \otimes \zeta (\nabla \Psi_k)^{-T}, \,\, \text{for}\,\, \sigma \in \Lambda, \\
 &:= A^1_{\xi, k, n} + A^2_{\zeta, k, n}, \\
 B_{\zeta, k, n} &= b^2_{\zeta, k, n} (\nabla \Psi_k)^{-1} \zeta, \,\, \text{for}\,\, \zeta \in \Lambda_T.
\end{align*}

\subsection{Linearized Euler-Boussinesq system}

Let us denote by $w_{k, n + 1}$ and $\theta_{k, n + 1}$ to be solutions of the linearized Euler-Boussinesq equations with temporally oscillatory forcing. For the well-posedness theory of smooth solutions of these equations, consult Appendix \ref{wl-psd}. Note that $w_{k, n + 1}$ and $\theta_{k, n + 1}$ satisfy 
\begin{equation}
 \label{LocalNewt}
\begin{cases}
    \partial_t w_{k, n + 1} + \bar v_q \cdot \nabla w_{k, n + 1} + w_{k, n + 1} \cdot \nabla \bar v_q 
+ \nabla p_{k, n + 1} - \theta_{k, n + 1} e_2 \\
= \displaystyle \sum_{\xi \in\Lambda_R} f_{\xi, k, n+1}(\mu_{q + 1} t) \mathbb P \div A^1_{\xi, k, n}(x,t)
+  \displaystyle \sum_{\zeta \in\Lambda_T} m_{\zeta, k, n+1}(\mu_{q + 1} t) \mathbb P \div A^2_{\zeta, k, n}(x,t), \\
    \div w_{k, n + 1} = 0,   \\
    w_{k, n + 1} \big |_{t = t_k}(x) = \frac{1}{\mu_{q+1}} \left[\displaystyle \sum_{\xi \in\Lambda_R} f^{[1]}_{\xi, k, n+1} (\mu_{q+1} t_k) \mathbb P \div A^1_{\xi, k, n}(x, t_k) 
    +\displaystyle \sum_{\zeta \in\Lambda_T} m^{[1]}_{\zeta, k, n+1} (\mu_{q+1} t_k) \mathbb P \div A^2_{\zeta, k, n}(x, t_k)\right], 
    \end{cases}
\end{equation}
and
\begin{equation} \label{LocalNewt1}
\begin{cases}
\partial_t \theta_{k, n + 1} + \bar v_q \cdot \nabla \theta_{k, n + 1} + w_{k, n+1} \cdot \nabla \bar \theta_q= \displaystyle \sum_{\zeta \in\Lambda_T} m_{\zeta, k, n+1}(\mu_{q + 1} t) \div B_{\zeta, k, n}(x,t),  \\ 
   \theta_{k, n + 1} \big |_{t = t_k}(x) = \frac{1}{\mu_{q+1}}\sum_{\zeta \in\Lambda_T} m^{[1]}_{\zeta, k, n+1} (\mu_{q+1} t_k) \div B_{\zeta, k, n}(x, t_k). 
    \end{cases}
\end{equation}
Here the functions $f_{\xi, k, n+1}:\mathbb R \rightarrow \mathbb R$, and $m_{\zeta, k, n+1}:\mathbb R \rightarrow \mathbb R$ are defined by 
\begin{equation*}
    f_{\xi, k, n+1} := 1 - g^2_{\xi, k, n+1}, \,\, \text{and}\,\, m_{\zeta, k, n+1} := 1 - h^2_{\zeta, k, n+1}, 
\end{equation*}
where 
\begin{equation*}
    g_{\xi, k, n+1} = 
    \begin{cases}
        g_{\xi, e, n+1}, & \text{if } k \text{ is even}, \\ 
        g_{\xi, o, n+1},  & \text{if } k \text{ is odd}.
    \end{cases}
   \quad \text{and} \quad h_{\zeta, k, n+1} = 
    \begin{cases}
        h_{\zeta, e, n+1}, & \text{if } k \text{ is even}, \\ 
        h_{\zeta, o, n+1},  & \text{if } k \text{ is odd}.
    \end{cases}
\end{equation*}
%We also denote by the sum  
%\begin{align*}
%&\sum_{\sigma \in \Lambda} \big(g_{\xi, k, n+1}^2 + h_{\zeta, k, n+1}^2\big) A_{\sigma, k, n} \\ 
%      & \quad := \sum_{\xi \in \Lambda_R} g_{\xi, k, n+1}^2 a_{\xi, k, n}^2 (\nabla \Psi_k)^{-1} \xi  \otimes \xi (\nabla \Psi_k)^{-T}  + \sum_{\zeta \in \Lambda_T} h_{\zeta, k, n+1}^2 b_{\zeta, k, n}^2 (\nabla \Psi_k)^{-1} \zeta  \otimes \zeta (\nabla \Psi_k)^{-T}.
%\end{align*}
Moreover, the functions $f^{[1]}_{\xi, k, n+1}$ and $m^{[1]}_{\zeta, k, n+1}$ denote the primitives of the functions $f_{\xi, k, n}$ and $m_{\zeta, k, n}$ respectively. Indeed, we have
\begin{equation*}
    f^{[1]}_{\xi, k, n+1} (t) = \int_0^t  f_{\xi, k, n+1} (s) ds, \quad \text{and}\quad m^{[1]}_{\zeta, k, n+1} (t) = \int_0^t  m_{\zeta, k, n+1} (s) ds.
\end{equation*}
Notice that, thanks to the unit $L^2$ norm of $g_{\xi, k, n+1}$ and $h_{\zeta, k, n+1}$, we conclude that $f^{[1]}_{\xi, k, n+1}$ and $m^{[1]}_{\zeta, k, n+1}$ are well-defined $1$-periodic functions.

\subsection{Newton perturbations} 

In view of the above, we can now define $(n+1)$-th Newton perturbations in terms of the velocity field $w_{k, n+1}$ and temperature $\theta_{k, n+1}$ as follows:
\begin{equation*}
w^{(t)}_{q + 1 , n + 1} (x, t) = \sum_{k \in \mathbb Z_{q,n}} \tilde \eta_k (t) w_{k, n + 1}(x, t), \quad \text{and} \quad
\theta^{(t)}_{q + 1 , n + 1} (x, t) = \sum_{k \in \mathbb Z_{q,n}} \tilde \eta_k (t) \theta_{k, n + 1}(x, t)\,.
\end{equation*}
Moreover, we also define
\begin{equation*}
    p^{(t)}_{q + 1 , n + 1} (x, t) = \sum_{k \in \mathbb Z_{q,n}} \tilde \eta_k (t) p_{k, n + 1}(x, t)\,.
\end{equation*}

To compute the new error terms $R_{q,n+1}$, $S_{q,n+1}$, and $P_{q+1,n+1}$ at the level $n+1$, we just need to plug in
the velocity field $v_{q,n+1} = v_{q, n} + w_{q+1, n+1}^{(t)}$ into equation \eqref{steps}. To that context, first note that $w_{q+1, n+1}^{(t)}$ satisfies
\begin{align*}
 \partial_t w^{(t)}_{q+1,n+1} &+ \bar v_q \cdot \nabla w^{(t)}_{q+1,n+1} + w^{(t)}_{q+1,n+1} \cdot \nabla \bar v_q + \nabla p^{(t)}_{q + 1 , n + 1} - \theta^{(t)}_{q+1,n+1} e_2
= \sum_{k \in \mathbb Z_{q,n}} \partial_t \tilde \eta_k w_{k, n+1} \\
   & + \sum_{k \in \mathbb Z_{q,n}} \sum_{\xi \in\Lambda_R} \tilde \eta_k (t) f_{\xi, k,n+1}(\mu_{q + 1} t) \mathbb P \div A^1_{\xi, k, n} 
   + \sum_{k \in \mathbb Z_{q,n}} \sum_{\zeta \in\Lambda_T} \tilde \eta_k (t) m_{\zeta, k,n+1}(\mu_{q + 1} t) \mathbb P \div A^2_{\zeta, k, n}.
\end{align*}
Thanks to the fact that $\tilde \eta_k A_{\sigma, k, n} = A_{\sigma, k, n}$ for all $k \in \mathbb Z$, relation \eqref{decomp} reveals that
\begin{align*}
 &   \sum_{k \in \mathbb Z_{q,n}} \sum_{\xi \in\Lambda_R} \tilde \eta_k (t) f_{\xi, k,n+1}(\mu_{q + 1} t) \mathbb P \div A^1_{\xi, k, n} 
   + \sum_{k \in \mathbb Z_{q,n}} \sum_{\zeta \in\Lambda_T} \tilde \eta_k (t) m_{\zeta, k,n+1}(\mu_{q + 1} t) \mathbb P \div A^2_{\zeta, k, n} \\
  &=\sum_{k \in \mathbb Z_{q,n}} \sum_{\sigma \in\Lambda} \mathbb P \div A_{\sigma, k, n} -  \sum_{k \in \mathbb Z_{q,n}} \sum_{\xi \in\Lambda_R} g_{\xi, k, n+1}^2  \mathbb P \div A^1_{\xi, k, n} 
  - \sum_{k \in \mathbb Z_{q,n}} \sum_{\zeta \in\Lambda_T}  h_{\zeta, k, n+1}^2  \mathbb P \div A^2_{\zeta, k, n} \\ 
    & =  - \mathbb P \div R_{q,n} -  \sum_{k \in \mathbb Z_{q,n}} \sum_{\xi \in\Lambda_R} g_{\xi, k, n+1}^2  \mathbb P \div A^1_{\xi, k, n} 
  - \sum_{k \in \mathbb Z_{q,n}} \sum_{\zeta \in\Lambda_T}  h_{\zeta, k, n+1}^2  \mathbb P \div A^2_{\zeta, k, n}.
\end{align*}
Keeping the above in mind, we conclude that the first equation of the system \eqref{steps}, at the level $(n+1)$, is satisfied with 
\begin{align}
    u_{q, n + 1} &= u_{q, n} + w_{q+1, n + 1}^{(t)} = v_q + \sum_{m =1}^{n+1} w_{q+1, m}^{(t)}, \label{Newvelo} \\
    p_{q, n + 1} &=  p_{q, n} + p_{q+1, n + 1}^{(t)} - \langle w_{q+1, n+1}^{(t)}, \sum_{m=1}^n w_{q+1, m}^{(t)} \rangle - \frac{1}{2} |w_{q+1, n+1}^{(t)}|^2 - \langle w_{q+1, n+1}^{(t)}, v_q - \bar v_q \rangle \label{Newpres} \\
    &- \Delta^{-1} \div \Big[ \div R_{q,n} + \sum_{k \in \mathbb Z_{q,n}} \sum_{\xi \in\Lambda_R} g_{\xi, k, n+1}^2  A^1_{\xi, k, n} + \sum_{k \in \mathbb Z_{q,n}} \sum_{\zeta \in\Lambda_T}  h_{\zeta, k, n+1}^2  A^2_{\zeta, k, n} \Big] \nonumber  \\
    R_{q, n + 1} &= \mathcal{R} \sum_{k \in \mathbb Z_{q,n}} \partial_t \tilde \eta_k w_{k, n+1}, \label{NewStr} \\
     S_{q, n + 1} &= S_{q, n} -  \sum_{k \in \mathbb Z_{q,n}} \sum_{\xi \in\Lambda_R} g_{\xi, k, n+1}^2  A^1_{\xi, k, n}
  - \sum_{k \in \mathbb Z_{q,n}} \sum_{\zeta \in\Lambda_T}  h_{\zeta, k, n+1}^2  A^2_{\zeta, k, n}, \label{NewNash} \\
    P_{q + 1, n + 1} &= P_{q + 1, n} + w_{q+1,n+1}^{(t)} \mathring \otimes w_{q+1, n+1}^{(t)} + \sum_{m = 1}^{n} \big (w_{q+1, n+1}^{(t)} \mathring \otimes  w_{q+1, m}^{(t)} +  w_{q+1, m}^{(t)} \mathring \otimes w_{q+1, n+1}^{(t)} \big ) \notag \\ 
    & \qquad \qquad + (v_q - \bar v_q) \mathring \otimes w_{q+1, n+1}^{(t)} + w_{q+1, n+1}^{(t)} \mathring \otimes (v_q - \bar v_q). \label{SmalStr}
\end{align}
Since $w_{k,n+1}$ has mean zero, the new stress term $R_{q,n+1}$ defined in \eqref{NewStr} is well-defined.  

Similarly, we plug in $\theta_{q,n+1} = \theta_{q, n} + \theta_{q+1, n+1}^{(t)}$ into third equation of \eqref{steps} to compute the new error terms $T_{q,n+1}$, $X_{q,n+1}$, and $Y_{q+1,n+1}$. Notice that $\theta_{q+1, n+1}^{(t)}$ satisfies
\begin{align*}
    \partial_t \theta^{(t)}_{q+1,n+1} &+ \bar v_q \cdot \nabla \theta^{(t)}_{q+1,n+1} 
    + w^{(t)}_{q+1,n+1} \cdot \nabla \bar \theta_q \\
   & \qquad = \sum_{k \in \mathbb Z_{q,n}} \sum_{\zeta \in\Lambda_T} \tilde \eta_k (t) m_{\zeta, k,n+1}(\mu_{q + 1} t) \div B_{\zeta, k, n}
 + \sum_{k \in \mathbb Z_{q,n}} \partial_t \tilde \eta_k \theta_{k, n+1}.
\end{align*}
As before, we conclude that the third equation of the system \eqref{steps}, at the level $(n+1)^{\text{th}}$, is satisfied with 
\begin{align} 
\theta_{q, n + 1} &= \theta_{q, n} + \theta_{q+1, n + 1}^{(t)} = \theta_q + \sum_{m =1}^{n+1} \theta_{q+1, m}^{(t)}, \label{Newvelo1}\\
T_{q, n + 1} &= \mathcal{R} \sum_{k \in \mathbb Z_{q,n}} \partial_t \tilde \eta_k \theta_{k, n+1}, \label{NewStr1} \\
X_{q, n + 1} &= X_{q, n} - \sum_{k \in \mathbb Z_{q,n}} \sum_{\zeta \in \Lambda_T}   
h^2_{\zeta, k, n+1} B_{\zeta, k, n}, \label{NewNash1} \\
Y_{q + 1, n + 1} &=  Y_{q + 1, n} + w_{q+1,n+1}^{(t)}  \theta_{q+1, n+1}^{(t)} + \sum_{m = 1}^{n} \big (w_{q+1, n+1}^{(t)} \theta_{q+1, m}^{(t)} +  w_{q+1, m}^{(t)} \theta_{q+1, n+1}^{(t)} \big ) \notag \\ 
    & \qquad \qquad + (v_q - \bar v_q) \, \theta_{q+1, n+1}^{(t)} + w_{q+1, n+1}^{(t)} \,(\theta_q - \bar \theta_q). \label{SmalStr1}
\end{align}
Again, since $\theta_{k,n+1}$ has mean zero, the new stress $T_{q,n+1}$ in \eqref{NewStr1} is well-defined. We are now in a position to state and prove the main inductive proposition related to the Newton perturbations.

\begin{prop} \label{NewIter}
Assume that the Reynolds stresses $R_{q, n}$ and $T_{q, n}$ satisfy 
\begin{align} 
&\|R_{q, n}\|_{N}, \, \|T_{q, n}\|_{N} \leq \delta_{q+1, n} \lambda_q^{N - \alpha}, \hspace{4.5cm}\forall N \in \{0, 1,..., L-1\}, \label{NewIter_1} \\
&\|\bar D_t R_{q, n}\|_N, \, \|\bar D_t T_{q, n}\|_N \leq \delta_{q+1,n} \tau_q^{-1} \lambda_q^{N - \alpha}, \hspace{3.15cm} \forall N \in \{0, 1,..., L-1\}, \label{NewIter_2} \\
&\|R_{q, n}\|_{N+L-1}, \, \|T_{q, n}\|_{N+L-1} \lesssim \delta_{q+1, n} \lambda_q^{L-1 -\alpha} \ell_q^{-N}, \hspace{2.05 cm}\, \forall N \geq 0, \label{NewIter_3} \\
&\|\bar D_t R_{q,n}\|_{N + L -1}, \, \|\bar D_t T_{q,n}\|_{N + L -1} \lesssim \delta_{q+1, n} \tau_q^{-1} \lambda_q^{L-1 - \alpha} \ell_{q}^{-N}, \qquad \forall N \geq 0. \label{NewIter_4}
\end{align}
Here all the implicit constants depend on $n$, $\Gamma$, $M$, $\alpha$ and $N$. Moreover, assume that 
\begin{eqnarray} \label{NewIter_5}
    \supp_t R_{q,n}, \,\, \supp_t T_{q,n} 
    &\subset& [-2 + (\de_q^{1/2}\la_q)^{-1} - 3n\tau_q, -1 -(\de_q^{1/2}\la_q)^{-1} + 3n\tau_q] \\ 
    && \cup [1 + (\de_q^{1/2}\la_q)^{-1} - 3n \tau_q, 2 - (\de_q^{1/2}\la_q)^{-1} + 3n \tau_q], \nonumber
\end{eqnarray}
Then the Reynolds stresses $R_{q, n+1}, T_{q, n+1}$ also satisfy \eqref{NewIter_1}-\eqref{NewIter_5} with $n$ replaced by $n+1$. 
\end{prop}

\begin{rem}
We remark that the claim of \eqref{NewIter_5} is immediate, in view of the definitions of $\tilde \eta_k$, $\mathbb Z_{q,n}$ and the fact that the union $\bigcup_{k'} \supp_t b_{\xi, k', n}$ is a $\tau_q$ neighbourhood of the error $T_{q,n}$. 
Indeed, we have 
\begin{equation*}
    \supp_t R_{q, n+1}, \,\,  \supp_t T_{q, n+1} \subset \overline{\mathcal{N}_{3\tau_q} (\supp_t R_{q, n} \cap \supp_t T_{q, n})}.
\end{equation*}
\end{rem}

\subsection{Proof of the Proposition~\ref{NewIter}}

In order to furnish the proof of the Proposition~\ref{NewIter}, we first need to obtain below mentioned estimates for $a_{\xi, k, n}$, $b_{\zeta, k, n}$, $A_{\sigma, k, n}$, and $B_{\zeta, k, n}$. In what follows, we begin with the following lemma. 
\begin{lem} 
\label{a_estim}
The following estimates, thanks to the assumptions of Proposition \ref{NewIter}, hold: 
\begin{align} 
&\|a_{\xi, k, n}\|_N, \, \|b_{\zeta, k, n}\|_N  \lesssim \delta_{q+1, n}^{1/2} \lambda_q^N, \hspace{4.4cm} \,\forall N \in \{0,1,..., L-1\}, \label{a_estim_1} \\
&\|\bar D_t a_{\xi, k, n}\|_N, \, \|\bar D_t b_{\zeta, k, n}\|_N  \lesssim \delta_{q+1, n}^{1/2} \tau_q^{-1} \lambda_q^N, \hspace{3.0 cm} \, \forall N \in \{0,1,..., L-1\}, \label{a_estim_2} \\
&\|a_{\xi, k, n}\|_{N+L-1}, \,  \|b_{\zeta, k, n}\|_{N+L-1} \lesssim \delta_{q+1, n}^{1/2} \lambda_q^{L-1} \ell_q^{-N}, \hspace{2.0cm} \forall N \geq 0, \label{a_estim_3} \\
&\|\bar D_t a_{\xi, k, n}\|_{N+L-1}, \, \|\bar D_t b_{\zeta, k, n}\|_{N+L-1} \lesssim \delta_{q+1, n}^{1/2} \lambda_q^{L- 1}  \tau_q^{-1} \ell_q^{-N}, \quad \,\,\,\, \forall N \geq 0, \label{a_estim_4}
\end{align}
where all the implicit constants depend on $n$, $\Gamma$, $M$, $\alpha$, and $N$.
\end{lem}

\begin{proof}

First we deal with terms related to $b_{\zeta, k, n}$. In fact, in view of the estimates of Lemma \ref{Flow_estim} and 
$\supp b_{\zeta, k, n} \subset \supp \eta_k \times \mathbb T^2$, we have (on $\supp \eta_k$)
    \begin{equation*}
        \left\| \lambda_q^{\alpha} \delta_{q+1,n}^{-1} \nabla \Psi_k \,T_{q,n} \right\|_0 \lesssim 1.
    \end{equation*}
Making use of Proposition \ref{comp_estim}, on $\supp b_{\zeta, k, n}$, we get
    \begin{eqnarray*}
        \|b_{\zeta, k, n}\|_N & \lesssim & \lambda_q^{-\alpha/2} \delta_{q+1, n}^{1/2} \|\ \lambda_q^{\alpha} \delta_{q+1,n}^{-1} \nabla \Psi_k \,T_{q,n}\|_N \\ 
        & \lesssim & \lambda_q^{-\alpha/2} \delta_{q + 1, n}^{1/2} \lambda_q^{\alpha} \delta_{q+1,n}^{-1} \big ( \|\nabla \Psi_k\|_N \| T_{q, n}\|_0 + \| \nabla \Psi_k\|_0 \| T_{q,n}\|_N \big), 
    \end{eqnarray*}
    for any $N > 0$. Therefore, Lemma~\ref{Flow_estim} and the assumptions of proposition \ref{NewIter} confirms the estimates related to $b_{\zeta, k, n}$ in \eqref{a_estim_1} and \eqref{a_estim_3}. Note that, same arguments also conclude that $b_{\zeta, k, n}^2$ satisfies the same estimates as $b_{\zeta, k, n}$ with $\delta_{q, n+1}^{1/2}$ replaced by $\delta_{q, n+1}$. To deal with the term $a_{\xi, k, n}$, let us first introduce the following short-hand notation:
    \begin{align*}
    \mathcal{E}_{\zeta, k, k', n}:= \nabla \Psi_k \Big(\sum_{\zeta \in \Lambda_T} \sum_{k' \in \mathbb Z_{q,n}} \delta_{q+1, n}^{-1}  b^2_{\zeta, k', n} (\nabla \Psi_{k'} )^{-1}(\zeta \otimes \zeta) (\nabla \Psi_{k'} )^{-T}\Big) \nabla \Psi_k^T.
    \end{align*}
Then, on $\supp \eta_k$, we have
    \begin{align*}
        \Big\|\nabla \Psi_k \nabla \Psi_k^T - \nabla \Psi_k \frac{R_{q,n}}{\delta_{q+1,n}} \nabla \Psi_k^T\Big\|_0 &\lesssim 1, \quad \|\mathcal{E}_{\xi, k, k', n}\|_0 \lesssim 1.
    \end{align*}
Therefore, again, making use of the Proposition~\ref{comp_estim}, we obtain
    \begin{align*}
        \|a_{\xi, k, n}\|_N & \lesssim \delta_{q+1, n}^{1/2} \Big(\|\nabla \Psi_k \nabla \Psi_k^T - \nabla \Psi_k \frac{R_{q,n}}{\delta_{q+1,n}} \nabla \Psi_k^T \|_N + \|\mathcal{E}_{\zeta, k, k', n}\|_N \Big)\\
        & \qquad \lesssim \delta_{q + 1, n}^{1/2} \Big ( \|\nabla \Psi_k\|_N +  \| \nabla \Psi_k\|_N \|\frac{R_{q,n}}{\delta_{q+1,n}}\|_0 + \| \nabla \Psi_k\|_0 \|\frac{R_{q,n}}{\delta_{q+1,n}}\|_N \\
        &  \qquad \quad+ \|\nabla \Psi_k\|_N \Big\|\sum_{\zeta \in \Lambda_T} \sum_{k' \in \mathbb Z_{q,n}} \delta_{q+1, n}^{-1}  b^2_{\zeta, k', n} (\nabla \Psi_{k'} )^{-1}(\zeta \otimes \zeta) (\nabla \Psi_{k'} )^{-T} \Big\|_0 \\
        & \qquad \qquad+ \delta_{q+1, n}^{-1}  \|\nabla \Psi_k\|^2_0 \| b^2_{\zeta, k', n}\|_N + \delta_{q+1, n}^{-1} \| \nabla \Psi_k\|_0 \| \nabla \Psi_k\|_N \| b^2_{\zeta, k', n}\|_{0} \Big), 
    \end{align*}
    for any $N > 0$. Once again, Lemma \ref{Flow_estim} and the assumptions of Proposition \ref{NewIter} gives the rest of the required estimates in \eqref{a_estim_1} and \eqref{a_estim_3}. 

For the material derivative estimates, we first note that
    \begin{eqnarray*}
      \|\bar D_t b_{\zeta, k, n}\|_N &\lesssim & \lambda_q^{-\alpha/2} \delta_{q+1, n}^{1/2}|\partial_t \eta_k| \|\Gamma^{1/2}_\zeta \big(\lambda_q^{\alpha}\delta_{q+1, n}^{-1} \nabla \Psi_k T_{q, n} \big) \|_N \\
         &+&  \lambda_q^{-\alpha/2} \delta_{q+1,n}^{1/2} \|D\Gamma^{1/2}_\zeta \big(\lambda_q^{\alpha}\delta_{q+1, n}^{-1} \nabla \Psi_k T_{q, n} \big)\|_N \|\bar D_t \big(\lambda_q^{\alpha}\delta_{q+1, n}^{-1} \nabla \Psi_k T_{q, n} \big) \|_0 \\
        & + &\lambda_q^{-\alpha/2}  \delta_{q+1,n}^{1/2} \|D\Gamma^{1/2}_\zeta \big(\lambda_q^{\alpha}\delta_{q+1, n}^{-1} \nabla \Psi_k T_{q, n} \big)\|_0 \|\bar D_t \big(\lambda_q^{\alpha}\delta_{q+1, n}^{-1} \nabla \Psi_k T_{q, n} \big) \|_N.
    \end{eqnarray*}
Moreover,
\begin{eqnarray*}
      \|\bar D_t b^2_{\zeta, k, n}\|_N &\lesssim & \lambda_q^{-\alpha} \delta_{q+1, n}|\partial_t \eta^2_k| \|\Gamma_\zeta \big(\lambda_q^{\alpha}\delta_{q+1, n}^{-1} \nabla \Psi_k T_{q, n} \big) \|_N \\
         &+&  \lambda_q^{-\alpha} \delta_{q+1,n} \|D\Gamma_\zeta \big(\lambda_q^{\alpha}\delta_{q+1, n}^{-1} \nabla \Psi_k T_{q, n} \big)\|_N \|\bar D_t \big(\lambda_q^{\alpha}\delta_{q+1, n}^{-1} \nabla \Psi_k T_{q, n} \big) \|_0 \\
        & + &\lambda_q^{-\alpha}  \delta_{q+1,n} \|D\Gamma_\zeta \big(\lambda_q^{\alpha}\delta_{q+1, n}^{-1} \nabla \Psi_k T_{q, n} \big)\|_0 \|\bar D_t \big(\lambda_q^{\alpha}\delta_{q+1, n}^{-1} \nabla \Psi_k T_{q, n} \big) \|_N.
\end{eqnarray*}
Furthermore,
    \begin{align*}
&\|\bar D_t a_{\xi, k, n}\|_N \lesssim \delta_{q+1, n}^{1/2}|\partial_t \eta_k| \|\gamma_\xi \big(\nabla \Psi_k \nabla \Psi_k^T - \nabla \Psi_k \frac{R_{q,n}}{\delta_{q+1,n}} \nabla \Psi_k^T - \mathcal{E}_{\zeta, k, k', n} \|_N \\
&+  \delta_{q+1,n}^{1/2} \|D\gamma_\xi \big( \nabla \Psi_k \nabla \Psi_k^T - \nabla \Psi_k \frac{R_{q,n}}{\delta_{q+1,n}} \nabla \Psi_k^T - \mathcal{E}_{\zeta, k, k', n} \big)\|_N 
\|\bar D_t \big(\nabla \Psi_k \nabla \Psi_k^T - \nabla \Psi_k \frac{R_{q,n}}{\delta_{q+1,n}} \nabla \Psi_k^T  - \mathcal{E}_{\zeta, k, k', n}  \big)\|_0 \\
& +  \delta_{q+1,n}^{1/2} \|D\gamma_\xi \big( \nabla \Psi_k \nabla \Psi_k^T - \nabla \Psi_k \frac{R_{q,n}}{\delta_{q+1,n}} \nabla \Psi_k^T - \mathcal{E}_{\zeta, k, k', n} \big)\|_0 
 \|\bar D_t \big(\nabla \Psi_k \nabla \Psi_k^T - \nabla \Psi_k \frac{R_{q,n}}{\delta_{q+1,n}} \nabla \Psi_k^T - \mathcal{E}_{\zeta, k, k', n}  \big)\|_N.
    \end{align*}
In the above expressions, all the terms not involving the material derivatives can be handled as before, thanks to Proposition \ref{comp_estim}. For the rest of the terms, we calculate
    \begin{align*}
        \|\bar D_t \big( \nabla \Psi_k \frac{R_{q,n}}{\delta_{q+1,n}} \nabla \Psi_k^T  \big) \|_N & \lesssim  \|\bar D_t \frac{R_{q,n}}{\delta_{q+1,n}}\|_N  + \|\bar D_t \frac{R_{q,n}}{\delta_{q+1,n}}\|_0 \|\nabla \Psi_k\|_N + \| \frac{R_{q,n}}{\delta_{q+1,n}}\|_N \|\bar D_t \nabla \Psi_k\|_0  \\ 
         & \qquad + \| \frac{R_{q,n}}{\delta_{q+1,n}}\|_0 \|\bar D_t \nabla \Psi_k\|_N + \| \frac{R_{q,n}}{\delta_{q+1,n}}\|_0 \|\bar D_t \nabla \Psi_k\|_0 \|\nabla \Psi_k\|_N, \\
         \|\bar D_t \mathcal{E}_{\zeta, k, k', n} \|_N & \lesssim \delta_{q+1,n}^{-1} \Big(\|\bar D_t c^2_{\zeta, k, n} \|_N  + \|\bar D_t c^2_{\zeta, k, n} \|_0 \|\nabla \Psi_k\|_N + \| c^2_{\zeta, k, n}\|_N \|\bar D_t \nabla \Psi_k\|_0\\ 
        & \qquad + \|c^2_{\zeta, k, n}\|_0 \|\bar D_t \nabla \Psi_k\|_N + \| c^2_{\zeta, k, n}\|_0 \|\bar D_t \nabla \Psi_k\|_0 \|\nabla \Psi_k\|_N \Big), \\
        \|\bar D_t \big( \nabla \Psi_k \nabla \Psi_k^T \big)\|_N &\lesssim \|\bar D_t \nabla \Psi_k\|_N + \|\bar D_t \nabla \Psi_k\|_0 \|\nabla \Psi_k\|_N. 
    \end{align*}
Therefore, Lemmas \ref{smoli_estim} and \ref{Flow_estim}, together with the assumptions of Proposition \ref{NewIter}, confirm the estimates \eqref{a_estim_2} and \eqref{a_estim_4}.
\end{proof}

\begin{cor} \label{a_cor}
The following estimates, thanks to the assumptions of Proposition \ref{NewIter}, hold:
\begin{align*}
&\|A_{\sigma, k, n}\|_N, \, \|B_{\zeta, k, n}\|_N \lesssim \delta_{q+1, n} \lambda_q^N, \hspace{4.2cm}\,  \forall N \in \{0,1,..., L-1\}, \\
&\|\bar D_t A_{\sigma, k, n} \|_N, \,  \|\bar D_t B_{\zeta, k, n} \|_N  \lesssim \delta_{q+1, n} \tau_q^{-1} \lambda_q^N, \hspace{2.9cm}  \forall N \in \{0,1,..., L-1\},\\
&\|A_{\sigma, k, n}\|_{N+L-1}, \,  \|B_{\zeta, k, n}\|_{N+L-1} \lesssim \delta_{q+1, n} \lambda_q^{L-1} \ell_q^{-N},  \hspace{1.7cm}\,\, \,\forall N \geq 0,\\
&\|\bar D_t A_{\sigma, k, n} \|_{N+ L - 1}, \, \|\bar D_t B_{\zeta, k, n} \|_{N+ L - 1} \lesssim \delta_{q+1, n} \lambda_q^{L-1} \tau_q^{-1} \ell_q^{-N}, \quad \,\, \forall N \geq 0.
\end{align*}
Here all the implicit bounds are dependent on $n$, $\Gamma$, $M$, $\alpha$, and $N$.
\end{cor}

\begin{proof}
Observe that the same reasoning applied in the proof of Lemma \ref{a_estim} can be used to deduce that $a_{\xi, k, n}^2$
satisfies analogous estimates to those of $a_{\xi, k, n}$, with $\delta_{q, n+1}^{1/2}$ replaced by $\delta_{q, n+1}$ in all instances. Then we can estimate
    \begin{equation*}
        \|A_{\sigma, k, n}\|_{N} \lesssim \|a^2_{\xi, k, n}\|_N  \|(\nabla \Psi_k)^{-1}\|_0^2 + \|a^2_{\xi, k, n}\|_0  \|(\nabla \Psi_k)^{-1}\|_N \|(\nabla \Psi_k)^{-1}\|_0,
    \end{equation*}
    and 
    \begin{align*}
        \|\bar D_t A_{\sigma, k, n}\|_{N} & \lesssim \|\bar D_t a^2_{\xi, k, n}\|_N \|(\nabla \Psi_k)^{-1}\|_0^2 + \|\bar D_t a^2_{\xi, k, n}\|_0 \|(\nabla \Psi_k)^{-1}\|_N  \|(\nabla \Psi_k)^{-1}\|_0  \\ 
        & \quad + \|a^2_{\xi, k, n}\|_N \|\bar D_t (\nabla \Psi_k)^{-1}\|_0 \|(\nabla \Psi_k)^{-1}\|_0 + \|a^2_{\xi, k, n}\|_0 \|\bar D_t (\nabla \Psi_k)^{-1}\|_N \|(\nabla \Psi_k)^{-1}\|_0 \\ 
        &\qquad + \|a^2_{\xi, k, n}\|_0 \|\bar D_t (\nabla \Psi_k)^{-1}\|_0 \|(\nabla \Psi_k)^{-1}\|_N.
    \end{align*}
    Similarly, we have
    \begin{equation*}
        \|B_{\zeta, k, n}\|_{N} \lesssim \|b^2_{\zeta, k, n}\|_N  \|(\nabla \Psi_k)^{-1}\|_0 + \|b^2_{\zeta, k, n}\|_0  \|(\nabla \Psi_k)^{-1}\|_N,
    \end{equation*}
    and 
    \begin{align*}
        \|\bar D_t B_{\zeta, k, n}\|_{N} &\lesssim \|\bar D_t b^2_{\zeta, k, n}\|_N \|(\nabla \Psi_k)^{-1}\|_0 + \|\bar D_t b^2_{\zeta, k, n}\|_0 \|(\nabla \Psi_k)^{-1}\|_N  \\ 
        &\qquad \qquad + \|b^2_{\zeta, k, n}\|_N \|\bar D_t (\nabla \Psi_k)^{-1}\|_0 + \|b^2_{\zeta, k, n}\|_0 \|\bar D_t (\nabla \Psi_k)^{-1}\|_N.
    \end{align*}
In view of Lemma~\ref{Flow_estim}, we conclude the proof of the corollary.
\end{proof}

To obtain the required estimates in Proposition~\ref{NewIter}, we also need a technical lemma concerning the Newton perturbations. To state it precisely, let $\psi_{k, n+1}$ and $\varphi_{k, n+1}$ represent the stream functions associated with $w_{k, n+1}$ and $\theta_{k, n+1}$ respectively. More precisely, we define $w_{k, n+1} = \nabla^{\perp} \psi_{k, n+1}$, and 
$\varphi_{k, n+1} = \mathcal{R} \theta_{k, n+1}$
Then, we have 
\begin{equation*}
    R_{q, n+1} = \mathcal{R} \nabla^{\perp} \sum_{k \in \mathbb Z_{q,n}} \partial_t \tilde \eta_k \psi_{k, n+1}, \,\, \text{and} \,\, T_{q, n+1} =  \sum_{k \in \mathbb Z_{q,n}} \partial_t \tilde \eta_k \varphi_{k, n+1}.
\end{equation*}
Keeping in mind that $\mathcal{R} \nabla^{\perp}$ is of Calder\'on-Zygmund type, we need to obtain estimates for the stream functions in order to prove the Proposition~\ref{NewIter}. To that context, notice that $\psi_{k, n+1}$ and $\varphi_{k, n+1}$ satisfy 
\begin{equation} \label{psi_eqn}
    \begin{cases}
        \partial_t \psi_{k, n + 1} + \bar v_q \cdot \nabla \psi_{k, n + 1} - 2 \Delta^{-1} \nabla^\perp \cdot \div (\psi_{k, n+1} \nabla^\perp \bar v_q) 
    - \Delta^{-1} \nabla^\perp \cdot \div(\varphi_{k, n + 1}) e_2 \\
    = \displaystyle \sum_{\xi \in \Lambda_R} f_{\xi, k, n}(\mu_{q+1} t) \Delta^{-1} \nabla^\perp \cdot \div A^1_{\xi, k, n} + \sum_{\zeta \in \Lambda_T} m_{\zeta, k, n}(\mu_{q+1} t) \Delta^{-1} \nabla^\perp \cdot \div A^2_{\zeta, k, n} \\ 
        \psi_{k, n+1} \big|_{t = t_k} \\
        = \frac{1}{\mu_{q+1}} \left[ \displaystyle \sum_{\xi \in \Lambda_R} f_{\xi, k, n}^{[1]}(\mu_{q+1} t_k) \Delta^{-1} \nabla^\perp \cdot \div A^1_{\xi, k, n} +
        \sum_{\zeta \in \Lambda_T} m_{\zeta, k, n}^{[1]}(\mu_{q+1} t_k) \Delta^{-1} \nabla^\perp \cdot \div A^2_{\zeta, k, n} \right]\Big|_{t = t_k},
    \end{cases}
\end{equation}
and
\begin{equation} \label{psi_eqn1}
    \begin{cases}
        \partial_t \varphi_{k, n + 1} + \bar v_q \cdot \nabla \varphi_{k, n + 1} - \varphi_{k, n+1} \cdot \nabla \bar v_q 
        -\psi_{k, n+1} \nabla^\perp \bar \theta_q + \nabla^\perp q_{k,n}= \displaystyle \sum_{\zeta \in \Lambda_T} m_{\zeta, k, n}(\mu_{q+1} t) B_{\zeta, k, n} \\ 
        \nabla^\perp \cdot \varphi_{k, n + 1} =0, \\
        \varphi_{k, n+1} \big|_{t = t_k} = \frac{1}{\mu_{q+1}} \displaystyle \sum_{\zeta \in \Lambda_T} m_{\zeta, k, n}^{[1]}(\mu_{q+1} t_k) (\text{Id} - \nabla^{\perp} (\nabla^{\perp})^{-1}) B_{\zeta, k, n}\big|_{t = t_k},
    \end{cases}
\end{equation}
where the scalar function $q_{k,n}$ and the curl-free vector filed $\varphi_{k, n+1}$ are unknowns in \eqref{psi_eqn1}. Notice that $(\nabla^{\perp})^{-1} = \Delta^{-1} \nabla^\perp \cdot$, and $\nabla^\perp q_{k,n}$ can be expressed as a Fourier multiplier:
\begin{align}
&\nabla^\perp q_{k,n} = \nabla^\perp  \Delta^{-1} \nabla^\perp \cdot (\sum_{\zeta \in \Lambda_T} m_{\zeta, k, n}(\mu_{q+1} t) B_{\zeta, k, n} 
- \bar v_q \cdot \nabla \varphi_{k, n + 1} + \varphi_{k, n+1} \cdot \nabla \bar v_q 
        +\psi_{k, n+1} \nabla^\perp \bar \theta_q ) \nonumber \\
        &\,\,= \nabla^\perp  \Delta^{-1} \nabla^\perp \cdot (\sum_{\zeta \in \Lambda_T} m_{\zeta, k, n}(\mu_{q+1} t) B_{\zeta, k, n} + \varphi_{k, n+1} \cdot \nabla \bar v_q +\psi_{k, n+1} \nabla^\perp \bar \theta_q ) -  \Delta^{-1} \nabla^\perp \div(\varphi_{k, n + 1}  (\nabla^\perp \bar v_q)^T) \nonumber \\
        &\,\,:=  \nabla^{\perp} (\nabla^{\perp})^{-1} \sum_{\zeta \in \Lambda_T} m_{\zeta, k, n}(\mu_{q+1} t) B_{\zeta, k, n}
        + q_{1,k,n}. \label{NEW}
\end{align}
We are now in a position to state the technical lemma. In what follows, we first define a vector function $\mathcal{S}_{k, n+1} := (\psi_{k, n+1}, \varphi_{k, n+1})$, and define $\| \mathcal{S}_{k, n+1}\|_m:= \| \psi_{k, n+1}\|_m + \| \varphi_{k, n+1}\|_m$ for any $m\ge 0$.

\begin{lem} \label{psi_estim}
The following estimates, thanks to the assumptions of Proposition \ref{NewIter}, hold on the $\supp \tilde \eta_k$:
\begin{align} 
&\|\mathcal{S}_{k, n+1}\|_{N+\alpha} \lesssim \frac{\delta_{q+1, n} \lambda_q^{N} \ell_q^{-\alpha}}{\mu_{q+1}}, \hspace{2.1cm} \forall N \in \{0,1,..., L-1\}, \label{psi_estim_1}\\ 
&\|\bar D_t \mathcal{S}_{k, n+1}\|_{N+\alpha} \lesssim  \delta_{q+1, n} \lambda_q^N \ell_q^{-\alpha}, \hspace{1.8cm} \forall N \in \{0, 1,..., L-1\}, \label{psi_estim_2} \\
&\|\mathcal{S}_{k, n+1}\|_{N + L - 1 + \alpha} \lesssim \frac{\delta_{q+1, n} \lambda_q^{L-1} \ell_q^{-N-\alpha}}{\mu_{q+1}},  \quad \,\,\, \forall N \geq 0,  \label{psi_estim_3} \\
&\|\bar D_t \mathcal{S}_{k, n+1}\|_{N+ L - 1 + \alpha} \lesssim \delta_{q+1, n}  \lambda_q^{L-1} \ell_q^{-N-\alpha}, \,\,\,\, \forall N \geq 0. \label{psi_estim_4}
\end{align}
Moreover, on the $\supp \partial_t \tilde \eta_k$, we have 
\begin{align} 
&\|\bar D_t \mathcal{S}_{k, n+1}\|_{N+\alpha} \lesssim \frac{\delta_{q+1, n} \lambda_q^N \ell_q^{-\alpha}}{\mu_{q+1}\tau_q}, \hspace{2cm} \forall N \in \{0, 1,..., L-1\}, \label{psi_estim_5} \\
&\|\bar D_t \mathcal{S}_{k, n+1}\|_{N+L -1 + \alpha} \lesssim \frac{\delta_{q+1, n}  \lambda_q^{L-1} \ell_q^{-N-\alpha}}{\mu_{q+1} \tau_q}, \quad \,\,\forall N \geq 0.  \label{psi_estim_6}
\end{align}
Here all the implicit constants depend on $n$, $\Gamma$, $M$, $\alpha$, and $N$.
\end{lem}

\begin{proof}
We first decompose $\psi_{k, n+1}$ into three parts as follows:
    \begin{equation} \label{psidecomp}
        \psi_{k, n+1}  = \tilde \psi_1 + \tilde \Xi_1 + \Xi_1,
    \end{equation}
    where $\tilde \psi_1$ solves 
    \begin{equation*}
        \begin{cases}
            \bar D_t \tilde \psi_1 =  2 \Delta^{-1} \nabla^\perp \cdot \div (\psi_{k, n+1} \nabla^\perp \bar v_q) + \Delta^{-1} \nabla^\perp \cdot \div (\varphi_{k, n + 1}) e_2\\
            \tilde \psi_1 \big|_{t = t_k} = 0,
        \end{cases}
    \end{equation*}
    and $\tilde \Xi_1$ solves the equation 
    \begin{equation*}
    \begin{cases}
        \bar D_t \tilde \Xi_1 = - \frac{1}{\mu_{q+1}} \left[ \displaystyle \sum_{\xi \in \Lambda_R} f^{[1]}_{\xi, k, n}(\mu_{q+1} \cdot ) \bar D_t \Delta^{-1} \nabla^\perp \cdot \div A^1_{\xi, k, n}
        + \sum_{\zeta \in \Lambda_T} m^{[1]}_{\zeta, k, n}(\mu_{q+1} \cdot ) \bar D_t \Delta^{-1} \nabla^\perp \cdot \div A^2_{\zeta, k, n} \right] \\ 
        \tilde \Xi_1 \big |_{t = t_k} = 0,
        \end{cases}
    \end{equation*}
    and 
    \begin{equation*}
        \Xi_1 = \frac{1}{\mu_{q+1}} \left[ \sum_{\xi \in \Lambda_R} f_{\xi, k, n}^{[1]}(\mu_{q+1} \cdot) \Delta^{-1} \nabla^\perp \cdot \div A^1_{\xi, k, n} + \sum_{\zeta \in \Lambda_T} m_{\zeta, k, n}^{[1]}(\mu_{q+1} \cdot) \Delta^{-1} \nabla^\perp \cdot \div A^2_{\zeta, k, n} \right]. 
    \end{equation*}
It is easy to verify the above decomposition \eqref{psidecomp}, thanks to the uniqueness of solutions for transport equations. 
Similarly, we can decompose $\varphi_{k, n+1}$ into the following three parts:
    \begin{equation} \label{psidecomp1}
        \varphi_{k, n+1}  = \tilde \psi_2 + \tilde \Xi_2 + \Xi_2,
    \end{equation}
where $\tilde \psi_2$ solves 
    \begin{equation*}
        \begin{cases}
            \bar D_t \tilde \psi_2 =  \varphi_{k, n+1} \cdot \nabla^\perp \bar v_q 
            + \psi_{k, n+1} \nabla^\perp \bar \theta_q + q_{1,k,n}\\
            \tilde \psi_2 \big|_{t = t_k} = 0,
        \end{cases}
    \end{equation*}
$\tilde \Xi_2$ is the solution to 
    \begin{equation*}
    \begin{cases}
        \bar D_t \tilde \Xi_2 = - \frac{1}{\mu_{q+1}} \displaystyle \sum_{\zeta \in \Lambda_T} m^{[1]}_{\zeta, k, n}(\mu_{q+1} \cdot ) \bar D_t  (\text{Id} - \nabla^{\perp} (\nabla^{\perp})^{-1}) B_{\zeta, k, n} \\ 
        \tilde \Xi_2 \big |_{t = t_k} = 0,
        \end{cases}
    \end{equation*}
and 
    \begin{equation*}
        \Xi_2 = \frac{1}{\mu_{q+1}} \sum_{\zeta \in \Lambda_T} m_{\zeta, k, n}^{[1]}(\mu_{q+1} \cdot) (\text{Id} - \nabla^{\perp} (\nabla^{\perp})^{-1}) B_{\zeta, k, n}. 
    \end{equation*}

\noindent \textbf{Estimates for $\tilde \psi_1$ when $N=0$:} 
Making use of the fact that $\Delta^{-1} \nabla^\perp \div$ is a sum of operators of Calder\'on-Zygmund type, we have 
    \begin{equation*}
        \|\bar D_t \tilde \psi_1\|_\alpha \lesssim \|\psi \nabla^\perp \bar v_q\|_\alpha + \|\varphi \|_\alpha \lesssim \|\psi\|_\alpha \|\bar v_q\|_{1 + \alpha} + \|\varphi\|_\alpha,
    \end{equation*}
then, it follows that 
        \begin{equation*}
        \|\tilde \psi_1(\cdot, t)\|_\alpha \lesssim \tau_q^{-1} \int_{t_k}^t \left( \|\psi_{k, n+1}(\cdot, s) \|_\alpha +  \|\varphi_{k, n+1}(\cdot, s) \|_\alpha \right) \,ds.
    \end{equation*}
    
\noindent \textbf{Estimates for $\tilde \psi_2$ when $N=0$:} 
As before, notice that
    \begin{equation*}
        \|\bar D_t \tilde \psi_2\|_\alpha \lesssim \|\psi \nabla^\perp \bar \theta_q + \varphi \nabla^\perp \bar v_q + q_1\|_\alpha  \lesssim \|\psi\|_\alpha \|\bar \theta_q\|_{1 + \alpha} +\|\varphi\|_\alpha \|\bar v_q\|_{1 + \alpha},
    \end{equation*}
and, hence, we have
        \begin{equation*}
        \|\tilde \psi_2(\cdot, t)\|_\alpha \lesssim \tau_q^{-1} \int_{t_k}^t \left( \|\psi_{k, n+1}(\cdot, s) \|_\alpha +  \|\varphi_{k, n+1}(\cdot, s) \|_\alpha \right) \,ds.
    \end{equation*}

\noindent \textbf{Estimates for $\tilde \Xi_1$ when $N = 0$:} 
We can use similar arguments to conclude
    \begin{eqnarray*}
        \|\bar D_t \tilde \Xi_1\|_{\alpha} &\lesssim & \frac{1}{\mu_{q+1}}\sup_\sigma \|\bar D_t A_{\sigma, k, n}\|_{\alpha} + \frac{1}{\mu_{q+1}}\sup_\sigma \|[\bar v_q \cdot \nabla, \Delta^{-1}\nabla^\perp \div] A_{\sigma, k, n}\|_{\alpha} \\ 
        & \lesssim & \frac{1}{\mu_{q+1}}\sup_\sigma \|\bar D_t A_{\sigma, k, n}\|_{\alpha} + \frac{1}{\mu_{q+1}} \|\bar v_q\|_{1+\alpha} \sup_\sigma \|A_{\sigma, k, n}\|_\alpha  
        \lesssim \frac{\delta_{q+1, n}  \lambda_q^\alpha}{\mu_{q+1} \tau_q}
    \end{eqnarray*}
where we have used the commutator estimate of Proposition~\ref{CZ_comm}. Therefore, in view of proposition \ref{transport_estim}, we conclude that, on $\supp \tilde \eta_k$, we have
    \begin{equation*}
        \|\tilde \Xi_1\|_\alpha \lesssim \frac{1}{\mu_{q+1}} \delta_{q+1, n} \lambda_q^{\alpha}.
    \end{equation*}
    
\noindent \textbf{Estimates for $\tilde \Xi_2$ when $N = 0$:} We can simply repeat the previous argument, as in the estimation of the term $\tilde \Xi_1$, to conclude that
 \begin{equation*}
        \|\tilde \Xi_2\|_\alpha \lesssim \frac{1}{\mu_{q+1}} \delta_{q+1, n} \lambda_q^{\alpha}.
    \end{equation*}
    
\noindent \textbf{Estimates for $\Xi_1$ and $\Xi_2$ when $N=0$:} 
To estimate the terms $\Xi_1$ and $\Xi_2$, we simply notice that
    \begin{equation*}
        \|\Xi_1\|_{\alpha} \lesssim \frac{1}{\mu_{q+1}} \sup_\sigma \|A_{\sigma, k, n}\|_{\alpha} \lesssim \frac{\delta_{q+1, n} \lambda_q^{\alpha}}{\mu_{q+1}}, \quad  \|\Xi_2\|_{\alpha} \lesssim \frac{1}{\mu_{q+1}} \sup_\zeta \|B_{\zeta, k, n}\|_{\alpha} \lesssim \frac{\delta_{q+1, n} \lambda_q^{\alpha}}{\mu_{q+1}}.
    \end{equation*}
We can now plug in the above estimates in \eqref{psidecomp} and \eqref{psidecomp1}, to obtain 
    \begin{align*}
        \|\psi_{k, n+1}(\cdot, t)\|_\alpha &\lesssim \frac{\delta_{q+1, n} \lambda_q^{\alpha}}{\mu_{q+1}} + \tau_q^{-1} \int_{t_k}^t \left( \|\psi_{k, n+1}(\cdot, s) \|_\alpha +  \|\varphi_{k, n+1}(\cdot, s) \|_\alpha \right) \,ds,  \\
        \|\varphi_{k, n+1}(\cdot, t)\|_\alpha &\lesssim \frac{\delta_{q+1, n} \lambda_q^{\alpha}}{\mu_{q+1}} + \tau_q^{-1} \int_{t_k}^t \big( \|\psi_{k, n+1}(\cdot, s)\|_\alpha + \|\varphi_{k, n+1}(\cdot, s)\|_\alpha \big) ds,  
    \end{align*}
In view of the above two estimates, Gr\"onwall's inequality ensures that, on $\supp \tilde \eta_k$, 
    \begin{equation*}
        \|\mathcal{S}_{k, n+1}\|_\alpha \lesssim \frac{\delta_{q+1, n} \lambda_q^{\alpha}}{\mu_{q+1}}.
    \end{equation*}

\noindent \textbf{Estimates for $\tilde \psi_1$ when $N \geq 1$:} 
To obtain the higher order estimates, we fix a multi-index $\theta$, with $|\theta| = N$. Then, we have
    \begin{equation*}
        \|\bar D_t \partial^\theta \tilde \psi_1\|_\alpha \lesssim \|\partial^\theta \bar D_t  \tilde \psi_1\|_\alpha + \|[\bar v_q \cdot \nabla, \partial^\theta] \tilde \psi_1\|_\alpha. 
    \end{equation*}
For the first term, we have
    \begin{align*}
        \|\partial^\theta \bar D_t  \tilde \psi_1\|_\alpha &\lesssim \|\partial^\theta(\psi_{k, n+1} \nabla^\perp \bar v_q)\|_\alpha + \| \partial^\theta \varphi_{k, n+1} \|_\alpha \\
        & \lesssim\|\bar v_q\|_{1+\alpha} \|\psi_{k, n+1}\|_{N+\alpha} + \|\bar v_q\|_{N+ 1+\alpha}\|\psi_{k, n+1}\|_{\alpha}
        + \|\varphi_{k, n+1}\|_{N+\alpha}\\
        & \lesssim\|\bar v_q\|_{1+\alpha} \left(\|\psi_{k, n+1}\|_{N+\alpha} + \|\varphi_{k, n+1}\|_{N+\alpha} \right)+ \|\bar v_q\|_{N+ 1+\alpha}\|\psi_{k, n+1}\|_{\alpha},
    \end{align*}
while for the other term, we have
    \begin{eqnarray*}
        \|[\bar v_q \cdot \nabla, \partial^\theta] \tilde \psi_1\|_\alpha &\lesssim& \|\bar v_q\|_{N+\alpha}\|\tilde \psi_1\|_{1+\alpha} + \|\bar v_q\|_{1+\alpha}\|\tilde \psi_1\|_{N+\alpha} \\ 
        &\lesssim & \|\bar v_q\|_{1+\alpha}\|\tilde \psi_1\|_{N+\alpha} + \|\bar v_q\|_{N+ 1+\alpha}\|\tilde \psi_1\|_{\alpha},
    \end{eqnarray*}
Therefore, thanks to proposition \ref{transport_estim}, we have    
\begin{align*}
\|\tilde \psi_1 (\cdot, t)\|_{N+ \alpha} & \lesssim \frac{\delta_{q+1, n} \lambda_q^{\alpha}\tau_q }{\mu_{q+1}} \|\bar v_q\|_{N+1+\alpha} + \|\bar v_q\|_{1+\alpha} \int_{t_k}^t \left(\|\psi_{k, n+1}\|_{N+\alpha} + \|\varphi_{k, n+1}\|_{N+\alpha} \right) \,ds \\
   & \qquad \qquad + \|\bar v_q\|_{1+\alpha} \int_{t_k}^t \|\tilde \psi_1(\cdot, s) \|_{N+\alpha}ds.
    \end{align*}
A simple application of Gr\"onwall's inequality yields
    \begin{equation*}
        \|\tilde \psi_1 (\cdot, t)\|_{N+ \alpha} \lesssim \frac{\delta_{q+1, n} \lambda_q^{\alpha} \tau_q}{\mu_{q+1}} \|\bar v_q\|_{N+1+\alpha}  + \tau_q^{-1} \int_{t_k}^t \left(\|\psi_{k, n+1}\|_{N+\alpha} + \|\varphi_{k, n+1}\|_{N+\alpha} \right) \,ds.
    \end{equation*}

\noindent \textbf{Estimates for $\tilde \psi_2$ when $N \geq 1$:} 
As before, notice that
    \begin{equation*}
        \|\bar D_t \partial^\theta \tilde \psi_2\|_\alpha \lesssim \|\partial^\theta \bar D_t  \tilde \psi_2\|_\alpha + \|[\bar v_q \cdot \nabla, \partial^\theta] \tilde \psi_2\|_\alpha.
    \end{equation*}
For the first term, we have
    \begin{align*}
       & \|\partial^\theta \bar D_t  \tilde \psi_2\|_\alpha \lesssim \|\partial^\theta(\psi_{k, n+1} \nabla^\perp \bar \theta_q)\|_\alpha + \|\partial^\theta(\varphi_{k, n+1} \nabla^\perp \bar v_q) + \partial^\theta q_1 \|_\alpha  \\
        & \quad \lesssim \|\bar \theta_q\|_{1+\alpha} \|\psi_{k, n+1}\|_{N+\alpha} + \|\bar \theta_q\|_{N+ 1+\alpha}\|\psi_{k, n+1}\|_{\alpha}
        + \|\bar v_q\|_{1+\alpha} \|\varphi_{k, n+1}\|_{N+\alpha} + \|\bar v_q\|_{N+ 1+\alpha}\|\varphi_{k, n+1}\|_{\alpha},
    \end{align*}
while for the last term, we have
    \begin{eqnarray*}
        \|[\bar v_q \cdot \nabla, \partial^\theta] \tilde \psi_2\|_\alpha \lesssim \|\bar v_q\|_{N+\alpha}\|\tilde \psi_2\|_{1+\alpha} + \|\bar v_q\|_{1+\alpha}\|\tilde \psi_2\|_{N+\alpha} 
\|\bar v_q\|_{1+\alpha}\|\tilde \psi_2\|_{N+\alpha} + \|\bar v_q\|_{N+ 1+\alpha}\|\tilde \psi_2\|_{\alpha}.
    \end{eqnarray*}
Therefore, as before, we conclude that
    \begin{align*}
        \|\tilde \psi_2 (\cdot, t)\|_{N+ \alpha} &\lesssim  \frac{\delta_{q+1, n} \lambda_q^{\alpha}\tau_q }{\mu_{q+1}} \|\bar v_q\|_{N+1+\alpha} + \frac{\delta_{q+1, n} \lambda_q^{\alpha}\tau_q }{\mu_{q+1}} \|\bar \theta_q\|_{N+1+\alpha} + \|\bar v_q\|_{1+\alpha} \int_{t_k}^t  \|\varphi_{k, n+1}\|_{N+\alpha}  \,ds \\ 
& + \|\bar v_q\|_{1+\alpha} \int_{t_k}^t \|\tilde \psi_2(\cdot, s) \|_{N+\alpha}ds + \|\bar \theta_q\|_{1+\alpha} \int_{t_k}^t \|\psi_{k, n+1}\|_{N+\alpha}\,ds.
    \end{align*}
Again, a simple application of Gr\"onwall's inequality implies
    \begin{equation*}
        \|\tilde \psi_2 (\cdot, t)\|_{N+ \alpha} \lesssim \frac{\delta_{q+1, n} \lambda_q^{\alpha} \tau_q}{\mu_{q+1}} \left( \|\bar v_q\|_{N+1+\alpha}  +  \|\bar \theta_q\|_{N+1+\alpha} \right)+ \tau_q^{-1} \int_{t_k}^t \left(\|\psi_{k, n+1}\|_{N+\alpha} + \|\varphi_{k, n+1}\|_{N+\alpha} \right) \,ds.
    \end{equation*}

\noindent \textbf{Estimates for $\tilde \Xi_1$ when $N \geq 1$:} 
We compute
    \begin{eqnarray*}
        \|\bar D_t \partial^\theta \tilde \Xi_1\|_\alpha &\lesssim& \|\bar D_t \tilde \Xi_1\|_{N+\alpha} + \|[\bar v_q \cdot \nabla, \partial^\theta] \tilde \Xi_1\|_\alpha  \\ 
        &\lesssim & \frac{1}{\mu_{q+1}}\sup_\sigma \big( \|\bar D_t A_{\sigma, k, n}\|_{N+\alpha} + \|[\bar v_q \cdot \nabla, \Delta^{-1} \nabla^\perp \div] A_{\sigma, k, n}\|_{N+\alpha} \big) + \|[\bar v_q \cdot \nabla, \partial^\theta] \tilde \Xi_1\|_\alpha \\ 
         &\lesssim & \frac{1}{\mu_{q+1}}\sup_\sigma \big( \|\bar D_t A_{\sigma, k, n}\|_{N+\alpha} + \|\bar v_q\|_{1+\alpha} \|A_{\sigma, k, n}\|_{N+\alpha} + \|\bar v_q\|_{N+1 + \alpha} \|A_{\sigma, k, n}\|_\alpha \big) \\ 
         && + \|\bar v_q\|_{N+1+\alpha} \|\tilde \Xi_1\|_\alpha + \|\bar v_q\|_{1+\alpha} \|\tilde \Xi_1\|_{N+ \alpha}, 
    \end{eqnarray*}
where we have used the Proposition~\ref{CZ_comm}, and treated the commutator term $[\bar v_q \cdot \nabla, \partial^\theta]$ as before. Then, thanks to Proposition~\ref{transport_estim} and Gr\"onwall's inequality, we conclude
    \begin{equation*}
        \|\tilde \Xi_1\|_{N+\alpha} \lesssim \frac{1}{\mu_{q+1}} \sup_\sigma \big( \tau_q \|\bar D_t A_{\sigma, k, n} \|_{N+\alpha} + \|A_{\sigma, k, n}\|_{N+\alpha}\big) + \frac{\delta_{q+1, n} \lambda_q^\alpha \tau_q}{\mu_{q+1}} \|\bar v_q\|_{N+1+\alpha}.
    \end{equation*}

\noindent \textbf{Estimates for $\tilde \Xi_2$ when $N \geq 1$:} A similar calculation reveals that
 \begin{equation*}
        \|\tilde \Xi_2\|_{N+\alpha} \lesssim \frac{1}{\mu_{q+1}} \sup_\zeta \big( \tau_q \|\bar D_t B_{\zeta, k, n} \|_{N+\alpha} + \|B_{\zeta, k, n}\|_{N+\alpha}\big) + \frac{\delta_{q+1, n} \lambda_q^\alpha \tau_q}{\mu_{q+1}} \|\bar v_q\|_{N+1+\alpha}.
    \end{equation*}

\noindent \textbf{Estimates for $\Xi_1$ and $\Xi_2$ when $N \geq 1$:} 
Finally, we note that
    \begin{equation*}
        \|\Xi_1\|_{N+\alpha} \lesssim \frac{1}{\mu_{q+1}} \sup_\sigma \|A_{\sigma, k, n}\|_{N+\alpha}, \quad \|\Xi_2\|_{N+\alpha} \lesssim \frac{1}{\mu_{q+1}} \sup_\zeta \|B_{\zeta, k, n}\|_{N+\alpha}. 
    \end{equation*}
We can now plug in the above estimates in \eqref{psidecomp} and \eqref{psidecomp1}, to obtain 
    \begin{align*}
       & \|\psi_{k, n+1}(\cdot, t)\|_{N+\alpha} \lesssim  \frac{1}{\mu_{q+1}}\sup_\sigma \big( \tau_q \|\bar D_t A_{\sigma, k, n}\|_{N+\alpha} + \|A_{\sigma, k, n}\|_{N+\alpha} + \delta_{q+1, n} \lambda_q^\alpha \tau_q \|\bar v_q\|_{N+1+\alpha} \big) \\ 
        & \qquad \qquad + \tau_q^{-1} \int_{t_k}^t \left(\|\psi_{k, n+1}(\cdot, s)\|_{N+\alpha} + \|\varphi_{k, n+1}(\cdot, s)\|_{N+\alpha} \right) \,ds, \\
       & \|\varphi_{k, n+1}(\cdot, t)\|_{N+\alpha} \lesssim  \frac{1}{\mu_{q+1}}\sup_\zeta \big( \tau_q \|\bar D_t B_{\zeta, k, n}\|_{N+\alpha} + \|B_{\zeta, k, n}\|_{N+\alpha} \big) \\
        & \qquad \qquad + \frac{\delta_{q+1, n} \lambda_q^\alpha \tau_q}{\mu_{q+1}} \left( \|\bar v_q\|_{N+1+\alpha} + \|\bar \theta_q\|_{N+1+\alpha}\right) + \tau_q^{-1} \int_{t_k}^t \left(\|\psi_{k, n+1}(\cdot, s)\|_{N+\alpha} + \|\varphi_{k, n+1}(\cdot, s)\|_{N+\alpha} \right) \,ds.
    \end{align*}
Thus, a simple application of Gr\"onwall's inequality reveals that
    \begin{align*}
        \|\mathcal{S}_{k, n+1}\|_{N+\alpha} & \lesssim \frac{1}{\mu_{q+1}}\sup_\sigma \big( \tau_q \|\bar D_t A_{\sigma, k, n}\|_{N+\alpha} + \|A_{\sigma, k, n}\|_{N+\alpha} + \delta_{q+1, n} \lambda_q^\alpha \tau_q \|\bar v_q\|_{N+1+\alpha} \big)\\
        & \quad +  \frac{1}{\mu_{q+1}}\sup_\zeta \big( \tau_q \|\bar D_t B_{\zeta, k, n}\|_{N+\alpha} + \|B_{\zeta, k, n}\|_{N+\alpha} + \delta_{q+1, n} \lambda_q^\alpha \tau_q \|\bar \theta_q\|_{N+1+\alpha} \big).
    \end{align*}
Hence, thanks to Lemma~\ref{smoli_1} and Corollary~\ref{a_cor}, the estimates \eqref{psi_estim_1} and \eqref{psi_estim_3} follow. 

Finally, for the material derivative estimates of $\mathcal{S}_{k, n+1}$, we recall \eqref{psi_eqn} and \eqref{psi_eqn} to conclude 
\begin{align*}
        \|\bar D_t \psi_{k, n+1}(\cdot, t)\|_{N+\alpha}  &\lesssim  \|\psi_{k, n+1} \nabla^\perp \bar v_q\|_{N+\alpha} + \sup_\sigma \|A_{\sigma, k, n}(\cdot, t)\|_{N+\alpha} \\ 
         &\lesssim  \|\psi_{k, n+1}\|_{N+\alpha} \| \bar v_q\|_{1+\alpha} + \|\psi_{k, n+1}\|_{\alpha} \|\bar v_q\|_{N+1+\alpha} + \sup_\sigma \|A_{\sigma, k, n}(\cdot, t)\|_{N+\alpha}. \\
            \|\bar D_t \varphi_{k, n+1}(\cdot, t)\|_{N+\alpha}  &\lesssim  \|\psi_{k, n+1} \nabla^\perp \bar \theta_q\|_{N+\alpha} + \|\varphi_{k, n+1} \nabla^\perp \bar v_q\|_{N+\alpha} + \sup_\zeta \|B_{\zeta, k, n}(\cdot, t)\|_{N+\alpha} \\ 
         & \lesssim  \|\psi_{k, n+1}\|_{N+\alpha} \| \bar \theta_q\|_{1+\alpha} + \|\psi_{k, n+1}\|_{\alpha} \|\bar \theta_q\|_{N+1+\alpha} \\ 
         &\qquad + \|\varphi_{k, n+1}\|_{N+\alpha} \| \bar v_q\|_{1+\alpha} + \|\varphi_{k, n+1}\|_{\alpha} \|\bar v_q\|_{N+1+\alpha} + \sup_\zeta \|B_{\zeta, k, n}(\cdot, t)\|_{N+\alpha}.
    \end{align*}
Hence, thanks to Corollary~\ref{a_cor} and the estimates in \eqref{psi_estim_1} and \eqref{psi_estim_3}, we have the desired 
estimates \eqref{psi_estim_2} and \eqref{psi_estim_4}. Moreover, the estimates \eqref{psi_estim_5} and \eqref{psi_estim_6} follow likewise thanks to the observation that $A_{\sigma, k, n} = B_{\zeta, k, n}= 0$ on $\supp \partial_t \tilde \eta_k$. 
\end{proof}

We are now prepared to demonstrate the main Proposition \ref{NewIter}.

\begin{proof} [Proof of Proposition \ref{NewIter}]
Since $\mathcal R \nabla^\perp$ is of Calder\'on-Zygmund type, and the set $\{\tilde \eta_k\}$ is locally finite, we conclude that 
\begin{align*}
    \|R_{q, n+1}\|_N & \lesssim \|R_{q, n+1}\|_{N+\alpha} \lesssim \tau_q^{-1} \sup_{k \in \mathbb Z_{q,n}} \|\psi_{k, n+1}\|_{N+\alpha} \\
    \|T_{q, n+1}\|_N & \lesssim \|T_{q, n+1}\|_{N+\alpha} \lesssim \tau_q^{-1} \sup_{k \in \mathbb Z_{q,n}} \|\varphi_{k, n+1}\|_{N+\alpha}.
\end{align*}
Hence for $N \in \{0,1,..., L-1\}$, in view of the Lemma~\ref{psi_estim}, there exists a constant $C$ (independent of $a > a_0$ and $q$) such that 
\begin{equation*}
    \|R_{q, n+1}\|_N, \, \|T_{q, n+1}\|_N \leq C \delta_{q+1, n}\bigg(\frac{\lambda_q}{\lambda_{q+1}}\bigg)^{1/3 - \beta} (\lambda_{q+1} \ell_q)^{-\alpha} \lambda_{q+1}^{-2\alpha} \lambda_q^N \leq (C \lambda_{q+1}^{-\alpha })\delta_{q+1, n+1} \lambda_q^{-\alpha} \lambda_q^N,
\end{equation*}
Therefore, the estimates \eqref{NewIter_1} follow by choosing sufficiently large $a_0$. Similarly, the estimates \eqref{NewIter_3} follow from Lemma~\ref{psi_estim}. For the estimates of the material derivative, on $\supp \partial_t \tilde \eta_k$, we have
\begin{align*}
    &\|\bar D_t R_{q, n+1}\|_{N+\alpha} \lesssim \sup_{k \in \mathbb Z_{q,n}} \big( \|\bar D_t(\partial_t \tilde \eta_k \psi_{k, n+1})\|_{N+\alpha} + \|[\bar v_q \cdot \nabla, \mathcal{R} \nabla^\perp] \partial_t \tilde \eta_k \psi_{k, n+1}\|_{N+\alpha} \big) \\ 
    &\lesssim \sup_{k \in \mathbb Z_{q,n}} \big( \tau_q^{-2} \|\psi_{k, n+1}\|_{N+\alpha} + \tau_q^{-1} \|\bar D_t \psi_{k, n+1}\|_{N+\alpha}  
    + \tau_q^{-1} \|\bar v_q\|_{1+\alpha}\|\psi_{k, n+1}\|_{N+\alpha} + \tau_q^{-1} \|\bar v_q\|_{N+1+\alpha} \|\psi_{k,n+1}\|_\alpha \big), \\
  & \|\bar D_t T_{q, n+1}\|_{N+\alpha} \lesssim \sup_{k \in \mathbb Z_{q,n}} \big( \|\bar D_t(\partial_t \tilde \eta_k \varphi_{k, n+1})\|_{N+\alpha} + \|[\bar v_q \cdot \nabla, \mathcal{R} \div] \partial_t \tilde \eta_k \varphi_{k, n+1}\|_{N+\alpha} \big) \\ 
    &\lesssim  \sup_{k \in \mathbb Z_{q,n}} \big( \tau_q^{-2} \|\varphi_{k, n+1}\|_{N+\alpha} + \tau_q^{-1} \|\bar D_t \varphi_{k, n+1}\|_{N+\alpha} 
    + \tau_q^{-1} \|\bar v_q\|_{1+\alpha}\|\varphi_{k, n+1}\|_{N+\alpha} + \tau_q^{-1} \|\bar v_q\|_{N+1+\alpha} \|\varphi_{k,n+1}\|_\alpha \big),
\end{align*}
where we have used again the Proposition~\ref{CZ_comm}. Hence, for $N \in \{0,1,...,L-1\}$, Lemmas \ref{smoli_estim} and \ref{psi_estim} imply that
\begin{align*}
    \|\bar D_t R_{q, n+1}\|_{N} &\leq C \tau_q^{-1} \delta_{q+1, n} \bigg(\frac{\lambda_q}{\lambda_{q+1}}\bigg)^{1/3 - \beta} (\lambda_{q+1} \ell_q)^{-\alpha} \lambda_{q+1}^{-2\alpha} \lambda_q^N, \\
     \|\bar D_t T_{q, n+1}\|_{N} &\leq C \tau_q^{-1} \delta_{q+1, n} \bigg(\frac{\lambda_q}{\lambda_{q+1}}\bigg)^{1/3 - \beta} (\lambda_{q+1} \ell_q)^{-\alpha} \lambda_{q+1}^{-2\alpha} \lambda_q^N.
\end{align*}
Note that here the constant $C$ is independent of the parameters $a > a_0$ and $q$. Finally, by choosing $a_0$ sufficiently large, we conclude the proof of the proposition. 
\end{proof}

\subsection{Estimates for the Newton perturbations}

We now derive estimates for the Newton perturbations for velocity and temperature. First note that
\begin{equation*}
    w_{q+1}^{(t)} = \sum_{n = 1}^\Gamma w_{q+1, n}^{(t)} \,\, \text{and} \,\, \theta_{q+1}^{(t)} = \sum_{n = 1}^\Gamma \theta_{q+1, n}^{(t)}.
\end{equation*}
Since Proposition \ref{NewIter} holds, and as we have previously noted, the conditions required for its application are indeed satisfied by $R_{q, 0}$ and $T_{q, 0}$, we can conclude that the results derived in Lemma \ref{psi_estim} are also valid for all indices $n \in \{0, 1, ... ,\Gamma -1\}$. Therefore we have the following lemma:
\begin{lem} \label{w_t_estim}
We have the following estimates: 
\begin{align}
&\|w_{q+1}^{(t)}\|_N, \, \|\theta_{q+1}^{(t)}\|_N \lesssim \frac{\delta_{q+1} \lambda_q^{N + 1} \ell_q^{-\alpha}}{\mu_{q+1}}, \hspace{3.0cm} \forall N \in \{0,1,...,L-2\}, \label{w_t_estim_1}\\
& \|\bar D_t w_{q+1}^{(t)}\|_{N}, \, \|\bar D_t \theta_{q+1}^{(t)}\|_{N} \lesssim \delta_{q+1} \lambda_q^{N+1} \ell_q^{-\alpha}, \hspace{2.3cm} \forall N \in  \{0,1,...,L-2\}, \label{w_t_estim_2} \\
&\|w_{q+1}^{(t)}\|_{N+L -2 }, \, \|\theta_{q+1}^{(t)}\|_{N+L -2 } \lesssim \frac{\delta_{q+1} \lambda_q^{L-1} \ell_q^{-N-\alpha}}{\mu_{q+1}}, \qquad \quad \forall N \geq 0, \label{w_t_estim_3} \\
&\|\bar D_t w_{q+1}^{(t)}\|_{N+L-2}, \, \|\bar D_t \theta_{q+1}^{(t)}\|_{N+L-2} \lesssim \delta_{q+1} \lambda_q^{L-1} \ell_q^{-N - \alpha}, \quad \forall N \geq 0, \label{w_t_estim_4}
\end{align}
Here all the implicit bounds are dependent on $\Gamma$, $M$, $\alpha$ and $N$. Moreover, we have
    \begin{eqnarray}
        \supp_t w_{q+1}^{(t)}, \,\, \supp_t \theta_{q+1}^{(t)} &\subset& [-2 + (\delta_q^{1/2} \lambda_q)^{-1} - 3\Gamma \tau_q, -1 - (\delta_q^{1/2} \lambda_q)^{-1} + 3\Gamma \tau_q] \\ 
        && \qquad \cup [1 + (\delta_q^{1/2} \lambda_q)^{-1} - 3 \Gamma \tau_q, 2 - (\delta_q^{1/2} \lambda_q)^{-1} + 3 \Gamma \tau_q]. \nonumber
    \end{eqnarray}
\end{lem}

\begin{proof}
First note that, $\delta_{q+1, n} \leq \delta_{q+1}$ for all $n$ and it is enough to give estimates for $w_{q+1, n+1}^{(t)}$. To obtain estimates \eqref{w_t_estim_1} and \eqref{w_t_estim_3}, we simply use Lemma \ref{psi_estim}. Indeed, we have
    \begin{equation*}
        \|w_{q+1, n+1}^{(t)}\|_N \lesssim \sup_{k \in \mathbb Z_{q,n}} \|\psi_{k, n+1}\|_{N+1+\alpha}, \quad \|\theta_{q+1, n+1}^{(t)}\|_N \lesssim \sup_{k \in \mathbb Z_{q,n}} \|\varphi_{k, n+1}\|_{N+1+\alpha}.
    \end{equation*}
For the other estimates, we calculate 
    \begin{align*}
        \bar D_t w_{q+1, n+1}^{(t)} &= \sum_{k \in \mathbb Z_{q,n}} \big( \partial_t \tilde \eta_k \nabla^\perp \psi_{k, n+1} + \tilde \eta_k \nabla^\perp \bar D_t \psi_{k, n+1} - \tilde \eta_k \nabla^\perp \bar v_q \nabla \psi_{k, n+1} \big), \\
         \bar D_t \theta_{q+1, n+1}^{(t)} &= \sum_{k \in \mathbb Z_{q,n}} \big( \partial_t \tilde \eta_k \div \varphi_{k, n+1} + \tilde \eta_k \div \bar D_t \varphi_{k, n+1} - \tilde \eta_k \nabla \bar v_q : \nabla \varphi_{k, n+1} \big).
    \end{align*}
Therefore, it follows that 
    \begin{align*}
        \|\bar D_t w_{q+1, n+1}^{(t)}\|_N & \lesssim \sup_{k \in \mathbb Z_{q,n}} \big( \tau_q^{-1} \|\psi_{k, n+1}\|_{N+1+\alpha} + \|\bar D_t \psi_{k, n+1}\|_{N + 1 + \alpha} \\
        & \qquad \qquad + \|\bar v_q\|_{N+1} \|\psi_{k, n}\|_{1+\alpha} + \|\bar v_q\|_1 \|\psi_{k, n+1}\|_{N+1 + \alpha} \big). \\
        \|\bar D_t \theta_{q+1, n+1}^{(t)}\|_N & \lesssim \sup_{k \in \mathbb Z_{q,n}} \big( \tau_q^{-1} \|\varphi_{k, n+1}\|_{N+1+\alpha} + \|\bar D_t \varphi_{k, n+1}\|_{N + 1 + \alpha} \\
        & \qquad \qquad + \|\bar v_q\|_{N+1} \|\varphi_{k, n}\|_{1+\alpha} + \|\bar v_q\|_1 \|\varphi_{k, n+1}\|_{N+1 + \alpha} \big).
    \end{align*}
Again, thanks to Lemmas \ref{smoli_estim} and \ref{psi_estim}, the proof of \eqref{w_t_estim_2} and \eqref{w_t_estim_4} follows. Finally, thanks to the definition of $w_{q+1}^{(t)}$ and Proposition~\ref{NewIter}, we obtain the desired temporal supports.
\end{proof}

For later purpose, namely for the Nash iteration step, we need to make use of the backward flow $\tilde \Psi_t$ starting at $t \in \mathbb R$ associated to $\bar v_{q, \Gamma} = \bar v_{q} + w_{q+1}^{(t)}$. Indeed, $\tilde \Psi_t$ satisfies 
\begin{equation} \label{Flow_t_gam}
    \begin{cases}
        \partial_s \tilde \Psi_t(x,s) + \bar v_{q, \Gamma}(x,s) \cdot \nabla \tilde \Psi_t (x,s) = 0, \\ 
        \tilde \Psi_t(x, t) = x.
    \end{cases}
\end{equation}
We also write down the forward flow, given by
\begin{equation} \label{Lagr_t_gam}
    \begin{cases}
        \frac{d}{ds} \tilde X_t(\alpha, s) = \bar v_{q, \Gamma}(X_t(\alpha, s),s), \\ 
        X_t(\alpha, t) = \alpha.
    \end{cases}
\end{equation}
Notice that, for $N \in \{1, 2,..., L-2\}$, we have
\begin{equation*}
    \|w_{q+1}^{(t)}\|_N, \,\,  \|\theta_{q+1}^{(t)}\|_N \lesssim \delta_{q}^{1/2} \lambda_q^N \frac{\delta_{q+1} \lambda_q \ell_q^{-\alpha}}{\mu_{q+1} \delta_q^{1/2}} \lesssim \delta_q^{1/2} \lambda_q^N \bigg( \frac{\delta_{q+1}}{\delta_q} \bigg)^{1/2} \bigg( \frac{\lambda_q}{\lambda_{q+1}}\bigg)^{1/3}. 
\end{equation*}
This means that $\bar v_{q, \Gamma}$ and $\bar \theta_{q, \Gamma}$ are small perturbations of $\bar v_q$ and $\bar \theta_q$ respectively. This also implies the following corollary.

\begin{cor} \label{Gamma_velo_estim}
We have the following estimates:
\begin{align*}
        &\|\bar v_{q, \Gamma} \|_N, \,\|\bar \theta_{q, \Gamma} \|_N \lesssim \delta_{q}^{1/2} \lambda_q^N, \hspace{2.1cm}\,\, \qquad \forall N \in \{1,2,...,L-2\}, \\
        &\|\bar v_{q, \Gamma}\|_{N + L-2}, \,\|\bar \theta_{q, \Gamma}\|_{N + L-2} \lesssim \delta_{q}^{1/2} \lambda_q^{L-2} \ell_q^{-N}, \quad \forall N \geq 0.
    \end{align*}
Here all the implicit bounds are dependent on $\Gamma$, $M$, $\alpha$, and $N$.
\end{cor}
Next, we collect some standard estimates related to the above mentioned forward and backward flows. The proof is similar to that of Lemma \ref{Flow_estim}, and we skip the details.
\begin{cor} \label{Flow_gam_estim}
For any fix $t \in \mathbb R$, and any $|s - t| < \tau$ with $\tau \leq  \|\bar v_{q, \Gamma}\|_1^{-1}$, we have the following estimates: 
\begin{align} 
& \|(\nabla \tilde \Psi_t)^{-1}(\cdot, s)\|_N + \|\nabla \tilde \Psi_t (\cdot, s)\|_N \lesssim \lambda_q^N, \hspace{4.4cm}\, \forall N \in\{0,1,..., L-3\}, \label{Flow_gam_estim_1} \\
&\|\bar D_{t, \Gamma} (\nabla \tilde \Psi_t)^{-1}(\cdot, s)\|_N + \|\bar D_{t, \Gamma} \nabla \tilde \Psi_t (\cdot, s)\|_N \lesssim \delta_q^{1/2} \lambda_q^{N+1}, \hspace{2.2cm} \forall N \in\{0,1,..., L-3\}, \label{Flow_gam_estim_2} \\
&\|D \tilde X_t(\cdot, s)\|_N  \lesssim \lambda_q^N, \hspace{7.5cm} \forall N \in \{0,1,..., L-3\},  \label{Flow_gam_estim_3} \\
&\|(\nabla \tilde \Psi_t)^{-1}(\cdot, s)\|_{N+L-3} + \|\nabla \tilde \Psi_t(\cdot, s)\|_{N+L-3}  \lesssim \lambda_q^{L-3} \ell_q^{-N}, \, \hspace{2.0cm} \forall N \geq 0,  \label{Flow_gam_estim_4} \\
&\|\bar D_{t, \Gamma} (\nabla \tilde \Psi_t)^{-1}(\cdot, s)\|_{N+L-3} + \|\bar D_{t, \Gamma} \nabla \tilde \Psi_t(\cdot, s)\|_{N+L-3}  \lesssim \delta_q^{1/2} \lambda_q^{L-2} \ell_q^{-N}, \, \,  \forall N \geq 0,  \label{Flow_gam_estim_5} \\
&\|D \tilde X_t(\cdot, s)\|_{N+L-3} \lesssim \lambda_q^{L-3} \ell_q^{-N}, \hspace{5.8cm}\, \forall N \geq 0. \label{Flow_gam_estim_6}
\end{align}
where the implicit constants depend only on $\Gamma$, $M$, $\alpha$, and $N$.  $\bar D_{t, \Gamma} = \partial_t + \bar v_{q, \Gamma} \cdot \nabla.$
\end{cor}

\begin{rem} \label{remark_flow_bd}
We remark that, by choosing $a_0$ sufficiently large, we can make the bound of $\|\nabla \tilde \Psi\|_0$ independent of the parameters of the construction.
\end{rem}

For later purpose, we also report the following elementary stability lemma which plays a pivotal role in the next section. For a proof, consult \cite[Lemma 3.12]{vikram}.
\begin{lem} \label{flow_stabil}
Let us fix $t \in \mathbb R$ and $\tau \leq (\|\bar v_{q, \Gamma}\|_1 + \|\bar v_q\|_1)^{-1}$. Moreover, let $\tilde \Psi_t$ is given by \eqref{Flow_t_gam}, and $\Psi_t$ is given by \eqref{Flow_t}. Then, for $N \in \{0,1,...,L-4\}$, and any $|s-t| < \tau$, we have
    \begin{equation}
        \|\nabla \Psi_t(\cdot, s) - \nabla \tilde \Psi_t(\cdot, s)\|_N + \|(\nabla \Psi_t(\cdot, s))^{-1} - (\nabla \tilde \Psi_t(\cdot, s))^{-1}\|_N \lesssim \tau \frac{\delta_{q+1} \lambda_q^2 \ell_q^{-\alpha}}{\mu_{q+1}}\lambda_q^{N}, 
    \end{equation}
    while for $N \geq 0$,
    \begin{equation}
        \|\nabla \Psi_t(\cdot, s) - \nabla \tilde \Psi_t(\cdot, s)\|_{N + L-4} + \|(\nabla \Psi_t(\cdot, s))^{-1} - (\nabla \tilde \Psi_t(\cdot, s))^{-1}\|_{N+ L-4} \lesssim \tau \frac{\delta_{q+1} \lambda_q^2 \ell_q^{-\alpha}}{\mu_{q+1}}\lambda_q^{L-4} \ell_q^{-N}. 
    \end{equation}
Here all the implicit constants depend on $\Gamma$, $M$, $\al$ and $N$. 
\end{lem}

\subsection{Final output of the Newton step}

Here we summarize the Newton step and write down the final Euler-Boussinesq-Reynolds system after $\Gamma$ steps of the Newton iterations. In that context, we have  
\begin{align}
\label{aftergamma}
    \partial_t v_{q, \Gamma} + \div(v_{q, \Gamma} \otimes v_{q, \Gamma}) + \nabla p_{q,\Gamma} - \theta_{q, \Gamma} e_2&= \div S_{q, \Gamma} + \div\big( R_{q, \Gamma} + P_{q + 1, \Gamma} \big), \\
      \partial_t \theta_{q, \Gamma} + \div(v_{q, \Gamma} \theta_{q, \Gamma}) &= \div X_{q, \Gamma} + \div\big( T_{q, \Gamma} + Y_{q + 1, \Gamma} \big), \nonumber \\
      \div v_{q, \Gamma} &= 0, \nonumber
\end{align}
where notice that 
\begin{itemize}
    \item [(a)]The velocity and temperature, after $\Gamma$ steps, are defined by 
    \begin{equation}
        v_{q, \Gamma} = v_q + w_{q+1}^{(t)} = v_q + \sum_{n = 1}^\Gamma w_{q+1, n}^{(t)}, \quad  \theta_{q, \Gamma} = \theta_q + \theta_{q+1}^{(t)} = \theta_q + \sum_{n = 1}^\Gamma \theta_{q+1, n}^{(t)};
    \end{equation}
    
    \item [(b)]Thanks to the inductive definition \eqref{Newpres} of $p_{q, \Gamma}$, we have
    \begin{align} \label{p_q_Gam}
        p_{q, \Gamma} & =  p_q + \sum_{n = 1}^\Gamma p_{q+1, n}^{(t)} - \sum_{n = 0}^{\Gamma-1} \Delta^{-1} \div \div \big( R_{q, n} + \sum_{k \in \mathbb Z_{q,n}} \sum_{\xi \in\Lambda_R} g_{\xi, k, n+1}^2  A^1_{\xi, k, n} + \sum_{k \in \mathbb Z_{q,n}} \sum_{\zeta \in\Lambda_T}  h_{\zeta, k, n+1}^2  A^2_{\zeta, k, n} \big) \nonumber \\
        & \qquad \qquad - \frac{|w_{q+1}^{(t)}|^2}{2} + \langle \bar v_q - v_q, w_{q+1}^{(t)} \rangle;
    \end{align}

    \item [(c)]The error $S_{q, \Gamma}$ will be reduced by the Nash perturbation: 
    \begin{align}
      S_{q, \Gamma} &=  - \sum_{n=0}^{\Gamma -1} \sum_{\zeta \in \Lambda_R} \sum_{k \in \mathbb Z_{q,n}} g_{\xi, k, n+1}^2 a_{\xi, k, n}^2 (\nabla \Psi_k)^{-1} \xi  \otimes \xi (\nabla \Psi_k)^{-T} \\
      & \qquad - \sum_{n=0}^{\Gamma -1} \sum_{\zeta \in \Lambda_T} \sum_{k \in \mathbb Z_{q,n}}h_{\zeta, k, n+1}^2 b_{\zeta, k, n}^2 (\nabla \Psi_k)^{-1} \zeta  \otimes \zeta (\nabla \Psi_k)^{-T}; \nonumber
    \end{align}
   
 \item [(d)]The error $X_{q, \Gamma}$ will be reduced by the Nash perturbation: 
    \begin{eqnarray}
        X_{q, \Gamma} = - \sum_{n=0}^{\Gamma -1} \sum_{\zeta \in \Lambda} \sum_{k \in \mathbb Z_{q,n}} h^2_{\zeta, k, n+1} b_{\zeta, k, n} =  - \sum_{\zeta, k, n} h^2_{\zeta, k, n+1}  c^2_{\zeta, k, n} (\nabla \Psi_k)^{-1} \zeta ; 
    \end{eqnarray}

    \item [(e)]The errors $R_{q, \Gamma}$ and $T_{q, \Gamma}$, in view of Proposition~\ref{NewIter}, are sufficiently small to be incorporated into $R_{q+1}$;

    \item [(f)]In view of \eqref{SmalStr}, we can write the residual error $P_{q+1, \Gamma}$ as 
    \begin{equation} \label{big_P_err}
    P_{q+1, \Gamma} = R_{q} - R_{q, 0} + w_{q+1}^{(t)} \mathring \otimes w_{q+1}^{(t)} + (v_q - \bar v_q) \mathring \otimes   w_{q+1}^{(t)} +  w_{q+1}^{(t)} \mathring \otimes (v_q - \bar v_q); 
    \end{equation}

 \item [(g)]In view of \eqref{SmalStr1}, we can write the residual error $Y_{q+1, \Gamma}$ as 
    \begin{equation} \label{big_P_err}
    Y_{q+1, \Gamma} = T_{q} - T_{q, 0} + w_{q+1}^{(t)} \theta_{q+1}^{(t)} + (v_q - \bar v_q) \theta_{q+1}^{(t)} +  w_{q+1}^{(t)} (\theta_q - \bar \theta_q).
    \end{equation}
\end{itemize}
Finally notice that, thanks to Proposition~\ref{NewIter} and Lemma~\ref{w_t_estim}, it holds that 
\begin{align*}
    & \supp_t S_{q, \Gamma} \cup \supp_t R_{q, \Gamma} \cup \supp_t P_{q+1, \Gamma} \cup  \supp_t X_{q, \Gamma} \cup \supp_t T_{q, \Gamma} \cup \supp_t Y_{q+1, \Gamma} \\
      &\subset [-2 + (\delta_q^{1/2} \lambda_q)^{-1} - 3\Gamma \tau_q, -1 - (\delta_q^{1/2} \lambda_q)^{-1} + 3\Gamma \tau_q]  
       \cup [1 + (\delta_q^{1/2} \lambda_q)^{-1} - 3 \Gamma \tau_q, 2 - (\delta_q^{1/2} \lambda_q)^{-1} + 3 \Gamma \tau_q]. 
        \end{align*}

\section{Nash Iterations} 
\label{section_Nash}

It is well-known that any typical convex integration scheme also suffers from loss of (material) derivative problem and requires another regularization procedure to resolve this problem. Therefore, in this section, we first follow closely \cite{IsPhD} and fix the length scale for the mollification of Reynolds stresses along the flow.

\subsection{Regularization of Reynolds stress} 
\label{Section_Moli_Along_flow} 
Let us take a standard temporal mollifier $\varrho$ on $\R$ and denote by $\tilde X_t$, the Lagrangian flow of $ \bar v_{q,\Gamma}$ starting at $t$ (cf. \eqref{Lagr_t_gam}). Then we fix the material mollification scale
\begin{equation*}
    \ell_{t, q} = \delta_{q}^{-1/2}\lambda_q^{-1/3} \lambda_{q+1}^{-2/3}.
\end{equation*}
It is easy to verify that, for all $\alpha > 0$ sufficiently small, we have
\begin{equation*}
    \delta_{q}^{1/2} \lambda_q < \ell_{t,q}^{-1} < \delta_{q+1}^{1/2} \lambda_{q+1}, \quad  \ell_{t, q} < \mu_{q+1}^{-1} <  \tau_q.
\end{equation*}
Next, following \cite{IsPhD}, we define 
\begin{align} 
    \bar R_{q,n}(x, t) &= \int_{-\ell_{t,q}}^{\ell_{t, q}} R_{q,n} (\tilde X_t(x, t + s), t + s) \varrho_{\ell_{t, q}}(s) ds, \label{tmoli_def}\\
     \bar T_{q,n}(x, t) &= \int_{-\ell_{t,q}}^{\ell_{t, q}} T_{q,n} (\tilde X_t(x, t + s), t + s) \varrho_{\ell_{t, q}}(s) ds. \label{tmoli_def_01}
\end{align}

We now collect the main estimates on $\bar R_{q,n}$ and $\bar T_{q,n}$. For a proof, one may consult \cite{IsPhD}, \cite{BDLS16}, or \cite[Lemma 4.1]{vikram}.
\begin{lem} 
\label{tmoli_estim}
We have the following estimates:
\begin{align} 
   & \|\bar R_{q,n}\|_N  \lesssim \delta_{q+1,n} \lambda_q^{N-\alpha}, \hspace{3.9cm} \|\bar T_{q,n}\|_N \lesssim \delta_{q+1,n} \lambda_q^{N-\alpha}, \nonumber \\
   &\hspace{11.1cm} \forall N\in\{0,1,...,L-2\}. \label{tmoli1} \\
   & \|\bar D_{t, \Gamma} \bar R_{q,n} \|_N \lesssim \tau_q^{-1} \delta_{q+1,n} \lambda_q^{N- \alpha}, \hspace{2.6cm}  \|\bar D_{t, \Gamma} \bar T_{q,n} \|_N \lesssim \tau_q^{-1} \delta_{q+1,n} \lambda_q^{N- \alpha}, \nonumber \\
   & \hspace{11.1cm} \forall N\in \{0,1,...,L-2\}. \label{tmoli2} \\
   & \|\bar D_{t, \Gamma}^2 \bar R_{q,n} \|_N \lesssim \ell_{t,q}^{-1} \tau_q^{-1} \delta_{q+1,n} \lambda_q^{N- \alpha}, \hspace{2.1cm} \|\bar D_{t, \Gamma}^2 \bar T_{q,n} \|_N \lesssim \ell_{t,q}^{-1} \tau_q^{-1} \delta_{q+1,n} \lambda_q^{N- \alpha}, \nonumber \\
   &\hspace{11.1cm} \forall N\in \{0,1,...,L-2\}. \label{tmoli3} \\
   & \|\bar R_{q,n}\|_{N+ L - 2} \lesssim \delta_{q+1,n}\lambda_q^{L-2-\alpha} \ell_q^{-N}, \hspace{2.2cm} \|\bar T_{q,n}\|_{N+ L - 2} \lesssim \delta_{q+1,n}\lambda_q^{L-2-\alpha} \ell_q^{-N}, \nonumber \\
   & \hspace{13.1cm} \forall N \geq 0. \label{tmoli4} \\
    & \|\bar D_{t, \Gamma} \bar R_{q,n} \|_{N+L-2} \lesssim \tau_q^{-1} \delta_{q+1,n} \lambda_{q}^{L-2-\alpha} \ell_q^{-N}, \hspace{0.9cm} \|\bar D_{t, \Gamma} \bar T_{q,n} \|_{N+L-2} \lesssim \tau_q^{-1} \delta_{q+1,n} \lambda_{q}^{L-2-\alpha} \ell_q^{-N},\nonumber \\
     &\hspace{13.1cm} \forall N \geq 0. \label{tmoli5} \\
   & \|\bar D_{t, \Gamma}^2 \bar R_{q,n} \|_{N+L-2} \lesssim \ell_{t,q}^{-1} \tau_q^{-1} \delta_{q+1,n} \lambda_{q}^{L-2-\alpha} \ell_q^{-N}, \quad \|\bar D_{t, \Gamma}^2 \bar T_{q,n} \|_{N+L-2} \lesssim \ell_{t,q}^{-1} \tau_q^{-1} \delta_{q+1,n} \lambda_{q}^{L-2-\alpha} \ell_q^{-N}, \nonumber \\
   & \hspace{13.1cm} \forall N \geq 0, \label{tmoli6}
\end{align}
where all the implicit bounds are dependent on $\Gamma$, $M$, $\alpha$ and $N$.
\end{lem}

\subsection{Nash perturbation} 
\label{Nash_Construction}
We make use of the shear flow in the directions of $\xi$ given by Lemma~\ref{geom}, and Lemma~\ref{lem:geo2} as building blocks of the Nash perturbations. More precisely, for every $\sigma \in \Lambda$, we define $\mathbb{W}_{\sigma}:\mathbb T_{2\pi}^2 \rightarrow \mathbb R^2$, and, for every $\zeta \in \Lambda_T$, we define $\mathbb{V}_{\zeta}:\mathbb T_{2\pi}^2 \rightarrow \mathbb R$ by
\begin{equation} \label{defW}
    \mathbb{W}_\sigma (x) = \frac{1}{\sqrt{2}} \big ( e^{i \sigma^\perp \cdot x} + e^{- i \sigma^\perp \cdot x} \big) \sigma, \quad
    \mathbb{V}_\zeta (x) = \frac{1}{\sqrt{2}} \big ( e^{i \zeta^\perp \cdot x} + e^{- i \zeta^\perp \cdot x} \big), 
\end{equation}
where $\mathbb T_{2\pi}^2 = \mathbb R^2/(2\pi \mathbb Z)^2$. For ease of notation, we also denote the corresponding stream-function by
\begin{equation*}
    \Phi_\sigma(x) = \frac{i}{\sqrt{2}}\big( e^{i\sigma^\perp \cdot x} - e^{-i \sigma^\perp \cdot x}\big). 
\end{equation*}
It is straightforward to check the following simple properties of the vectors $\mathbb{W}_{\sigma}$, and the scalars $\mathbb{V}_{\zeta}$ presented in the lemma below. 
\begin{lem}
The vector fields $\mathbb{W}_{\sigma}:\mathbb T_{2\pi}^2 \rightarrow \mathbb R^2$, and the scalars $\mathbb{V}_{\zeta}:\mathbb T_{2\pi}^2 \rightarrow \mathbb R$, as defined above, satisfy 
\begin{equation*}
    \begin{cases}
        \div (\mathbb{W}_{\sigma} \otimes \mathbb{W}_{\sigma}) = 0, \\ 
        \div (\mathbb{V}_{\zeta} \mathbb{W}_{\zeta}) = 0, \\ 
        \div \mathbb{W}_{\sigma} = 0.
    \end{cases}
\end{equation*}
and 
\begin{equation*}
    \fint_{\mathbb T_{2\pi} ^2} \mathbb{W}_{\sigma} \otimes \mathbb{W}_{\sigma} = \sigma \otimes \sigma, \quad \fint_{\mathbb T_{2\pi} ^2} \mathbb{V}_{\zeta} \mathbb{W}_{\zeta} = \zeta.
\end{equation*}
\end{lem}

With the help of the above mentioned building blocks, we are ready to describe the Nash perturbations. To that context, let us first denote by $\tilde\Psi_k$, the backwards flow of $\bar v_{q,\Ga}$ starting from $t = t_k$:
\begin{equation*}
    \begin{cases}
        \partial_t \td\Psi_k + \bar v_{q,\Ga} \cdot \nabla \td\Psi_k = 0, \\ 
        \Psi \big|_{t = t_k} = x.
    \end{cases}
\end{equation*}
Next, we define 
\begin{equation*}
\bar b_{\zeta, k, n} = \lambda_q^{-\alpha/2}\delta_{q+1, n}^{1/2} \eta_k \Gamma^{1/2}_\zeta \bigg(\lambda_q^{\alpha}\delta_{q+1, n}^{-1} \nabla \tilde \Psi_k \bar T_{q, n} \bigg),
\end{equation*}
\begin{align*}
\bar a_{\xi, k, n} =  \delta_{q+1, n}^{1/2} & \eta_k \gamma_\xi \bigg(\nabla \tilde \Psi_k \nabla \tilde \Psi_k^T - \nabla \tilde \Psi_k \frac{\bar R_{q, n}}{\delta_{q+1,n}} \nabla \tilde \Psi_k^T \\
& \qquad - \nabla \tilde \Psi_k \Big(\sum_{\zeta \in \Lambda_T} \sum_{k' \in \mathbb Z_{q,n}} \delta_{q+1, n}^{-1}  \bar b^2_{\zeta, k', n} (\nabla \tilde \Psi_{k'} )^{-1}(\zeta \otimes \zeta) (\nabla \tilde \Psi_{k'} )^{-T}\Big) \nabla \tilde \Psi_k^T \bigg), \nonumber
\end{align*}
and the principal part of the Nash perturbation by 
\begin{align}
    w_{q+1}^{(p)} &:= \sum_{n = 0}^{\Gamma - 1} \sum_{k \in \mathbb Z_{q,n}} \sum_{\xi\in\Lambda_R} g_{\xi,k,n+1} \bar a_{\xi,k,n} (\na\tilde\Psi_k)^{-1} \mathbb{W}_\xi (\la_{q+1} \tilde\Psi_k) \nonumber \\
& \qquad \qquad + \sum_{n = 0}^{\Gamma - 1} \sum_{k \in \mathbb Z_{q,n}} \sum_{\zeta \in\Lambda_T} h_{\zeta,k,n+1} \bar b_{\zeta,k,n} (\na\tilde\Psi_k)^{-1} \mathbb{W}_\zeta (\la_{q+1} \tilde\Psi_k). \label{new_001}\\
\theta_{q+1}^{(p)} &:= \sum_{n = 0}^{\Gamma - 1} \sum_{k \in \mathbb Z_{q,n}} \sum_{\zeta \in\Lambda_T} h_{\zeta,k,n+1} \bar b_{\zeta,k,n} \mathbb{V}_\zeta (\la_{q+1} \tilde\Psi_k). \label{new_002}
\end{align}
Arguing as before, we first conclude that $\bar a_{\xi, k, n}$ is well-defined. Moreover, with the help of the stream-function, we may write
\begin{equation*}
    (\na\tilde\Psi_k)^{-1} \mathbb{W}_\sigma (\la_{q+1} \tilde\Psi_k) = \frac{1}{\lambda_{q+1}} \nabla^\perp \big( \Phi_\sigma (\lambda_{q+1} \tilde \Psi_k)\big). 
\end{equation*}
However, $w^{(p)}_{q+1}$ may not be divergence-free, and we need a ``corrector'' to make it divergence-free. In what follows, we define the ``corrector'' as follows.
\begin{align}\label{w.c}
    w_{q+1}^{(c)} &:= \frac{1}{\lambda_{q+1}} \sum_{n = 0}^{\Gamma - 1} \sum_{k \in \mathbb Z_{q,n}} \sum_{\xi\in\Lambda_R} g_{\xi,k,n+1} \nabla^\perp \bar a_{\xi, k, n} \Phi_\xi(\lambda_{q+1}\tilde \Psi_k) \\
    & \qquad \qquad + \frac{1}{\lambda_{q+1}} \sum_{n = 0}^{\Gamma - 1} \sum_{k \in \mathbb Z_{q,n}} \sum_{\zeta\in\Lambda_T} h_{\zeta,k,n+1} \nabla^\perp \bar b_{\zeta, k, n} \Phi_\zeta(\lambda_{q+1}\tilde \Psi_k). \nonumber
\end{align}
Then it is clear that the total Nash perturbation for the velocity $w_{q+1}^{(s)} := w_{q+1}^{(p)} + w_{q+1}^{(c)}$ is divergence-free, and
\begin{align} \label{total_Nash}
    w_{q+1}^{(s)} &= \frac{1}{\lambda_{q+1}} \nabla^\perp \bigg( \sum_{n = 0}^{\Gamma - 1} \sum_{k \in \mathbb Z_{q,n}} \sum_{\xi\in\Lambda_R} g_{\xi, k, n+1} \bar a_{\xi, k, n} \Phi_\xi(\lambda_{q+1} \tilde \Psi_k) \bigg) \nonumber \\
 &\qquad \qquad +  \frac{1}{\lambda_{q+1}} \nabla^\perp \bigg( \sum_{n = 0}^{\Gamma - 1} \sum_{k \in \mathbb Z_{q,n}} \sum_{\zeta\in\Lambda_T} h_{\zeta, k, n+1} \bar b_{\zeta, k, n} \Phi_\zeta(\lambda_{q+1} \tilde \Psi_k) \bigg)
\end{align}
Furthermore, we also define the ``mean free corrector'' $\theta_{q+1}^{(c)}(t) := - \int_{\T^2} \theta_{q+1}^{(p)} (x,t)\,dx$. Indeed, in view of this, the total Nash perturbation for the temperature $\theta_{q+1}^{(s)} = \theta_{q+1}^{(p)} + \theta_{q+1}^{(c)}$ is mean free.

We now denote the total perturbation for the velocity and the temperature by 
\begin{equation*}
    w_{q+1} = w_{q+1}^{(t)} + w_{q+1}^{(s)}, \quad \Theta_{q+1} = \theta_{q+1}^{(t)} + \theta_{q+1}^{(s)}.
\end{equation*}
Note that, thanks to Lemma~\ref{osc_prof} and the fact $\supp_t \bar a_{\xi, k, n}, \supp_t \bar b_{\zeta, k, n}, \supp_t \bar b_{\zeta, k, n} \subset \supp_t \eta_k$, we conclude that the terms in both sums have disjoint temporal supports with respect to each pair.

\subsection{New Euler-Boussinesq-Reynolds system} 
We can now write down the Euler-Boussinesq-Reynolds system at the level $q+1$. Indeed, let the velocity $v_{q+1}$ and temperature $\theta_{q+1}$ of proposition \ref{Main_prop} be 
\begin{equation}
    v_{q+1} = v_{q, \Gamma} + w_{q+1}^{(s)} = v_q + w_{q+1}^{(t)} +  w_{q+1}^{(s)} = v_q + w_{q+1},
\end{equation}
\begin{equation}
    \theta_{q+1} = \theta_{q, \Gamma} + \theta_{q+1}^{(s)} = \theta_q + \theta_{q+1}^{(t)} +  \theta_{q+1}^{(s)} = \theta_q + \Theta_{q+1}.
\end{equation}
Then the Euler-Bousinesq-Reynolds system at level $q+1$ takes the form
\begin{eqnarray*}
    \partial_t v_{q+1} + \div (v_{q+1} \otimes v_{q+1}) + \nabla p_{q+1} - \theta_{q+1} e_2= \div R_{q+1}, & \\
    \partial_t \theta_{q+1} + \div (v_{q+1} \theta_{q+1}) = \div T_{q+1}, &
\end{eqnarray*}
with 
\begin{equation}
    p_{q+1} = p_{q, \Gamma}  + \langle \bar v_q - v_q, w_{q+1}^{(s)} \rangle,
\end{equation}
and 
\begin{equation}
    R_{q+1} = R_{q+1, L} + R_{q+1, O} +  R_{q+1, R}, \quad  T_{q+1} = T_{q+1, L} + T_{q+1, O} +  T_{q+1, R},
\end{equation}
where 
\begin{itemize}
\item [(1)] The linear errors, denoted by $R_{q+1, L}$ and $T_{q+1, L}$, are defined by 
    \begin{equation*}
    R_{q+1, L} = \mathcal{R}\big(\bar D_{t, \Gamma} w_{q+1}^{(s)} + w_{q+1}^{(s)}\cdot \nabla \bar v_{q,\Gamma} + \theta_{q+1}^{(s)} e_2 \big), \quad T_{q+1, L} = \mathcal{R}\big(\bar D_{t, \Gamma} \theta_{q+1}^{(s)} + w_{q+1}^{(s)}\cdot \nabla \bar \theta_{q,\Gamma} \big).
     \end{equation*}
\item [(2)] The oscillation errors, denoted by $R_{q+1, O}$ and $T_{q+1, O}$, are defined by
    \begin{equation*}
        R_{q+1, O} = \mathcal{R} \div(S_{q, \Gamma} + w_{q+1}^{(s)} \otimes w_{q+1}^{(s)}), \quad
        T_{q+1, O} = \mathcal{R} \div(X_{q, \Gamma} + w_{q+1}^{(s)} \theta_{q+1}^{(s)}).
    \end{equation*}
\item [(3)] The residual errors, denoted by $R_{q+1, R}$ and $T_{q+1, R}$, are defined by
     \begin{align*}
         R_{q+1, R}&=  R_{q, \Gamma} + P_{q+1, \Gamma} + w_{q+1}^{(s)} \mathring \otimes (v_q - \bar v_q) + (v_q - \bar v_q)\mathring \otimes w_{q+1}^{(s)}, \\
         T_{q+1, R}&=  T_{q, \Gamma} + Y_{q+1, \Gamma} + w_{q+1}^{(s)}  (\theta_q - \bar \theta_q) + (v_q - \bar v_q)\theta_{q+1}^{(s)}
     \end{align*}
\end{itemize}

\begin{rem} 
Notice that $w_{q+1}^{(t)}$, $R_{q, \Gamma}$, $P_{q+1, \Gamma}$, $S_{q,\Gamma}$, $T_{q, \Gamma}$, $Y_{q+1, \Gamma}$, and $X_{q,\Gamma}$ have temporal supports inside the set 
\begin{equation*}
    [-2 + (\delta_q^{1/2} \lambda_q)^{-1} - 3\Gamma \tau_q, -1 - (\delta_q^{1/2} \lambda_q)^{-1} + 3\Gamma \tau_q] \cup [1 + (\delta_q^{1/2} \lambda_q)^{-1} - 3 \Gamma \tau_q, 2 - (\delta_q^{1/2} \lambda_q)^{-1} + 3 \Gamma \tau_q].
\end{equation*}
Therefore, we conclude the same for $w_{q+1}^{(s)}$, $\theta_{q+1}^{(s)}$, and, hence, for $w_{q+1}$, $\Theta_{q+1}$, $R_{q+1}$ and $T_{q+1}$. However, Notice that we can choose $a_0$ large enough to make sure that
\begin{equation*}
    (\delta_{q+1}^{1/2} \lambda_{q+1})^{-1} + 3 \Gamma (\delta_{q}^{1/2} \lambda_q)^{-1} \lambda_{q+1}^{-\alpha} < (\delta_q^{1/2} \lambda_q)^{-1}.
\end{equation*}
Hence, the conditions \eqref{Main_prop_eqn_2} and \eqref{ind_est_6} are true at the level $q+1$.
\end{rem}

\subsection{Estimates for both the perturbations}
We begin by collecting estimates on the amplitudes of both Nash perturbations in the following lemma. The proof of the lemma follows closely the proof of Lemma~\ref{a_estim}, with the help of Corollary~\ref{Flow_gam_estim} and Lemma \ref{tmoli_estim}.

\begin{lem} \label{a_bar_estim}
The following estimates are valid: 
\begin{align} 
    &\|\bar a_{\xi, k, n}\|_N \lesssim \delta_{q+1, n}^{1/2} \lambda_q^N, \hspace{3.0cm} \|\bar b_{\zeta, k, n}\|_N \lesssim \delta_{q+1, n}^{1/2} \lambda_q^N, \hspace{1.4cm} \,\forall N \in \{0,1,..., L-3\}. \label{a_bar_estim_1} \\
    &\|\bar D_{t, \Gamma} \bar a_{\xi, k, n}\|_N \lesssim \delta_{q+1, n}^{1/2} \tau_q^{-1} \lambda_q^N, \qquad \qquad \quad \|\bar D_{t, \Gamma} \bar b_{\zeta, k, n}\|_N \lesssim \delta_{q+1, n}^{1/2} \tau_q^{-1} \lambda_q^N, \,\,\, \forall N \in \{0,1,..., L-3\}. \label{a_bar_estim_2}\\
    &\|\bar a_{\xi, k, n}\|_{N+L-3} \lesssim \delta_{q+1, n}^{1/2} \lambda_q^{L-3} \ell_q^{-N}, \hspace{1.3cm}\, \|\bar b_{\zeta, k, n}\|_{N+L-3} \lesssim \delta_{q+1, n}^{1/2} \lambda_q^{L-3} \ell_q^{-N}, \hspace{1.4cm} \forall N \geq 0. \label{a_bar_estim_3} \\
    &\|\bar D_{t, \Gamma} \bar a_{\xi, k, n}\|_{N+L-3}  \lesssim \delta_{q+1, n}^{1/2} \lambda_q^{L- 3}  \tau_q^{-1} \ell_q^{-N}, \,\, \|\bar D_{t, \Gamma} \bar b_{\zeta, k, n}\|_{N+L-3}  \lesssim \delta_{q+1, n}^{1/2} \lambda_q^{L- 3}  \tau_q^{-1} \ell_q^{-N}, \,\,\, \forall N \geq 0, \label{a_bar_estim_4}
\end{align}
where all the implicit bounds are dependent on $\Gamma$, $M$, $\al$ and $N$.
\end{lem}

We can now give the estimates of the perturbations $w_{q+1}^{(s)}$, $\theta_{q+1}^{(s)}$. In fact, at this point we can also fix the constant $M_0$ and verify the inductive estimates \eqref{ind_est_1} and \eqref{ind_est_2}.

\begin{lem} \label{velo_estim_fin}
There exists a positive constant $M_0$, which depends only on the parameters $\beta$ and $L$, such that 
\begin{equation} 
\label{spatial_pert_est}
\|w_{q+1}^{(s)}\|_{N}, \, \|\theta_{q+1}^{(s)}\|_{N} \leq \frac{M_0}{2} \delta_{q+1}^{1/2} \lambda_{q+1}^N, \,\,\, \forall N \in \{0, 1,..., L\}.
\end{equation}
Moreover, we also have
\begin{equation}
\label{total_pert_est}
\|w_{q+1}\|_N, \, \|\Theta_{q+1}\|_N \leq M_0 \delta_{q}^{1/2} \lambda_{q+1}^N, \,\,\, \forall N \in \{0,1,...,L\}.
\end{equation}
\end{lem}

\begin{proof}
We can follow \cite{vikram} to establish the bounds related to the velocity perturbation $w_{q+1}$. Note that in the proof, for a constant $C_L$, we need to choose $M_0 = 2 C_L \sup_{\xi, k, n} |g_{\xi, k, n}|$ to get the desired estimates for the velocity perturbation. Now we only present brief details about the estimates related to the temperature perturbation $\Theta_{q+1}$. Notice that, the corrector term satisfies $|\theta^{(c)}_{q+1}(t)| \lesssim \|\theta_{q+1}^{(p)}\|_0 $, while thanks to disjoint temporal supports of $\{h_{\zeta, k, n+1} \bar b_{\zeta, k, n}\}_{\zeta, k, n}$, we have from \eqref{new_002} 
\begin{eqnarray*}
    [\theta_{q+1}^{(p)}]_N &\leq& \sup_{\zeta, k, n} |h_{\zeta, k, n+1}| [\bar b_{\zeta, k, n} \mathbb{V}_\zeta(\lambda_{q+1} \tilde \Psi_k)]_{N} \\ 
    & \lesssim & \sup_{\zeta, k, n} |h_{\zeta, k, n+1}| \big([\mathbb{V}_\zeta(\lambda_{q+1} \tilde \Psi_k)]_{N} \|\bar b_{\zeta, k, n}\|_0  + \|\mathbb{V}_\zeta(\lambda_{q+1} \tilde \Psi_k)\|_0 \|\bar b_{\zeta, k, n}\|_{N}\big).
\end{eqnarray*}
We can make use of proposition \ref{comp_estim} to conclude
\begin{eqnarray*}
    [\mathbb{V}_\zeta(\lambda_{q+1} \tilde \Psi_k)]_{N} & \lesssim & \big(\|D\mathbb{V}_\zeta(\lambda_{q+1}\cdot)\|_{N-1} \|\nabla \tilde \Psi\|_0^{N}    + [\mathbb{V}_\zeta(\lambda_{q+1}\cdot)]_1 \|\nabla \tilde \Psi_k\|_{N-1}) \\ 
    & \lesssim & (\lambda_{q+1}^{N+1} + \lambda_{q+1} \ell_q^{-N}).
\end{eqnarray*}
Now we can follow again \cite{vikram}, choose $a_0$ large enough, and 
$$
M_0 = \max\big\{2 C_L \sup_{\xi, k, n} |g_{\xi, k, n}|, 2 D_L \sup_{\zeta, k, n} |h_{\zeta, k, n}| \big\},
$$
where $C_L, D_L$ are constants depending on $L$ only, to conclude the proof of the estimate \eqref{spatial_pert_est}.

Finally, taking advantage of lemma \ref{w_t_estim}, we conclude that 
\begin{equation*}
    \|\theta_{q+1}^{(t)}\|_N \lesssim \frac{\delta_{q+1}\lambda_q}{\mu_{q+1}} \ell_q^{-N-\alpha} \lesssim  \delta_{q+1}^{1/2}\bigg(\frac{\lambda_q}{\lambda_{q+1}}\bigg)^{1/3} \lambda_{q+1}^{N}.
\end{equation*}
Therefore, by selecting $a_0$ to be sufficiently large, we can complete the proof of the estimate \eqref{total_pert_est}.
\end{proof}

\begin{cor} \label{velo_corrol_fin}
The following estimates hold: 
\begin{equation} \label{propag_velo_1}
    \|v_{q+1}\|_0, \, \|\theta_{q+1}\|_0 \leq M (1 - \delta_{q+1}^{1/2}), 
\end{equation}
\begin{equation} \label{propag_velo_2}
    \|v_{q+1}\|_N, \, \|\theta_{q+1}\|_N \leq M \delta_{q+1}^{1/2} \lambda_{q+1}^N, \,\,\, \forall N \in \{1,2,...,L\}
\end{equation}
\begin{equation} \label{propag_velo_3}
    \|v_{q+1} - v_q\|_0 + \frac{1}{\lambda_{q+1}} \|v_{q+1} - v_q\|_1 + \|\theta_{q+1} - \theta_q\|_0 + \frac{1}{\lambda_{q+1}} \|\theta_{q+1} - \theta_q\|_1 \leq 2 M \delta_{q+1}^{1/2}.
\end{equation}
\end{cor}

\begin{proof}
For a proof, consult \cite{vikram}.
\end{proof}

\subsection{Estimates for the linear errors \texorpdfstring{$R_{q + 1, L}$ and $T_{q+1,L}$}{rrrr}} 

We first write 
\begin{equation*}
    R_{q+1, L} = \underbrace{\mathcal{R} (w_{q+1}^{(s)} \cdot \nabla \bar v_{q, \Gamma})}_{\text{Nash error}} + \underbrace{\mathcal{R}( \bar D_{t, \Gamma} w_{q+1}^{(s)})}_{\text{Transport error}} + \,\mathcal{R} (\theta_{q+1}^{(s)} e_2),
\end{equation*}
and we give details of the estimates separately. However, let us first gather some preliminary estimates on material derivatives.

\begin{lem} 
\label{high_mat_der}
The following bounds are valid: 
\begin{align}
& \|\bar D_{t, \Gamma} \nabla \bar v_{q, \Gamma}\|_{N}, \, \|\bar D_{t, \Gamma} \nabla \bar \theta_{q, \Gamma}\|_{N} \lesssim \delta_q \lambda_q^{N+2}, \hspace{4.1cm} \,\,\,\,\, \forall N\in \{0,1,...,L-4\},  \label{mat_spa_velo_1} \\
&\|\bar D_{t, \Gamma} \nabla \bar v_{q, \Gamma}\|_{N + L -4}, \, \|\bar D_{t, \Gamma} \nabla \bar \theta_{q, \Gamma}\|_{N + L -4} \lesssim \delta_q \lambda_q^{L-2} \ell_q^{-N}, \hspace{2.3cm} \forall N\geq 0, \label{mat_spa_velo_2} \\
&\|\bar D_{t, \Gamma}^2 \bar a_{\xi, k, n}\|_N, \, \|\bar D_{t, \Gamma}^2 \bar b_{\zeta, k, n}\|_N \lesssim \delta_{q+1, n}^{1/2} \tau_q^{-1} \ell_{t,q}^{-1}\lambda_q^N, \hspace{3.1cm} \,\forall N \in \{0,1,..., L-4\}, \label{a_bar_2nd_material} \\
&\|\bar D_{t, \Gamma}^2 \bar a_{\xi, k, n}\|_{N + L-4 }, \, \|\bar D_{t, \Gamma}^2 \bar b_{\zeta, k, n}\|_{N + L-4 }  \lesssim \delta_{q+1, n}^{1/2} \lambda_q^{L-4}  \tau_q^{-1} \ell_{t,q}^{-1}\ell_q^{-N}, \qquad  \forall N \geq 0, \label{a_bar_2nd_material_2}
\end{align}
where all the implicit bounds are dependent on $\Gamma$, $M$, $\alpha$ and $N$. 
\end{lem}

\begin{proof}
For the proof of \eqref{mat_spa_velo_1} and \eqref{mat_spa_velo_2} related to the velocity filed, we simply follow \cite{vikram}. However, for the same related to temperature, we give details below. As usual, we first write    
    \begin{equation*}
        \bar D_{t, \Gamma} \nabla \bar \theta_{q, \Gamma} = \bar D_{t} \nabla \bar \theta_{q} + \bar D_{t} \nabla \theta_{q+1}^{(t)} + \theta_{q+1}^{(t)}\cdot \nabla \nabla\bar \theta_{q, \Gamma}.
    \end{equation*}
We can rewrite the first term as 
    \begin{equation*}
        \bar D_{t} \nabla \bar \theta_{q} = \nabla \bar D_{t} \bar \theta_q - (\nabla \bar \theta_q)^2.
    \end{equation*}
Now to get the desired estimate, we mollify the second equation in Euler-Boussinesq-Reynolds system \eqref{ER} to obtain 
    \begin{equation*}
        \bar D_t \bar \theta_q + \div \big( (v_q \theta_q)*\zeta_{\ell_q} - \bar v_q \bar \theta_q \big) = \div T_{q, 0},     
    \end{equation*}
    and, therefore, 
    \begin{eqnarray*}
        \|\bar D_t \nabla \bar \theta_q\|_N & \lesssim & \|\bar D_t \bar \theta_q\|_{N+1} + \|\bar \theta_q\|_{N+1}\|\bar \theta_q\|_1 \\ 
        & \lesssim & \|(v_q \theta_q)*\zeta_{\ell_q} - \bar v_q  \bar \theta_q\|_{N+2}  + \|T_{q, 0}\|_{N+2} + \|\bar \theta_q\|_{N+1}\|\bar \theta_q\|_1.
    \end{eqnarray*}
Making use of the Constantin-E-Titi commutator estimate (cf. Proposition \ref{CET_comm}), we have 
    \begin{equation*}
        \|(v_q \theta_q)*\zeta_{\ell_q} - \bar v_q  \bar \theta_q\|_{N+2} \lesssim \delta_q \lambda_q^{N+2}, \,\,\, \forall N \in \{0,1,...,L-4\}, 
    \end{equation*}
    \begin{equation*}
        \|(v_q \theta_q)*\zeta_{\ell_q} - \bar v_q  \bar \theta_q\|_{N+L-2} \lesssim \ell_q^{2-N} \delta_q \lambda_q^{L} \lesssim \delta_q \lambda_q^{L-2} \ell_q^{-N}, \,\,\, \forall N \geq 0.
    \end{equation*}
To deal with the second term, we make use of Lemma \ref{smoli_estim}, and we use inductive estimates to bound the last term. To summarize, $\bar D_{t} \nabla \bar \theta_q$ satisfies the desired estimates.
Next, we write
    \begin{equation*}
        \|\bar D_{t} \nabla \theta_{q+1}^{(t)}\|_N \lesssim \|\bar D_t \theta_{q+1}^{(t)}\|_{N+1} + \|\bar \theta_q\|_{N+1}\|\theta_{q+1}^{(t)}\|_1 +  \|\bar \theta_q\|_{1}\|\theta_{q+1}^{(t)}\|_{N+1}.
    \end{equation*}
Making use of Lemma \ref{w_t_estim} and using $\mu_{q+1} > \delta_q^{1/2} \lambda_q$, we obtain
    \begin{align*}
        &\|\bar D_{t} \nabla \theta_{q+1}^{(t)}\|_N \lesssim \delta_{q+1} \lambda_q^{N+2} \ell_q^{-\alpha}, \,\,\, \forall N \in \{0, 1,..., L-4\}, \\
        &\|\bar D_{t} \nabla \theta_{q+1}^{(t)}\|_{N+L-4} \lesssim \delta_{q+1} \lambda_q^{L-2} \ell_q^{-N-\alpha}, \,\,\, \forall N \geq 0.
    \end{align*}
Therefore, $\bar D_t \nabla w_{q+1}^{(t)}$ also satisfies the desired estimates. Finally, observe that
    \begin{equation*}
        \|\theta_{q+1}^{(t)} \cdot \nabla \nabla \bar \theta_{q, \Gamma}\|_N \lesssim \|\theta_{q+1}^{(t)}\|_N \|\bar \theta_{q, \Gamma}\|_2 + \|\theta_{q+1}^{(t)}\|_{0} \|\bar \theta_{q, \Gamma}\|_{N+2},
    \end{equation*}
and in view of Lemma \ref{w_t_estim} and Corollary \ref{Gamma_velo_estim}, we conclude 
    \begin{align*}
       & \|w_{q+1}^{(t)} \cdot \nabla \nabla \bar v_{q, \Gamma}\|_N \lesssim \delta_{q+1} \frac{\delta_q^{1/2} \lambda_q}{\mu_{q+1}} \ell_q^{- \alpha} \lambda_q^{N+2}, \,\,\, \forall N \in \{0,1,...,L-4\}, \\
       & \|w_{q+1}^{(t)} \cdot \nabla \nabla \bar v_{q, \Gamma}\|_{L-4+N} \lesssim \delta_{q+1} \frac{\delta_q^{1/2} \lambda_q}{\mu_{q+1}}  \lambda_q^{L-2} \ell_q^{- N - \alpha}, \,\,\, \forall N \geq 0,
    \end{align*}
Therefore, the estimates \eqref{mat_spa_velo_1} and \eqref{mat_spa_velo_2} are proven. Next, we move on to proving estimates \eqref{a_bar_2nd_material} and \eqref{a_bar_2nd_material_2}. To do so, let us first introduce some notations:
\begin{align*}
\mathcal{X}^1_{k,n} &:= \big( \nabla \tilde \Psi_k \nabla \tilde \Psi_k^T  - \nabla \tilde \Psi_k \frac{\bar R_{q,n}}{\delta_{q+1, n}} \nabla \tilde \Psi_k^T\big) \\
\mathcal{X}^2_{k,n} &:= - \nabla \tilde \Psi_k \Big(\sum_{\zeta \in \Lambda_T} \sum_{k' \in \mathbb Z_{q,n}} \delta_{q+1, n}^{-1}  \bar b^2_{\zeta, k', n} (\nabla \tilde \Psi_{k'} )^{-1}(\zeta \otimes \zeta) (\nabla \tilde \Psi_{k'} )^{-T}\Big) \nabla \tilde \Psi_k^T, \quad 
\mathcal{X}^3_{k,n} := \big(\lambda_q^{\alpha}\delta_{q+1, n}^{-1} \nabla \tilde \Psi_k \bar T_{q, n} \big).
\end{align*}
With the help of the above notations, we can write 
    \begin{align*}
      & \bar D_{t, \Gamma}^2 \bar a_{\xi, k, n} =  \underbrace{\delta_{q+1, n}^{1/2} \partial_t^2 \eta_k \gamma_\xi \big(\mathcal{X}^1_{k,n} + \mathcal{X}^2_{k,n} \big)}_{A_1} 
       +  \underbrace{2\,\delta_{q+1, n}^{1/2} \partial_t \eta_k D\gamma_\xi \big( \mathcal{X}^1_{k,n} + \mathcal{X}^2_{k,n} \big) \bar D_{t, \Gamma} \big(\mathcal{X}^1_{k,n} + \mathcal{X}^2_{k,n} \big)}_{A_2} \\
        & \,\,+  \underbrace{\delta_{q+1, n}^{1/2} \eta_k \bar D_{t, \Gamma}\bigg[ D\gamma_\xi \big( \mathcal{X}^1_{k,n} + \mathcal{X}^2_{k,n}\big)\bigg] \bar D_{t, \Gamma} \big( \mathcal{X}^1_{k,n} + \mathcal{X}^2_{k,n}\big)}_{A_3} 
+ \underbrace{\delta_{q+1, n}^{1/2} \eta_k D\gamma_\xi \big(\mathcal{X}^1_{k,n} + \mathcal{X}^2_{k,n} \big) \bar D_{t, \Gamma}^2 \big( \mathcal{X}^1_{k,n} + \mathcal{X}^2_{k,n} \big)}_{A_4},
    \end{align*}
    and 
     \begin{align}
       &\bar D_{t, \Gamma}^2 \bar b_{\zeta, k, n} = \underbrace{\lambda_q^{-\alpha/2}\delta_{q+1, n}^{1/2} \partial_t^2 \eta_k \Gamma^{1/2}_\zeta \big(\mathcal{X}^3_{k,n}\big)}_{B_1} 
        +  \underbrace{2 \lambda_q^{-\alpha/2} \delta_{q+1, n}^{1/2} \partial_t \eta_k D\Gamma^{1/2}_\zeta \big(\mathcal{X}^3_{k,n} \big) \bar D_{t, \Gamma} \big(\mathcal{X}^3_{k,n}\big)}_{B_2} \label{est0001}\\
        & \quad +  \underbrace{\lambda_q^{-\alpha/2} \delta_{q+1, n}^{1/2} \eta_k \bar D_{t, \Gamma}\bigg[ D\Gamma^{1/2}_\zeta \big( \mathcal{X}^3_{k,n} \big)\bigg] \bar D_{t, \Gamma} \big( \mathcal{X}^3_{k,n} \big)}_{B_3} 
+ \underbrace{\lambda_q^{-\alpha/2} \delta_{q+1, n}^{1/2} \eta_k D\Gamma^{1/2}_\zeta \big(\mathcal{X}^3_{k,n} \big) \bar D_{t, \Gamma}^2 \big( \mathcal{X}^3_{k,n}\big)}_{B_4} \nonumber.
    \end{align}
Thanks to Lemma \ref{tmoli_estim} and Corollary \ref{Flow_gam_estim}, we can estimate above terms. Indeed, arguing as in the proof of Lemma \ref{a_estim}, we have
    \begin{align*}
    &\|\gamma_\xi \big( \mathcal{X}^1_{k,n} + \mathcal{X}^2_{k,n}\big)\|_N \lesssim \lambda_q^{N}, \hspace{3.2cm}
         \|\Gamma^{1/2}_\zeta \big( \mathcal{X}^3_{k,n} \big)\|_N \lesssim \lambda_q^{N}, \,\,\, \forall N \in \{0,1,...,L-4\}, \\
    &\|\gamma_\xi \big(  \mathcal{X}^1_{k,n} + \mathcal{X}^2_{k,n}\big)\|_{N+L -4} \lesssim \lambda_q^{L - 4}\ell_q^{-N},
       \hspace{1.4cm}  \|\Gamma^{1/2}_\zeta \big(  \mathcal{X}^3_{k,n} \big)\|_{N+L -4} \lesssim \lambda_q^{L - 4}\ell_q^{-N}, \,\,\, \forall N \geq 0, \\
    &\|\bar D_{t, \Gamma} \gamma_\xi \big(  \mathcal{X}^1_{k,n} + \mathcal{X}^2_{k,n}\big)\|_N \lesssim \lambda_q^{N} \tau_q^{-1}, \hspace{1.9cm}  \|\bar D_{t, \Gamma} \Gamma^{1/2}_\zeta \big(  \mathcal{X}^3_{k,n} \big)\|_N \lesssim \lambda_q^{N} \tau_q^{-1}, \, \forall N \in \{0,1,...,L-4\}, \\
    &\|\bar D_{t, \Gamma} \gamma_\xi \big(  \mathcal{X}^1_{k,n} + \mathcal{X}^2_{k,n}\big)\|_{N+L - 4} \lesssim \lambda_q^{L - 4}\tau_q^{-1}\ell_q^{-N}, \,\, \|\bar D_{t, \Gamma} \Gamma^{1/2}_\zeta \big(  \mathcal{X}^3_{k,n} \big)\|_{N+L - 4} \lesssim \lambda_q^{L - 4}\tau_q^{-1}\ell_q^{-N}, \,\,\, \forall N \geq 0.
    \end{align*}
We can also write similar estimates when $D\gamma_\xi, D \Gamma_\xi$ replaces $\gamma_\xi$ and $\Gamma_\xi$ respectively. This would be relevant for expressions $A_2$, $A_3$, $A_4$, $B_2$, $B_3$, and $B_4$. Hence, we have 
    \begin{align*}
        \|A_1\|_N + \|A_2\|_N + \|A_3\|_N \lesssim \delta_{q+1, n}^{1/2} \tau_q^{-2} \lambda_q^N, \,\,\, \forall N \in \{0,1,...,L-4\}, \\
         \|B_1\|_N + \|B_2\|_N + \|B_3\|_N \lesssim \delta_{q+1, n}^{1/2} \tau_q^{-2} \lambda_q^N, \,\,\, \forall N \in \{0,1,...,L-4\},
    \end{align*}
    and
    \begin{align*}
        \|A_1\|_{N+L -4} + \|A_2\|_{N+L -4} + \|A_3\|_{N+L-4} \lesssim \delta_{q+1, n}^{1/2} \tau_q^{-2} \lambda_q^{L-4} \ell_q^{-N}, \,\,\, \forall N \geq 0, \\
         \|B_1\|_{N+L -4} + \|B_2\|_{N+L -4} + \|B_3\|_{N+L-4} \lesssim \delta_{q+1, n}^{1/2} \tau_q^{-2} \lambda_q^{L-4} \ell_q^{-N}, \,\,\, \forall N \geq 0.
    \end{align*}
To deal with the rest of the terms $A_4$ and $B_4$, we require two material derivative estimates of $\bar c^2_{\zeta, k, n}$ which can be derived as in \eqref{est0001}. Therefore, we have all the required estimates except the estimate of $\bar D_{t, \Gamma}^2 \nabla \tilde \Psi_k$. In fact, following \cite{vikram}, we have
    \begin{equation*}
        \|\bar D_{t, \Gamma}^2 \nabla \tilde \Psi_k\|_N \lesssim \delta_q \lambda_q^2 \lambda_q^N \lesssim \tau_q^{-2} \lambda_q^N, \,\,\, \forall N\in \{0,1,...,L-4\},
    \end{equation*}
    \begin{equation*}
        \|\bar D_{t, \Gamma}^2 \nabla \tilde \Psi_k\|_{N+L-4} \lesssim \delta_q \lambda_q^2 \lambda_q^{L-4} \ell_q^{-N} \lesssim \tau_q^{-2}  \lambda_q^{L-4} \ell_q^{-N}, \,\,\, \forall N \geq 0.
    \end{equation*}
Now we can use corollary \ref{Flow_gam_estim} and lemma \ref{tmoli_estim} to conclude
    \begin{equation*}
        \|\bar D_{t, \Gamma}^2 \big( \mathcal{X}^1_{k,n} + \mathcal{X}^2_{k,n}\big)\|_N, \, \|\bar D_{t, \Gamma}^2 \big( \mathcal{X}^3_{k,n} \big)\|_N\lesssim \tau_q^{-1} \ell_{t, q}^{-1} \lambda_q^{N}, \,\,\, \forall N\in \{0,1,...,L-4\},
    \end{equation*}
    \begin{equation*}
        \|\bar D_{t, \Gamma}^2 \big( \mathcal{X}^1_{k,n} + \mathcal{X}^2_{k,n}\big)\|_{N+L-4}, \, \|\bar D_{t, \Gamma}^2 \big( \mathcal{X}^3_{k,n} \big)\|_{N+L-4}\lesssim \tau_q^{-1} \ell_{t, q}^{-1} \lambda_q^{L-4} \ell_q^{-N}, \,\,\, \forall N\geq 0.
    \end{equation*}
This concludes the proof of the lemma.
\end{proof}

\begin{lem} \label{Nash_err_estim}
We have the following estimates: 
\begin{align*}
   & \|\mathcal{R} (w_{q+1}^{(s)} \cdot \nabla \bar v_{q, \Gamma})\|_N, \, \|\mathcal{R} (w_{q+1}^{(s)} \cdot \nabla \bar \theta_{q, \Gamma})\|_N \lesssim \frac{\delta_{q}^{1/2} \delta_{q+1}^{1/2} \lambda_q}{\lambda_{q+1}^{1-\alpha}} \lambda_{q+1}^N, \hspace{1.5cm} \forall N \geq 0, \\
       & \|\bar D_{t,\Ga}\mathcal{R} (w_{q+1}^{(s)} \cdot \nabla \bar v_{q, \Gamma})\|_N, \, \|\bar D_{t,\Ga}\mathcal{R} (w_{q+1}^{(s)} \cdot \nabla \bar \theta_{q, \Gamma})\|_N  \lesssim \frac{\mu_{q+1}\delta_{q}^{1/2} \delta_{q+1}^{1/2} \lambda_q}{\lambda_{q+1}^{1-\alpha}} \lambda_{q+1}^N, \,\,\, \forall N \geq 0, 
\end{align*}
where the implicit constants only depend on $\Gamma$, $M$, $\alpha$, and $N$. 
\end{lem}

\begin{proof}
We remark that the first and second estimates of lemma \eqref{Nash_err_estim}, related to the velocity field, can be established following the same argument as in \cite[Lemma 4.8]{vikram}, modulo cosmetic changes. In fact, other two estimates also follow using same argument. Indeed, we can write 
\begin{align*}
    \mathcal{R} (w_{q+1}^{(s)} \cdot \nabla \bar \theta_{q, \Gamma}) &= - \frac{1}{\lambda_{q+1}} \mathcal{R} \div \sum_{\xi \in \Lambda_R, k, n} g_{\xi, k, n+1} \bar a_{\xi, k, n} \Phi_\xi (\lambda_{q+1}\tilde \Psi_k) \nabla^\perp \bar \theta_{q, \Gamma} \\
   & \qquad - \frac{1}{\lambda_{q+1}} \mathcal{R} \div \sum_{\zeta \in \Lambda_T, k, n} h_{\zeta, k, n+1} \bar b_{\zeta, k, n} \Phi_\zeta (\lambda_{q+1}\tilde \Psi_k) \nabla^\perp \bar \theta_{q, \Gamma}.
\end{align*}
Using disjointness of temporal supports of the terms in the sum above, and $\mathcal{R} \div$ being a sum of operators of Calder\'on-Zygmund type, we can follow the arguments as in \cite[Lemma 4.8]{vikram} to conclude the proof of the lemma.
\end{proof}

Next, we move on to the estimation of the transport terms. To proceed, first note the important remark which will be used in the following lemma:
\begin{equation*}
\mathcal{R} (g (t, \cdot)+ h(t)) = \mathcal{R} (g (t, \cdot)),
\end{equation*}
for every smooth periodic time-dependent vector field $g$ and for every $h$ which depends only on time. Moreover, a straightforward adaptation of \cite[Appendix F, G and H]{BDLIS15} reveals that Proposition~\ref{prop.inv.div} and Proposition~\ref{prop.comm} are also true when $\mathcal{R}$ acts on a scalar function. 
\begin{lem} \label{Tranp_err_estim}
The following estimates hold
\begin{align}
    & \|\mathcal{R}(\bar D_{t, \Gamma} w^{(s)}_{q+1})\|_N, \, \|\mathcal{R}(\bar D_{t, \Gamma} \theta^{(s)}_{q+1})\|_N \lesssim \frac{\delta_{q+1} \lambda_q^{2/3}}{\lambda_{q+1}^{2/3 - 5\alpha}} \lambda_{q+1}^N, \hspace{1.9cm}\,\, \forall N \geq 0, \label{11} \\
    & \|\bar D_{t,\Ga} \mathcal{R}(\bar D_{t, \Gamma} w^{(s)}_{q+1})\|_N, \, \|\bar D_{t,\Ga} \mathcal{R}(\bar D_{t, \Gamma} \theta^{(s)}_{q+1})\|_N \lesssim \frac{\ell_{t,q}^{-1}\delta_{q+1} \lambda_q^{2/3}}{\lambda_{q+1}^{2/3 - 5\alpha}} \lambda_{q+1}^N, \,\,\, \forall N \geq 0, \label{12}
\end{align}
where the implicit constants only depend on $\Gamma$, $M$, $\alpha$, and $N$.
\end{lem}

\begin{proof}
The proof of \eqref{11} and \eqref{12}, related to velocity perturbation $w^{(s)}_{q+1}$, are very similar to \cite[Lemma 4.9]{vikram}. Indeed, we can write    
\begin{align*}
        \mathcal{R}(\bar D_{t, \Gamma} w_{q+1}^{(s)}) =& \frac{1}{\lambda_{q+1}} \mathcal{R}\nabla^\perp\bigg( \bar D_{t, \Gamma}\sum_{\xi \in \Lambda_R, k, n} g_{\xi, k, n+1}\bar a_{\xi, k, n} \Phi_\xi(\lambda_{q+1} \tilde \Psi_k)\bigg)\\
        & - \frac{1}{\lambda_{q+1}} \mathcal{R}\div\bigg(\sum_{\xi\in \Lambda_R, k, n} g_{\xi, k, n+1} \bar a_{\xi, k, n} \Phi_\xi(\lambda_{q+1}\tilde \Psi_k) \nabla^\perp \bar v_{q, \Gamma}\bigg)\\
       & + \frac{1}{\lambda_{q+1}} \mathcal{R}\nabla^\perp\bigg( \bar D_{t, \Gamma}\sum_{\zeta \in \Lambda_T, k, n} h_{\zeta, k, n+1} \bar b_{\zeta, k, n} \Phi_\zeta(\lambda_{q+1} \tilde \Psi_k)\bigg) \\
        & - \frac{1}{\lambda_{q+1}} \mathcal{R}\div\bigg(\sum_{\zeta \in \Lambda_T, k, n} h_{\zeta, k, n+1} \bar b_{\zeta, k, n} \Phi_\zeta(\lambda_{q+1}\tilde \Psi_k) \nabla^\perp \bar v_{q, \Gamma}\bigg).
    \end{align*}
We can deal with each of the term as in \cite[Lemma 4.9]{vikram}. We leave the details to the interested reader.    
To derive the estimates \eqref{11} and \eqref{12}, related to the temperature perturbation, first notice that, we can write 
\begin{align*}
\theta_{q+1}^{(p)} &= \sum_{n = 0}^{\Gamma - 1} \sum_{k \in \mathbb Z_{q,n}} \sum_{\zeta\in\Lambda_T} \underbrace{h_{\zeta,k,n+1} \bar b_{\zeta,k,n}}_{:=\mathcal{E}_{\zeta,k,n+1}} \mathbb{V}_\zeta (\la_{q+1} \tilde\Psi_k) \\
&= \sum_{n = 0}^{\Gamma - 1} \sum_{k \in \mathbb Z_{q,n}} \sum_{\zeta\in\Lambda_T} \frac{1}{\sqrt{2}}\mathcal{E}_{\zeta,k,n+1} 
\big ( e^{i \lambda_{q+1}\zeta^\perp \cdot \tilde\Psi_k} + e^{- i \la_{q+1}\zeta^\perp \cdot \tilde\Psi_k} \big).
\end{align*}
Keeping in mind that $\mathbb{V}_\zeta (\la_{q+1} \tilde\Psi_k)$ is carried by the flow of $\bar v_{q, \Gamma}$, we can calculate
\begin{align*}
\bar  D_{t, \Gamma} \theta_{q+1}^{(p)} &= \sum_{n = 0}^{\Gamma - 1} \sum_{k \in \mathbb Z_{q,n}} \sum_{\zeta\in\Lambda_T} \frac{1}{\sqrt{2}} \big(\bar  D_{t, \Gamma} \mathcal{E}_{\zeta,k,n+1} \big)
\big ( e^{i \lambda_{q+1}\zeta^\perp \cdot \tilde\Psi_k} + e^{- i \la_{q+1}\zeta^\perp \cdot \tilde\Psi_k} \big) \\
\bar  D^2_{t, \Gamma} \theta_{q+1}^{(p)} &= \sum_{n = 0}^{\Gamma - 1} \sum_{k \in \mathbb Z_{q,n}} \sum_{\zeta\in\Lambda_T} \frac{1}{\sqrt{2}} \big(\bar  D^2_{t, \Gamma} \mathcal{E}_{\zeta,k,n+1} \big)
\big ( e^{i \lambda_{q+1}\zeta^\perp \cdot \tilde\Psi_k} + e^{- i \la_{q+1}\zeta^\perp \cdot \tilde\Psi_k} \big)
\end{align*}
Next, we already know that 
 \begin{align}
        \|\bar D_{t, \Gamma}(\mathcal{E}_{\zeta,k,n+1}) \|_{N+\alpha} &\lesssim \delta_{q+1}^{1/2} \mu_{q+1} \ell_q^{-N-\alpha} \lesssim \delta_{q+1} \lambda_q^{2/3}\lambda_{q+1}^{1/3 + 4\alpha} \ell_q^{-N-\alpha} \label{new_03}\\
        \|\bar D^2_{t, \Gamma}(\mathcal{E}_{\zeta,k,n+1}) \|_{N+\alpha} &\lesssim \ell_{t,q}^{-1}\mu_{q+1} \delta_{q+1}^{1/2} \ell_q^{-N} \lesssim \ell_{t,q}^{-1} \delta_{q+1} \lambda_q^{2/3}\lambda_{q+1}^{1/3 + 4\alpha} \ell_q^{-N} \label{new_04}
    \end{align}
Therefore, we can apply Proposition~\ref{prop.inv.div} to conclude
\begin{align*}
 \left\|\mathcal{R}\left(\bar D_{t, \Gamma} \theta^{(s)}_{q+1}\right)\right\|_\alpha &\lesssim \frac{\|\bar  D_{t, \Gamma} \mathcal{E}_{\zeta,k,n+1}\|_0}{\la_{q+1}^{1-\alpha}} + \frac{\|\bar  D_{t, \Gamma} \mathcal{E}_{\zeta,k,n+1}\|_{A+\alpha} + \|\bar  D_{t, \Gamma} \mathcal{E}_{\zeta,k,n+1}\|_0\|\nabla \tilde\Psi_k\|_{A+\alpha}}{\la_{q+1}^{A-\alpha}} \\
 & \lesssim \frac{\delta_{q+1} \lambda_q^{2/3}\lambda_{q+1}^{1/3 + 4\alpha}}{\la_{q+1}^{1-\alpha}}
 = \frac{\delta_{q+1} \lambda_q^{2/3}}{\lambda_{q+1}^{2/3 - 5\alpha}},
\end{align*}
where to obtain the last inequality we choose $A \in \N$ sufficiently large, in particular $A > \alpha + \frac{2b}{b-1}$. To obtain the estimates for the higher derivatives, consider $N \geq 1$ and let $\theta$ be a multi-index with $|\theta| = N-1$, and $i \in \{1,2\}$. Then, leveraging the fact that $\partial_i \mathcal{R}$ can be expressed as a sum of Calder\'on-Zygmund operators, we proceed to estimate each term involved. 
\begin{eqnarray*}
    \big\|\partial_i \partial^\theta \mathcal{R}\left(\bar D_{t, \Gamma} \theta^{(s)}_{q+1}\right)\big\|_\alpha &\lesssim& \| \left(\bar D_{t, \Gamma} \theta^{(s)}_{q+1}\right)\|_{N-1+\alpha} \\ 
    & \lesssim & \|\bar D_{t, \Gamma}(\mathcal{E}_{\zeta,k,n+1})\|_{N-1+\alpha} + \lambda_{q+1}^{N-1+\alpha}\|\bar D_{t, \Gamma}(\mathcal{E}_{\zeta,k,n+1})\|_0  \\
    & \lesssim & \delta_{q+1} \lambda_q^{2/3}\lambda_{q+1}^{1/3 + 4\alpha} (\lambda_{q+1}^{N-1+\alpha} + \ell_q^{-N + 1 - \alpha}) \lesssim \frac{ \delta_{q+1} \lambda_q^{2/3}}{\lambda_{q+1}^{2/3 - 5\alpha}} \lambda_{q+1}^N,
\end{eqnarray*}
and this gives the estimate \eqref{11}, related to the temperature perturbation.

Next, to derive the estimate \eqref{12}, we write 
$$
\bar D_{t,\Ga} \mathcal{R}(\bar D_{t, \Gamma} \theta^{(s)}_{q+1}) = \mathcal{R}(\bar D^2_{t, \Gamma} \theta^{(s)}_{q+1})
+ [\bar D_{t,\Ga} , \mathcal{R}] (\bar D_{t, \Gamma} \theta^{(s)}_{q+1}) =  \mathcal{R}(\bar D^2_{t, \Gamma} \theta^{(s)}_{q+1})
+ [\bar v_{q, \Gamma} \cdot, \mathcal{R}] (\na \bar D_{t, \Gamma} \theta^{(s)}_{q+1}).
$$
To deal with the first term in the above expression, we again use Proposition~\ref{prop.inv.div}, and choose $A > \alpha + \frac{2b}{b-1}$, to conclude
\begin{align*}
 \left\|\mathcal{R}\left(\bar D^2_{t, \Gamma} \theta^{(s)}_{q+1}\right)\right\|_\alpha & \lesssim \frac{\|\bar  D^2_{t, \Gamma} \mathcal{E}_{\zeta,k,n+1}\|_0}{\la_{q+1}^{1-\alpha}} + \frac{\|\bar  D^2_{t, \Gamma} \mathcal{E}_{\zeta,k,n+1}\|_{A+\alpha} + \|\bar  D^2_{t, \Gamma} \mathcal{E}_{\zeta,k,n+1}\|_0\|\nabla \tilde\Psi_k\|_{A+\alpha}}{\la_{q+1}^{A-\alpha}} \\
 & \lesssim \frac{\ell_{t,q}^{-1} \delta_{q+1} \lambda_q^{2/3}\lambda_{q+1}^{1/3 + 4\alpha}}{\la_{q+1}^{1-\alpha}}
 = \frac{\ell_{t,q}^{-1} \delta_{q+1} \lambda_q^{2/3}}{\lambda_{q+1}^{2/3 - 5\alpha}} =  \delta_{q+1} \delta_q^{1/2} \lambda_q \lambda_{q+1}^{ 5\alpha},
 \end{align*}
 and we also have
\begin{align*}
\big\|\partial_i \partial^\theta \mathcal{R}\left(\bar D^2_{t, \Gamma} \theta^{(s)}_{q+1}\right)\big\|_\alpha \lesssim
\frac{\ell_{t,q}^{-1} \delta_{q+1} \lambda_q^{2/3}}{\lambda_{q+1}^{2/3 - 5\alpha}} \lambda_{q+1}^N =  \delta_{q+1} \delta_q^{1/2} \lambda_q \lambda_{q+1}^{ N+ 5\alpha}.
\end{align*}
Finally, to deal with the commutator term, we make use of the Proposition~\ref{prop.comm}. However, in order to apply the Proposition~\ref{prop.comm}, we first introduce the ``phase'' $\phi_{\zeta, k} (x,t):= e^{i \lambda_{q+1} \zeta^\perp \cdot [\tilde\Psi_k(x,t) -x]}$ so that
\begin{align*}
\phi_{\zeta, k} (x,t) e^{i \lambda_{q+1} \zeta^\perp \cdot x} = e^{i \lambda_{q+1} \zeta^\perp \cdot \tilde\Psi_k(x,t)}
\end{align*}
Moreover, we also have the estimates, for any $N\ge1$
\begin{align}\label{new_02}
\| \phi_{\zeta, k} \|_N \lesssim \lambda_{q+1} \| D \tilde\Psi_k \|_{N-1} + \lambda^N_{q+1} \| D \tilde\Psi_k - \mathrm{Id}\|_0^{N}.
\end{align}
Let us now denote by 
$$
\Omega_{\zeta,k}:= \sum_{n = 0}^{\Gamma - 1} \sum_{k \in \mathbb Z_{q,n}} \sum_{\zeta\in\Lambda_T} \phi_{\zeta, k} \bar  D_{t, \Gamma} \mathcal{E}_{\zeta,k,n+1},
$$
and write
\begin{align*}
[\bar v_{q, \Gamma} \cdot, \mathcal{R}] (\na \bar D_{t, \Gamma} \theta^{(s)}_{q+1}) = 
[\bar v_{q, \Gamma} \cdot, \mathcal{R}] \na \Omega_{\zeta,k}  e^{i \lambda_{q+1} \zeta^\perp \cdot x} 
+ i \lambda_{q+1}[\bar v_{q, \Gamma} \cdot \zeta^\perp, \mathcal{R}] \Omega_{\zeta,k}  e^{i \lambda_{q+1} \zeta^\perp \cdot x} 
\end{align*}
Now we can apply the Proposition~\ref{prop.comm} to both terms on the right hand side of the above expression to conclude
\begin{align*}
& \left\|[\bar v_{q, \Gamma} \cdot, \mathcal{R}] (\na \bar D_{t, \Gamma} \theta^{(s)}_{q+1})\right\|_\alpha \\
& \quad \lesssim \frac{ \| \bar v_{q, \Gamma} \|_1 \| \phi_{\zeta, k} \bar  D_{t, \Gamma} \mathcal{E}_{\zeta,k,n+1}\|_1}{\la_{q+1}^{2-\alpha}} + \frac{C}{\la_{q+1}^{A-\alpha}} \sum_{i=0}^{A-1} \|\phi_{\zeta, k} \bar  D_{t, \Gamma} \mathcal{E}_{\zeta,k,n+1}\|_{1 +i+\alpha} \|\bar v_{q, \Gamma}\|_{A-i+\alpha}\\
 & \qquad + \frac{ \la_{q+1}\| \bar v_{q, \Gamma} \|_1 \|\phi_{\zeta, k} \bar  D_{t, \Gamma} \mathcal{E}_{\zeta,k,n+1}\|_0}{\la_{q+1}^{2-\alpha}} 
+ \frac{C \la_{q+1}}{\la_{q+1}^{A-\alpha}} \sum_{i=0}^{A-1} \|\phi_{\zeta, k} \bar  D_{t, \Gamma} \mathcal{E}_{\zeta,k,n+1}\|_{i+\alpha} \|\bar v_{q, \Gamma}\|_{A-i+\alpha}.
\end{align*}
We can make use of the bounds given in Corollary~\ref{Gamma_velo_estim}, \eqref{new_03}, \eqref{new_02}, and choose $A >\frac{(4b-1)}{3(b-1)}$, to conclude
\begin{align*}
& \left\|[\bar v_{q, \Gamma} \cdot, \mathcal{R}] (\na \bar D_{t, \Gamma} \theta^{(s)}_{q+1})\right\|_\alpha 
\lesssim \frac{\delta_{q+1} \delta_q^{1/2} \lambda_q^{13/6}}{\lambda_{q+1}^{7/6 - 5\alpha}} 
+ \frac{\delta_{q+1} \delta_q^{1/2} \lambda_q^{5/3}}{\lambda_{q+1}^{2/3 - 5\alpha}} \lesssim \delta_{q+1} \delta_q^{1/2} \lambda_q \lambda_{q+1}^{ 5\alpha}.
\end{align*}
In order to estimate higher derivatives, let $N \geq 1$ be a fixed integer, and consider a multi-index $\theta$ such that the total order of the derivatives, $|\theta|$, is equal to $N$. Then for any smooth functions $b$ and $F$, we have
\begin{align*} 
\partial^{\theta} [b, \mathcal{R}] (F) =  [\partial^{\theta} b, \mathcal{R}] (F)
+  [b, \mathcal{R}] (\partial^{\theta} F).
\end{align*}
Therefore, we apply Proposition~\ref{prop.comm} to conclude
\begin{align*}
& \left\|[\bar v_{q, \Gamma} \cdot, \mathcal{R}] (\na \bar D_{t, \Gamma} \theta^{(s)}_{q+1})\right\|_{N+\alpha} \\
& \quad \lesssim \frac{ \| \bar v_{q, \Gamma} \|_{N+1} \| \phi_{\zeta, k} \bar  D_{t, \Gamma} \mathcal{E}_{\zeta,k,n+1}\|_1}{\la_{q+1}^{2-\alpha}} + \frac{C}{\la_{q+1}^{A-\alpha}} \sum_{i=0}^{A-1} \|\phi_{\zeta, k} \bar  D_{t, \Gamma} \mathcal{E}_{\zeta,k,n+1}\|_{1 +i+\alpha} \|\bar v_{q, \Gamma}\|_{N+A-i+\alpha}\\
& \qquad + \frac{ \| \bar v_{q, \Gamma} \|_1 \|\phi_{\zeta, k} \bar  D_{t, \Gamma} \mathcal{E}_{\zeta,k,n+1}\|_{N+1}}{\la_{q+1}^{2-\alpha}} 
+ \frac{C }{\la_{q+1}^{A-\alpha}} \sum_{i=0}^{A-1} \|\phi_{\zeta, k} \bar  D_{t, \Gamma} \mathcal{E}_{\zeta,k,n+1}\|_{N+i+\alpha} \|\bar v_{q, \Gamma}\|_{A-i+\alpha} \\
 & \qquad + \frac{ \la^{N+1}_{q+1}\| \bar v_{q, \Gamma} \|_1 \|\phi_{\zeta, k} \bar  D_{t, \Gamma} \mathcal{E}_{\zeta,k,n+1}\|_0}{\la_{q+1}^{2-\alpha}} 
+ \frac{C \la^{N+1}_{q+1}}{\la_{q+1}^{A-\alpha}} \sum_{i=0}^{A-1} \|\phi_{\zeta, k} \bar  D_{t, \Gamma} \mathcal{E}_{\zeta,k,n+1}\|_{i+\alpha} \|\bar v_{q, \Gamma}\|_{A-i+\alpha} \\
& \quad + \frac{\lambda_{q+1} \| \bar v_{q, \Gamma} \|_{N+1} \| \phi_{\zeta, k} \bar  D_{t, \Gamma} \mathcal{E}_{\zeta,k,n+1}\|_0}{\la_{q+1}^{2-\alpha}} + \frac{C\lambda_{q+1}}{\la_{q+1}^{A-\alpha}} \sum_{i=0}^{A-1} \|\phi_{\zeta, k} \bar  D_{t, \Gamma} \mathcal{E}_{\zeta,k,n+1}\|_{i+\alpha} \|\bar v_{q, \Gamma}\|_{N+A-i+\alpha}\\
& \qquad + \frac{ \lambda_{q+1}\| \bar v_{q, \Gamma} \|_1 \|\phi_{\zeta, k} \bar  D_{t, \Gamma} \mathcal{E}_{\zeta,k,n+1}\|_{N}}{\la_{q+1}^{2-\alpha}} 
+ \frac{C \la_{q+1}}{\la_{q+1}^{A-\alpha}} \sum_{i=0}^{A-1} \|\phi_{\zeta, k} \bar  D_{t, \Gamma} \mathcal{E}_{\zeta,k,n+1}\|_{N-1+i+\alpha} \|\bar v_{q, \Gamma}\|_{A-i+\alpha} \\
 & \quad + \frac{\la^{N+1}_{q+1}\| \bar v_{q, \Gamma} \|_1 \|\phi_{\zeta, k} \bar  D_{t, \Gamma} \mathcal{E}_{\zeta,k,n+1}\|_0}{\la_{q+1}^{2-\alpha}} 
+ \frac{C \la^{N+1}_{q+1}}{\la_{q+1}^{A-\alpha}} \sum_{i=0}^{A-1} \|\phi_{\zeta, k} \bar  D_{t, \Gamma} \mathcal{E}_{\zeta,k,n+1}\|_{i+\alpha} \|\bar v_{q, \Gamma}\|_{A-i+\alpha} \\
& \quad \lesssim \delta_{q+1} \delta_q^{1/2} \lambda_q \lambda_{q+1}^{ N+ 5\alpha}
\end{align*}
Here we have again used the bounds given in Corollary~\ref{Gamma_velo_estim}, \eqref{new_03}, \eqref{new_02}, and $A >\frac{(4b-1)}{3(b-1)}$. This finishes the proof of the lemma.
\end{proof}

Finally, we consider the following linear error term. To that context, we have the following lemma.
\begin{lem} \label{final}
The following estimates hold
\begin{align}
 \|\mathcal{R}(\theta^{(s)}_{q+1} e_2)\|_N \lesssim \frac{\delta^{1/2}_{q+1}}{\lambda_{q+1}^{1 - \alpha}} \lambda_{q+1}^N, \quad \|\bar D_{t,\Ga} \mathcal{R}(\theta^{(s)}_{q+1} e_2)\|_N \lesssim \frac{\ell_{t,q}^{-1}\delta^{1/2}_{q+1} }{\lambda_{q+1}^{1 - 5\alpha}} \lambda_{q+1}^N, \,\,\, \forall N \geq 0, \label{122}
\end{align}
where all the implicit bounds are dependent on $\Gamma$, $M$, $\alpha$, and $N$.
\end{lem}

\begin{proof}
Note that it is enough to get the estimates for $\theta^{(p)}_{q+1}$, since for every smooth periodic time-dependent vector field $g$ and for every $h$ which depends only on time, we have $\mathcal{R} (g (t, \cdot)+ h(t)) = \mathcal{R} (g (t, \cdot))$.
Therefore, in order to establish the estimates \eqref{122}, we first recall that 
\begin{align*}
\theta_{q+1}^{(p)} e_2&= \sum_{n = 0}^{\Gamma - 1} \sum_{k \in \mathbb Z_{q,n}} \sum_{\zeta\in\Lambda_T} \underbrace{h_{\zeta,k,n+1} \bar b_{\zeta,k,n}}_{:=\mathcal{E}_{\zeta,k,n+1}} \mathbb{V}_\zeta (\la_{q+1} \tilde\Psi_k) e_2\\
&= \sum_{n = 0}^{\Gamma - 1} \sum_{k \in \mathbb Z_{q,n}} \sum_{\zeta\in\Lambda_T} \frac{1}{\sqrt{2}}\mathcal{E}_{\zeta,k,n+1} 
\big ( e^{i \lambda_{q+1}\zeta^\perp \cdot \tilde\Psi_k} + e^{- i \la_{q+1}\zeta^\perp \cdot \tilde\Psi_k} \big) e_2.
\end{align*}
Indeed, we can follow the proof of the Lemma~\ref{Tranp_err_estim} to get the desired estimates. Indeed, for the first estimate, we only need informations on the perturbations $\theta^{(s)}_{q+1}$, unlike in the previous case where we needed material derivative informations. For the second estimate also we can use similar arguments as in Lemma~\ref{Tranp_err_estim}. For the sake of brevity, we skip the details. 
\end{proof}

\subsection{Oscillation error \texorpdfstring{$R_{q+1, O}$}{rororo}}
Let us first write
\begin{equation*}
    R_{q+1, O} = \underbrace{\mathcal{R}\div(S_{q, \Gamma} + w_{q+1}^{(p)} \otimes w_{q+1}^{(p)})}_{\text{Principal oscillation error}} + \underbrace{\mathcal{R} \div (w_{q+1}^{(p)} \otimes w_{q+1}^{(c)} + w_{q+1}^{(c)}\otimes w_{q+1}^{(p)} + w_{q+1}^{(c)} \otimes w_{q+1}^{(c)})}_{\text{Divergence corrector error}}.
\end{equation*}
To deal with the first term, we closely follow \cite{vikram}. Indeed, we first recall the notation 
\begin{equation*}
    A_{\sigma, k, n} = a^2_{\xi, k, n} (\nabla  \Psi_k)^{-1} \xi \otimes \xi (\nabla \Psi_k)^{-T} + b^2_{\zeta, k, n} (\nabla  \Psi_k)^{-1} \zeta \otimes \zeta (\nabla \Psi_k)^{-T} := A^1_{\xi, k, n} + A^2_{\zeta, k, n},
\end{equation*}
and analogously denote 
\begin{equation*}
    \bar A_{\sigma, k, n} = \bar a^2_{\xi, k, n} (\nabla \tilde \Psi_k)^{-1} \xi \otimes \xi (\nabla \tilde \Psi_k)^{-T} + \bar b^2_{\zeta, k, n} (\nabla \tilde \Psi_k)^{-1} \zeta \otimes \zeta (\nabla \tilde \Psi_k)^{-T}
    := \bar A^1_{\xi, k, n} + \bar A^2_{\zeta, k, n}.
\end{equation*}
Then we compute
\begin{align*}
& w_{q+1}^{(p)} \otimes w_{q+1}^{(p)} 
   =  \sum_{\xi \in \Lambda_R, k, n} g_{\xi, k, n+1}^2 \bar a_{\xi, k, n}^2 (\nabla \tilde \Psi_k)^{-1} (\mathbb W_{\xi} \otimes \mathbb W_{\xi})(\lambda_{q+1} \tilde \Psi_k) (\nabla \tilde \Psi_k)^{-T} \\
     &+  \sum_{\zeta \in \Lambda_T, k, n} h_{\zeta, k, n+1}^2 \bar b_{\zeta, k, n}^2 (\nabla \tilde \Psi_k)^{-1} (\mathbb W_{\zeta} \otimes \mathbb W_{\zeta})(\lambda_{q+1} \tilde \Psi_k) (\nabla \tilde \Psi_k)^{-T} \\ 
    & =  \sum_{\xi \in \Lambda_R, k, n} g_{\xi, k, n+1}^2 \bar a_{\xi, k, n}^2 (\nabla \tilde \Psi_k)^{-1} \xi \otimes \xi (\nabla \tilde \Psi_k)^{-T} 
  + \sum_{\zeta \in \Lambda_T, k, n} h_{\zeta, k, n+1}^2 \bar b_{\zeta, k, n}^2 (\nabla \tilde \Psi_k)^{-1} \zeta \otimes \zeta (\nabla \tilde \Psi_k)^{-T} \\
    & + \sum_{\xi \in \Lambda_R, k, n} g^2_{\xi, k, n+1} \bar a_{\xi, k, n}^2 (\nabla \tilde \Psi_k)^{-1} (\mathbb{P}_{\neq 0} \mathbb{W}_\xi \otimes \mathbb{W}_\xi)(\la_{q+1} \tilde\Psi_k) (\nabla \tilde \Psi_k)^{-T} \\ 
    & + \sum_{\zeta \in \Lambda_T, k, n} h^2_{\zeta, k, n+1} \bar b_{\zeta, k, n}^2 (\nabla \tilde \Psi_k)^{-1} (\mathbb{P}_{\neq 0} \mathbb{W}_\zeta \otimes \mathbb{W}_\zeta)(\la_{q+1} \tilde\Psi_k) (\nabla \tilde \Psi_k)^{-T} \\ 
    &=  -S_{q, \Gamma}  + \sum_{\xi \in \Lambda_R, k, n} g_{\xi, k, n+1}^2 (\bar A^1_{\xi,k,n} - A^1_{\xi, k, n}) + \sum_{\zeta \in \Lambda_T, k, n} h_{\zeta, k, n+1}^2 (\bar A^2_{\zeta,k,n} - A^2_{\zeta, k, n}) \\ 
    & + \sum_{\xi \in \Lambda_R, k, n} g^2_{\xi, k, n+1} \bar a_{\xi, k, n}^2 (\nabla \tilde \Psi_k)^{-1} (\mathbb{P}_{\neq 0} \mathbb{W}_\xi \otimes \mathbb{W}_\xi)(\la_{q+1} \tilde\Psi_k) (\nabla \tilde \Psi_k)^{-T} \\
    & + \sum_{\zeta \in \Lambda_T, k, n} h^2_{\zeta, k, n+1} \bar b_{\zeta, k, n}^2 (\nabla \tilde \Psi_k)^{-1} (\mathbb{P}_{\neq 0} \mathbb{W}_\zeta \otimes \mathbb{W}_\zeta)(\la_{q+1} \tilde\Psi_k) (\nabla \tilde \Psi_k)^{-T},
\end{align*}
where $\mathbb P_{\neq 0}$ represents the Fourier projection operator that extracts the mean-zero components of 
$2$-tensors. We can also write 
\begin{equation*}
    \mathbb P_{\neq 0} \mathbb W_\sigma \otimes \mathbb W_\sigma = \Omega_\sigma \sigma \otimes \sigma,
\end{equation*}
where $\Omega_\sigma$ is defined by 
\begin{equation*}
    \Omega_\sigma(x) = \frac{1}{2}(e^{2i\sigma^\perp \cdot x} + e^{-2i\sigma^\perp \cdot x}).
\end{equation*}
This enables us to decompose the oscillation error into distinct components as:
\begin{eqnarray*}
    \mathcal{R}\div(S_{q, \Gamma} + w_{q+1}^{(p)} \otimes w_{q+1}^{(p)}) &=& \mathcal{R} \div \bigg(\sum_{\xi \in \Lambda_R, k, n} g^2_{\xi, k, n+1} \Omega_\xi(\lambda_{q+1} \tilde \Psi_k) \bar A^1_{\xi, k, n}\bigg) \\
    && \hspace{-5cm} +\mathcal{R} \div \bigg(\sum_{\zeta \in \Lambda_T, k, n} h^2_{\zeta, k, n+1} \Omega_\zeta(\lambda_{q+1} \tilde \Psi_k) \bar A^2_{\zeta, k, n}\bigg) 
    + \mathcal{R}\div \bigg(\sum_{\xi \in \Lambda_R, k, n} g_{\xi, k, n+1}^2 ( \bar A^1_{\xi,k,n} - A^1_{\xi, k, n})\bigg)\\
    && + \mathcal{R}\div \bigg(\sum_{\zeta \in \Lambda_T, k, n} h_{\zeta, k, n+1}^2 ( \bar A^2_{\zeta,k,n} - A^2_{\zeta, k, n})\bigg).
\end{eqnarray*}
To get estimates for the above terms, we need the following estimates. The proof of the following lemma is similar to \cite[Lemma 4.10]{vikram}. We skip the details.

\begin{lem} \label{big_a_bar_estim}
The following bounds are valid: 
\begin{align*}
 &\|\bar A^1_{\xi, k, n}\|_N, \, \|\bar A^2_{\zeta, k, n}\|_N \lesssim \delta_{q+1, n} \lambda_q^N, \hspace{5cm}\,\, \forall N \in \{0,1,..., L-3\}, \\
 &\|\bar D_{t, \Gamma} \bar A^1_{\xi, k, n} \|_N, \, \|\bar D_{t, \Gamma} \bar A^2_{\zeta, k, n} \|_N  \lesssim \delta_{q+1, n} \tau_q^{-1} \lambda_q^N, \hspace{3.2cm} \forall N \in \{0,1,..., L-3\},\\
&\|\bar A^1_{\xi, k, n}\|_{N+L-3}, \, \|\bar A^2_{\zeta, k, n}\|_{N+L-3} \lesssim \delta_{q+1, n} \lambda_q^{L-3} \ell_q^{-N},  \hspace{2.6cm} \,\forall N \geq 0,\\
&\|\bar D_{t, \Gamma} \bar A^1_{\xi, k, n} \|_{N+L-3}, \, \|\bar D_{t, \Gamma} \bar A^2_{\zeta, k, n} \|_{N+L-3} \lesssim \delta_{q+1, n} \lambda_q^{L-3} \tau_q^{-1} \ell_q^{-N}, \hspace{0.7cm} \,\forall N \geq 0,
\end{align*}
where all the implicit bounds are dependent on $\Gamma$, $M$, $\alpha$ and $N$. Moreover, it holds that 
\begin{equation*}
    \|\bar A^1_{\xi, k, n} - A^1_{\xi,k, n}\|_0, \, \|\bar A^2_{\zeta, k, n} - A^2_{\zeta,k, n}\|_0 \lesssim  \delta_{q+1,n}\frac{\delta_{q+1}^{1/2} \lambda_q^{1/3}}{\delta_{q}^{1/2} \lambda_{q+1}^{1/3}}.
\end{equation*}
\end{lem}
We can now state the estimates in the following two lemmas. Again, we skip the proofs since they are very similar to \cite[Lemma 4.11]{vikram} and \cite[Lemma 4.12]{vikram}.
\begin{lem}
The following bounds are valid:
\begin{equation} \label{high-high-estim}
    \bigg\|\mathcal{R} \div \bigg(\sum_{\xi \in \Lambda_R, k, n} g^2_{\xi, k, n+1} \Omega_\xi(\lambda_{q+1} \tilde \Psi_k) \bar A^1_{\xi, k, n}\bigg)\bigg\|_{N} \lesssim \frac{\delta_{q+1} \lambda_q}{\lambda_{q+1}^{1 - 2\alpha}} \lambda_{q+1}^N, \qquad \forall N \geq 0,
\end{equation}
\begin{equation} \label{high-high-estim_01}
    \bigg\|\mathcal{R} \div \bigg(\sum_{\zeta \in \Lambda_T, k, n} h^2_{\zeta, k, n+1} \Omega_\zeta(\lambda_{q+1} \tilde \Psi_k) \bar A^2_{\zeta, k, n}\bigg)\bigg\|_{N} \lesssim \frac{\delta_{q+1} \lambda_q}{\lambda_{q+1}^{1 - 2\alpha}} \lambda_{q+1}^N, \qquad \,\forall N \geq 0,
\end{equation}
\begin{equation}
    \bigg\|\bar D_{t,\Ga} \mathcal{R} \div \bigg(\sum_{\xi \in \Lambda_R, k, n} g^2_{\xi, k, n+1} \Omega_\xi(\lambda_{q+1} \tilde \Psi_k) \bar A^1_{\xi, k, n}\bigg)\bigg\|_{N} \lesssim \mu_{q+1}\delta_{q+1} \lambda_{q+1}^\alpha \lambda_{q+1}^N, \,\,\, \forall N \geq 0.
\end{equation}
\begin{equation}
    \bigg\|\bar D_{t,\Ga} \mathcal{R} \div \bigg(\sum_{\zeta \in \Lambda_T, k, n} h^2_{\zeta, k, n+1} \Omega_\zeta(\lambda_{q+1} \tilde \Psi_k) \bar A^2_{\zeta, k, n}\bigg)\bigg\|_{N} \lesssim \mu_{q+1}\delta_{q+1} \lambda_{q+1}^\alpha \lambda_{q+1}^N, \,\,\, \forall N \geq 0.
\end{equation}
where all the implicit bounds are dependent on $\Gamma$, $M$, $\alpha$, $b$, and $N$. 
\end{lem}

\begin{lem}
The following bounds are valid:
\begin{equation}
    \bigg\|\mathcal{R} \div \bigg(\sum_{\xi \in \Lambda_R, k, n} g_{\xi, k, n+1}^2 (\bar A^1_{\xi, k, n} - A^1_{\xi, k, n})\bigg)\bigg\|_N \lesssim \delta_{q+1}\frac{\delta_{q+1}^{1/2} \lambda_{q}^{1/3}}{\delta_{q}^{1/2}\lambda_{q+1}^{1/3}} \lambda_{q+1}^\alpha \lambda_{q+1}^N, \,\,\, \forall N \geq 0.
\end{equation}
\begin{equation}
    \bigg\|\mathcal{R} \div \bigg(\sum_{\zeta \in \Lambda_T, k, n} h_{\zeta, k, n+1}^2 (\bar A^2_{\zeta, k, n} - A^2_{\zeta, k, n})\bigg)\bigg\|_N \lesssim \delta_{q+1}\frac{\delta_{q+1}^{1/2} \lambda_{q}^{1/3}}{\delta_{q}^{1/2}\lambda_{q+1}^{1/3}} \lambda_{q+1}^\alpha \lambda_{q+1}^N, \,\,\, \forall N \geq 0.
\end{equation}
\begin{equation}
    \bigg\|\bar D_{t, \Gamma}\mathcal{R} \div \bigg(\sum_{\xi \in \Lambda_R, k, n} g_{\xi, k, n+1}^2 (\bar A^1_{\xi, k, n} - A^1_{\xi, k, n})\bigg)\bigg\|_{N} \lesssim \mu_{q+1}\delta_{q+1} \lambda_{q+1}^\alpha \lambda_{q+1}^{N}, \,\,\, \forall N \geq 0
\end{equation}
\begin{equation}
    \bigg\|\bar D_{t, \Gamma}\mathcal{R} \div \bigg(\sum_{\zeta \in \Lambda_T, k, n} h_{\zeta, k, n+1}^2 (\bar A^2_{\zeta, k, n} - A^2_{\zeta, k, n})\bigg)\bigg\|_{N} \lesssim \mu_{q+1}\delta_{q+1} \lambda_{q+1}^\alpha \lambda_{q+1}^{N}, \,\,\, \forall N \geq 0
\end{equation}
where all the implicit bounds are dependent on $\Gamma$, $M$, $\alpha$ and $N$. 
\end{lem}

Finally, we derive the estimates for the error introduced by the divergence corrector.

\begin{lem} \label{div_corr_estim}
The following bounds are valid:
\begin{align}
    \left\|\mathcal{R} \div (w_{q+1}^{(p)} \otimes w_{q+1}^{(c)} + w_{q+1}^{(c)}\otimes w_{q+1}^{(p)} + w_{q+1}^{(c)} \otimes w_{q+1}^{(c)})\right\|_{N} &\lesssim \frac{\delta_{q+1} \lambda_q}{\lambda_{q+1}^{1-\alpha}} \lambda_{q+1}^N, \qquad \,\,\forall N \geq 0, \\
    \left\|\bar D_{t,\Ga} \mathcal{R} \div (w_{q+1}^{(p)} \otimes w_{q+1}^{(c)} + w_{q+1}^{(c)}\otimes w_{q+1}^{(p)} + w_{q+1}^{(c)} \otimes w_{q+1}^{(c)})\right\|_{N} &\lesssim \mu_{q+1} \frac{\delta_{q+1} \lambda_q}{\lambda_{q+1}^{1-\alpha}} \lambda_{q+1}^N, \,\,\, \forall N \geq 0,
\end{align}
where all the implicit bounds are dependent on $\Gamma$, $M$, $\alpha$, and $N$. 
\end{lem}

\begin{proof}
The proof follows directly from \cite[Lemma 4.13]{vikram}, modulo cosmetic changes.
\end{proof}

\subsection{Oscillation error \texorpdfstring{$T_{q+1, O}$}{rororo}}
Keeping in mind that $\div w_{q+1}^{(s)}=0$, we can rewrite the error as 
\begin{equation*}
    T_{q+1, O} = \underbrace{\mathcal{R}\div(X_{q, \Gamma} + \theta_{q+1}^{(p)} w_{q+1}^{(p)})}_{\text{Principal oscillation error}} + \underbrace{\mathcal{R} \div (\theta_{q+1}^{(p)} w_{q+1}^{(c)})}_{\text{Divergence-mean corrector error}}.
\end{equation*}
As usual, our aim is to cancel the error $X_{q, \Gamma}$ with the low frequency part of the interaction of $\theta_{q+1}^{(p)}$ and $w_{q+1}^{(p)}$. To that context, let us first recall that
\begin{equation*}
B_{\zeta, k, n} = b^2_{\zeta, k, n} (\nabla \Psi_k)^{-1} \zeta, \,\, \text{and}\,\, \bar B_{\zeta, k, n} = \bar b^2_{\zeta, k, n} (\nabla \tilde \Psi_k)^{-1} \zeta.
\end{equation*}
Making use of disjointness properties, we write
\begin{align*}
  &  \theta_{q+1}^{(p)} w_{q+1}^{(p)}
    =  \sum_{\zeta \in \Lambda_T, k, n} h_{\zeta, k, n+1}^2 \bar b_{\zeta, k, n}^2  (\nabla \tilde \Psi_k)^{-1} (\mathbb V_{\zeta}  \mathbb W_{\zeta})(\lambda_{q+1} \tilde \Psi_k) \\
  &=  \sum_{\zeta \in \Lambda_T, k, n} h_{\zeta, k, n+1}^2 \bar b_{\zeta, k, n}^2 (\nabla \tilde \Psi_k)^{-1} \zeta 
    + \sum_{\zeta \in \Lambda_T, k, n} h^2_{\zeta, k, n+1} \bar b_{\zeta, k, n}^2 (\nabla \tilde \Psi_k)^{-1} (\mathbb{P}_{\neq 0} \mathbb{V}_\zeta \mathbb{W}_\zeta)(\la_{q+1} \tilde\Psi_k) \\ 
        &=  -X_{q, \Gamma}  + \sum_{\zeta \in \Lambda_T, k, n} h_{\zeta, k, n+1}^2 (\bar B_{\zeta,k,n} - B_{\zeta, k, n})  
    + \sum_{\zeta \in \Lambda_T, k, n} h^2_{\zeta, k, n+1} \bar b_{\zeta, k, n}^2 (\nabla \tilde \Psi_k)^{-1} (\mathbb{P}_{\neq 0} \mathbb{V}_\zeta \mathbb{W}_\zeta)(\la_{q+1} \tilde\Psi_k).
    \end{align*}
As before we have $\mathbb P_{\neq 0} \mathbb V_\zeta \mathbb W_\zeta = \Omega_\zeta \zeta$, and we can write
\begin{align*}
    \mathcal{R}\div(X_{q, \Gamma} + \theta_{q+1}^{(p)} w_{q+1}^{(p)}) &= \mathcal{R} \div \bigg(\sum_{\zeta \in \Lambda_T, k, n} h^2_{\zeta, k, n+1} \Omega_\zeta(\lambda_{q+1} \tilde \Psi_k) \bar B_{\zeta, k, n}\bigg)\\
    & \qquad \qquad + \mathcal{R}\div \bigg(\sum_{\zeta \in \Lambda_T, k, n} h_{\zeta, k, n+1}^2 ( \bar B_{\zeta,k,n} - B_{\zeta, k, n})\bigg).
\end{align*}
In order to estimate each of the term above, we first need the following lemma.

\begin{lem} \label{big_a_bar_estim}
The following bounds are valid: 
\begin{align*}
 &\|\bar B_{\zeta, k, n}\|_N \lesssim \delta_{q+1, n} \lambda_q^N, \hspace{5cm} \forall N \in \{0,1,..., L-3\}, \\
 &\|\bar D_{t, \Gamma} \bar B_{\zeta, k, n} \|_N \lesssim \delta_{q+1, n} \tau_q^{-1} \lambda_q^N, \hspace{3.8cm} \forall N \in \{0,1,..., L-3\},\\
&\|\bar B_{\zeta, k, n}\|_{N+L-3} \lesssim \delta_{q+1, n} \lambda_q^{L-3} \ell_q^{-N},  \hspace{3.3cm}\, \forall N \geq 0,\\
&\|\bar D_{t, \Gamma} \bar B_{\zeta, k, n} \|_{N+L-3} \lesssim \delta_{q+1, n} \lambda_q^{L-3} \tau_q^{-1} \ell_q^{-N}, \hspace{2.1cm} \forall N \geq 0,
\end{align*}
where all the implicit bounds are dependent on $\Gamma$, $M$, $\alpha$ and $N$. Moreover, we have
\begin{equation*}
    \|\bar B_{\zeta, k, n} - B_{\zeta,k, n}\|_0 \lesssim  \delta_{q+1,n}\frac{\delta_{q+1}^{1/2} \lambda_q^{1/3}}{\delta_{q}^{1/2} \lambda_{q+1}^{1/3}}.
\end{equation*}
\end{lem}

\begin{proof}
Note that we only need to verify the last estimate, rest of the estimates follow from Corollary \ref{a_cor}, with the help of Lemma \ref{a_bar_estim} and Corollary \ref{Flow_gam_estim}. In what follows, we first write
\begin{eqnarray*}
        \bar B_{\zeta, k, n} - B_{\zeta, k, n} = (\bar b_{\zeta, k, n}^2 - b_{\zeta, k, n}^2) (\nabla \tilde \Psi_k)^{-1}\zeta 
        + b_{\zeta, k, n}^2\big( (\nabla \tilde \Psi_k)^{-1} - (\nabla \Psi_k)^{-1}\big) \zeta.
    \end{eqnarray*}
Therefore,
    \begin{equation*}
        \|\bar B_{\zeta, k, n} - B_{\zeta, k, n}\|_0 \lesssim \|\bar b_{\zeta, k, n}^2 - b_{\zeta, k, n}^2\|_0 + \|b_{\zeta, k, n}^2\|_0 \|(\nabla \tilde \Psi_k)^{-1} - (\nabla \Psi_k)^{-1}\|_0.
    \end{equation*}
For the first term, an application of the mean value inequality yields
    \begin{eqnarray*}
        \|\bar b_{\zeta, k, n}^2 - b_{\zeta, k, n}^2\|_0  \lesssim  \|\nabla \tilde \Psi_k - \nabla \Psi_k\|_0 \|\bar T_{q, n}\|_0 + \|T_{q, n} - \bar T_{q, n}\|_0 \|\nabla \tilde \Psi_k \|.
    \end{eqnarray*}
Since mollification estimates yield  
    \begin{equation*}
        \|T_{q, n} - \bar T_{q,n}\|_0 \lesssim \|\bar D_{t, \Gamma} T_{q, n}\|_0 \ell_{t,q} \lesssim \delta_{q+1, n} \tau_q^{-1} \ell_{t,q},
    \end{equation*}
we can make use of lemmas \ref{a_estim}, \ref{flow_stabil}, \ref{tmoli_estim}, and \ref{a_bar_estim} to conclude    \begin{eqnarray*}
        \|\bar B_{\zeta, k, n} - B_{\zeta, k, n}\|_0 \lesssim & \delta_{q+1, n}\bigg( \tau_q \frac{\delta_{q+1} \lambda_q^2 \ell_q^{-\alpha}}{\mu_{q+1}} + \tau_q^{-1} \ell_{t,q}\bigg) \lesssim \delta_{q+1, n} \bigg( \frac{\delta_{q+1}^{1/2} \lambda_q^{1/3}}{\delta_{q}^{1/2}\lambda_{q+1}^{1/3}} + \frac{\lambda_{q}^{2/3}}{\lambda_{q+1}^{2/3}} \lambda_{q+1}^\alpha \bigg).
    \end{eqnarray*}
This concludes the proof of the lemma by choosing $\alpha$ sufficiently small. 
\end{proof}

We can now follow \cite{vikram} to state results regarding the high-high-high oscillation error. Indeed, we have the following lemmas.

\begin{lem}
The following bounds are valid: 
\begin{equation} \label{high-high-estim}
    \bigg\|\mathcal{R} \div \bigg(\sum_{\zeta \in \Lambda_T, k, n} h^2_{\zeta, k, n+1} \Omega_\zeta(\lambda_{q+1} \tilde \Psi_k) \bar B_{\zeta, k, n}\bigg)\bigg\|_{N} \lesssim \frac{\delta_{q+1} \lambda_q}{\lambda_{q+1}^{1 - 2\alpha}} \lambda_{q+1}^N, \qquad \,\forall N \geq 0,
\end{equation}
\begin{equation}
    \bigg\|\bar D_{t,\Ga} \mathcal{R} \div \bigg(\sum_{\zeta \in \Lambda_T, k, n} h^2_{\zeta, k, n+1} \Omega_\zeta(\lambda_{q+1} \tilde \Psi_k) \bar B_{\zeta, k, n}\bigg)\bigg\|_{N} \lesssim \mu_{q+1}\delta_{q+1} \lambda_{q+1}^\alpha \lambda_{q+1}^N, \,\,\, \forall N \geq 0.
\end{equation}
where all the implicit bounds are dependent on $\Gamma$, $M$, $\alpha$, $b$, and $N$. 
\end{lem}

\begin{lem}
The following bounds are valid: 
\begin{equation}
    \bigg\|\mathcal{R} \div \bigg(\sum_{\zeta \in \Lambda_T, k, n} h_{\zeta, k, n+1}^2 (\bar B_{\zeta, k, n} - B_{\zeta, k, n})\bigg)\bigg\|_N \lesssim \delta_{q+1}\frac{\delta_{q+1}^{1/2} \lambda_{q}^{1/3}}{\delta_{q}^{1/2}\lambda_{q+1}^{1/3}} \lambda_{q+1}^\alpha \lambda_{q+1}^N, \,\,\, \forall N \geq 0.
\end{equation}
\begin{equation}
    \bigg\|\bar D_{t, \Gamma}\mathcal{R} \div \bigg(\sum_{\zeta \in \Lambda_T, k, n} h_{\zeta, k, n+1}^2 (\bar B_{\zeta, k, n} - B_{\zeta, k, n})\bigg)\bigg\|_{N} \lesssim \mu_{q+1}\delta_{q+1} \lambda_{q+1}^\alpha \lambda_{q+1}^{N}, \,\,\, \forall N \geq 0
\end{equation}
where all the implicit bounds are dependent on $\Gamma$, $M$, $\alpha$ and $N$. 
\end{lem}

\begin{lem} \label{div_corr_estim_01}
The following bounds are valid: 
\begin{align}
    \left\|\mathcal{R} \div (\theta_{q+1}^{(p)} w_{q+1}^{(c)} )\right\|_{N} &\lesssim \frac{\delta_{q+1} \lambda_q}{\lambda_{q+1}^{1-\alpha}} \lambda_{q+1}^N, \qquad\,\, \forall N \geq 0, \\
    \left\|\bar D_{t,\Ga} \mathcal{R} \div (\theta_{q+1}^{(p)} w_{q+1}^{(c)} )\right\|_{N} &\lesssim \mu_{q+1} \frac{\delta_{q+1} \lambda_q}{\lambda_{q+1}^{1-\alpha}} \lambda_{q+1}^N, \,\,\, \forall N \geq 0,
\end{align}
where all the implicit bounds are dependent on $\Gamma$, $M$, $\alpha$, and $N$. 
\end{lem}

\begin{proof}
The proof follows directly from \cite[Lemma 4.13]{vikram}, modulo cosmetic changes.
\end{proof}

\subsection{Residual errors \texorpdfstring{$R_{q+1, R}$}{rrrrr}, and \texorpdfstring{$T_{q+1, R}$}{rrrrr}}
\label{err.residual} 

First notice that, in view of $P_{q+1, \Gamma}$ and $Y_{q+1, \Gamma}$, we can write 
\begin{align*}
    R_{q+1, R} &=R_{q, \Gamma} + w_{q+1}^{(t)} \mathring \otimes w_{q+1}^{(t)} + w_{q+1} \mathring \otimes (v_q - \bar v_q) + (v_q - \bar v_q) \mathring \otimes w_{q+1} + (R_{q} - R_{q, 0}), \\
    T_{q+1, R} &= T_{q, \Gamma} + w_{q+1}^{(t)} \theta_{q+1}^{(t)} + w_{q+1} (\theta_q - \bar \theta_q) + (v_q - \bar v_q) \Theta_{q+1} + (T_{q} - T_{q, 0}).
\end{align*}
Let us now estimates for each of the terms above. To begin with, we have the following lemma.
\begin{lem}
The following bounds are valid: 
    \begin{align*}
        \|R_{q, \Gamma}\|_N, \,  \|T_{q, \Gamma}\|_N & \lesssim \frac{\delta_{q+1} \lambda_q}{\lambda_{q+1}} \lambda_{q+1}^N, \qquad \qquad \quad \,\,\,\,\forall N \geq 0, \\
                \|\bar D_{t, \Gamma} R_{q, \Gamma}\|_N, \,  \|\bar D_{t, \Gamma} T_{q, \Gamma}\|_N &\lesssim \tau_q^{-1} \frac{\delta_{q+1} \lambda_q}{\lambda_{q+1}} \lambda_{q+1}^N, \hspace{1.5cm} \forall N \geq 0,
    \end{align*}
where all the implicit bounds are dependent on $\Gamma$, $M$, $\alpha$, and $N$.
\end{lem}

\begin{proof}
The proof of the lemma follows from \cite[Lemma 4.14]{vikram}, thanks to the Proposition \ref{NewIter} and the Lemma \ref{w_t_estim}.
\end{proof}

\begin{lem} \label{Newt_err_estim}
The following bounds are valid:
\begin{align*}
    \|w_{q+1}^{(t)} \mathring \otimes w_{q+1}^{(t)}\|_N, \, \|w_{q+1}^{(t)} \theta_{q+1}^{(t)}\|_N & \lesssim \frac{\delta_{q+1} \lambda_q^{2/3}}{\lambda_{q+1}^{2/3}} \lambda_{q+1}^N, \qquad  \qquad \, \forall N \geq 0, \\
   \|\bar D_{t, \Gamma}(w_{q+1}^{(t)} \mathring \otimes w_{q+1}^{(t)})\|_N, \,  \|\bar D_{t, \Gamma}(w_{q+1}^{(t)} \theta_{q+1}^{(t)})\|_N&\lesssim \mu_{q+1} \frac{\delta_{q+1} \lambda_q^{2/3}}{\lambda_{q+1}^{2/3}} \lambda_{q+1}^N, \hspace{0.8cm} \forall N \geq 0,
\end{align*}
where all the implicit bounds are dependent on $\Gamma$, $M$, $\alpha$, and $N$.
\end{lem}

\begin{proof}
The proof of the lemma follows from \cite[Lemma 4.14]{vikram}, thanks to Lemma \ref{w_t_estim}.
\end{proof}

Let us now provide estimates for the error introduced by spatial mollification.

\begin{lem} \label{spatial_moli_estim_err_fin}
The following bounds are valid:
\begin{align} 
    \left\|w_{q+1} \mathring\otimes (v_q - \bar v_q) + (v_q - \bar v_q) \mathring\otimes w_{q+1}\right\|_{N} & \lesssim \frac{\delta_{q+1}^{1/2} \delta_q^{1/2} \lambda_q}{\lambda_{q+1}} \lambda_{q+1}^N, \qquad \forall N \in \{0,1,..., L\} \label{spat_moli_estim_1} \\
    \left\|w_{q+1} (\theta_q - \bar \theta_q) + (v_q - \bar v_q) \Theta_{q+1}\right\|_{N} & \lesssim \frac{\delta_{q+1}^{1/2} \delta_q^{1/2} \lambda_q}{\lambda_{q+1}} \lambda_{q+1}^N, \qquad \,\forall N \in \{0,1,..., L\} \label{spat_moli_estim_10} \\
    \left\|\bar D_{t,\Ga} \left(w_{q+1} \mathring\otimes (v_q - \bar v_q) + (v_q - \bar v_q) \mathring\otimes w_{q+1}\right)\right\|_{N} &\lesssim \delta_{q+1}^{1/2} \lambda_q^{1/3}\lambda_{q+1}^{2/3} \frac{\delta_{q+1}^{1/2} \delta_q^{1/2} \lambda_q}{\lambda_{q+1}} \lambda_{q+1}^N, \label{spat_moli_estim_2} \\ 
    & \hspace{3.8cm} \,\forall N \in \{0,1,...,L-1\}, \nonumber \\
    \left\|\bar D_{t,\Ga} \left(w_{q+1} (\theta_q - \bar \theta_q) + (v_q - \bar v_q) \Theta_{q+1}\right)\right\|_{N} &\lesssim \delta_{q+1}^{1/2} \lambda_q^{1/3}\lambda_{q+1}^{2/3} \frac{\delta_{q+1}^{1/2} \delta_q^{1/2} \lambda_q}{\lambda_{q+1}} \lambda_{q+1}^N, \label{spat_moli_estim_20} \\ 
    & \hspace{3.9cm} \forall N \in \{0,1,...,L-1\}, \nonumber \\
    \|R_q - R_{q, 0}\|_N, \, \|T_q - T_{q, 0}\|_N & \lesssim \frac{\delta_{q+1} \lambda_q}{\lambda_{q+1}} \lambda_{q+1}^N, \hspace{1.5cm} \forall N \in \{0,1,...,L\}, \label{spat_moli_estim_3} \\
    \|\bar D_{t, \Gamma}(R_{q} - R_{q, 0})\|_N, \, \|\bar D_{t, \Gamma}(T_{q} - T_{q, 0})\|_N & \lesssim \delta_{q+1} \delta_q^{1/2} \lambda_q \lambda_{q+1}^N, \hspace{1.0cm} \forall N \in \{0,1,...,L-1\}, \label{spat_moli_estim_4}
\end{align}
where all the implicit bounds are dependent on $\Gamma$, $M$, $\alpha$, and $N$. 
\end{lem}
\begin{proof}
For the proof of \eqref{spat_moli_estim_1}, \eqref{spat_moli_estim_2}, \eqref{spat_moli_estim_3} and \eqref{spat_moli_estim_4}, we refer to the reader \cite[Lemma 4.16]{vikram}. Here we only give details of the estimates \eqref{spat_moli_estim_10} and \eqref{spat_moli_estim_20}.

First observe that, thanks to Proposition \ref{moli} and inductive assumptions on $ \theta_q$, we have
\begin{equation*}
\|v_q - \bar v_q\|_N, \, \|\theta_q - \bar \theta_q\|_N \lesssim \delta_{q}^{1/2} \frac{\lambda_q}{\lambda_{q+1}} \lambda_{q+1}^{N}.
\end{equation*}
Therefore, we can use estimates of Lemma \ref{velo_estim_fin} to conclude \eqref{spat_moli_estim_10}. Indeed, we have
    \begin{eqnarray*}
        \left\|w_{q+1} (\theta_q - \bar \theta_q) + (v_q - \bar v_q) \Theta_{q+1}\right\|_{N} \lesssim \|w_{q+1}\|_N \|\theta_q - \bar \theta_q\|_0 + \|\Theta_{q+1}\|_0 \|v_q - \bar v_q\|_N 
        \lesssim \frac{\delta_{q+1}^{1/2} \delta_q^{1/2}\lambda_q}{\lambda_{q+1}} \lambda_{q+1}^N,
    \end{eqnarray*}
To estimate \eqref{spat_moli_estim_20}, we first recall that we have (cf. Lemma \ref{div_corr_estim} and Lemma \ref{Newt_err_estim})     
\begin{equation*}
        \|\bar D_{t, \Gamma} w_{q+1}^{(s)}\|_N, \, \|\bar D_{t, \Gamma} \theta_{q+1}^{(s)}\|_N \lesssim \mu_{q+1} \delta_{q+1}^{1/2} \lambda_{q+1}^N, \quad 
        \|\bar D_{t, \Gamma} w_{q+1}^{(t)}\|_N, \, \|\bar D_{t, \Gamma} \theta_{q+1}^{(t)}\|_N \lesssim \mu_{q+1} \frac{\delta_{q+1}^{1/2}\lambda_q^{1/3}}{\lambda_{q+1}^{1/3}} \lambda_{q+1}^N,
    \end{equation*}
Moreover, following \cite{vikram}, we have
\begin{equation*}
        \|\bar D_{t, \Gamma} (v_q - \bar v_q)\|_N \lesssim \delta_{q+1}^{1/2} \lambda_q^{1/3}\lambda_{q+1}^{2/3} \frac{\delta_q^{1/2} \lambda_q}{\lambda_{q+1}} \lambda_{q+1}^N, \,\,\, \forall N \in \{0,1,...,L-1\},
    \end{equation*}
while for the other term, we notice that
    \begin{equation*}
        \bar D_{t, \Gamma} (\theta_q - \bar \theta_q) = w_{q+1}^{(t)} \cdot \nabla (\theta_q - \bar \theta_q) + \bar D_t (\theta_q - \bar \theta_q),
    \end{equation*}
where 
    \begin{equation*}
        \|w_{q+1}^{(t)} \cdot \nabla (\theta_q - \bar \theta_q)\|_N \lesssim \frac{\delta_{q+1}^{1/2}\lambda_q^{1/3}}{\lambda_{q+1}^{1/3}} \delta_q^{1/2} \frac{\lambda_q}{\lambda_{q+1}} \lambda_{q+1}^{N+1}, \,\,\, \forall N \in \{0,1...,L-1\}.
    \end{equation*}
Therefore, we need to estimate the term $\bar D_t (v_q - \bar v_q)$. To that context, we write
    \begin{equation*}
        \bar D_{t} (\theta_q - \bar \theta_q) =  (\partial_t \theta_q + v_q \cdot \nabla \theta_q) - (\partial_t \theta_q + v_q \cdot \nabla \theta_q)*\zeta_{\ell_q} + (\bar v_q - v_q) \cdot \nabla \theta_q + \div((v_q \theta_q)*\zeta_{\ell_q} - \bar v_q \bar \theta_q).
    \end{equation*}
Thanks to the Euler-Biussinesq-Reynolds system \ref{ER}, we can write
    \begin{equation*}
        \|(\partial_t v_q + v_q \cdot \nabla \theta_q) - (\partial_t v_q + v_q \cdot \nabla \theta_q)*\zeta_{\ell_q}\|_N \lesssim  \|\div T_q - \div T_{q, 0}\|_N,
    \end{equation*}
and Proposition \ref{moli} yields 
\begin{equation*}
        \|\div T_q - \div T_{q, 0}\|_N \lesssim \frac{\delta_{q+1} \lambda_q^2}{\lambda_{q+1}} \lambda_{q+1}^N, \,\,\, \forall N \in \{0,1,...,L-1\}.
    \end{equation*}
    Moreover, we have
    \begin{equation*}
        \|(\bar v_q - v_q) \cdot \nabla \theta_q\|_N \lesssim \|\bar v_q - v_q\|_N \|\theta_q\|_1 + \|\bar v_q - v_q\|_0\|\theta_q\|_{N+1} \lesssim \frac{\delta_q \lambda_q^2}{\lambda_{q+1}} \lambda_{q+1}^N, \,\,\, \forall N \in \{0,1,...,L-1\},
    \end{equation*}
and Constantin-E-Titi commutator estimate (cf. Proposition \ref{CET_comm}) yields
    \begin{equation*}
        \|\div((v_q \theta_q)*\zeta_{\ell_q} - \bar v_q \bar \theta_q)\|_N \lesssim \frac{\delta_q \lambda_q^2}{\lambda_{q+1}}\lambda_{q+1}^{N}, \,\,\, \forall N \geq 0. 
    \end{equation*}
Therefore, we conclude that 
    \begin{equation*}
        \|\bar D_{t, \Gamma} (\theta_q - \bar \theta_q)\|_N \lesssim \delta_{q+1}^{1/2} \lambda_q^{1/3}\lambda_{q+1}^{2/3} \frac{\delta_q^{1/2} \lambda_q}{\lambda_{q+1}} \lambda_{q+1}^N, \,\,\, \forall N \in \{0,1,...,L-1\}.
    \end{equation*}
This essentially finishes the proof of the lemma.
\end{proof}

\subsection{The Pressure term} Recalling the expression \eqref{p_q_Gam} for $p_{q, \Gamma}$, we write 
\begin{align*}
    p_{q+1} &= p_q + \sum_{n=1}^\Gamma p_{q+1, n}^{(t)}  - \frac{|w_{q+1}^{(t)}|^2}{2} + \langle \bar v_q - v_q, w_{q+1} \rangle \\
    &\quad - \sum_{n=0}^{\Gamma-1} \Delta^{-1} \div \div\big(R_{q,n} + \sum_{k \in \mathbb Z_{q,n}} \sum_{\xi \in\Lambda_R} g_{\xi, k, n+1}^2  A^1_{\xi, k, n} + \sum_{k \in \mathbb Z_{q,n}} \sum_{\zeta \in\Lambda_T}  h_{\zeta, k, n+1}^2  A^2_{\zeta, k, n} \big).
\end{align*}

\begin{lem} 
The following bounds are valid: 
\begin{equation}
    \|p_{q+1}\|_N \lesssim \frac{\delta_{q+1}^{1/2}\delta_q^{1/2}\lambda_q^{1/3}}{\lambda_{q+1}^{1/3}} \lambda_{q+1}^N, \,\,\, \forall N \in \{1, 2,...,L\},
\end{equation}
where all the implicit bounds are dependent on $\Gamma$, $M$, $\alpha$, and $N$.
\end{lem}

\begin{proof}
    We have 
    \begin{align*}
        \|p_{q+1}\|_N &\lesssim  \|p_q\|_N + \||w_{q+1}^{(t)}|^2\|_N + \|\langle \bar v_q - v_q, w_{q+1} \rangle\|_N + \sum_{n=1}^\Gamma \|p_{q+1, n}^{(t)}\|_N  \\
        & \quad + \sum_{n = 0}^{\Gamma-1}\|\Delta^{-1} \div \div \big( R_{q, n} + \sum_{k \in \mathbb Z_{q,n}} \sum_{\xi \in\Lambda_R} g_{\xi, k, n+1}^2  A^1_{\xi, k, n} + \sum_{k \in \mathbb Z_{q,n}} \sum_{\zeta \in\Lambda_T}  h_{\zeta, k, n+1}^2  A^2_{\zeta, k, n} \big)\|_N 
    \end{align*}
First notice that, inductive assumptions give us
    \begin{equation*}
        \|p_q\|_N \lesssim \frac{\delta_{q+1}^{1/2}\delta_q^{1/2} \lambda_q^{1/3}}{\lambda_{q+1}^{1/3}}, \,\,\, \forall N \in \{1, 2,...,L\}.
    \end{equation*}
Next, recall that $p_{q+1, n}^{(t)} = \sum_k \tilde \eta_k p_{k, n}$, and \eqref{LocalNewt} implies 
    \begin{equation*}
        p_{k,n} = \Delta^{-1} [-2 \div (w_{k, n} \cdot \nabla \bar v_q) + \div (\theta_{k,n} e_2)]=   \Delta^{-1} \div \div (2 \psi_{k, n} \nabla^\perp \bar v_q + A_{\varphi}),
    \end{equation*}
where $A_{\varphi}$ denotes the $2 \times 2$ matrix with entries $A_{\varphi} (1,1)= A_{\varphi} (1,2)=0$, and $A_{\varphi} (2,1) = A_{\varphi} (2,2)= \varphi_{k,n}$. Since $\Delta^{-1} \div \div$ is a Calder\'on-Zygmund type operator, we have, for $N \geq 0$, 
    \begin{eqnarray*}
        \| p_{k,n}\|_{N+\alpha} &\lesssim & \|\psi_{k, n}\|_{N+\alpha}\|\nabla \bar v_q\|_\alpha + \|\psi_{k, n}\|_{\alpha} \|\nabla \bar v_q\|_{N+\alpha} +  \|\varphi_{k, n}\|_{N+\alpha} \\ 
        & \lesssim & \frac{\delta_{q+1} }{\mu_{q+1}}\delta_q^{1/2} \lambda_q \lambda_{q+1}^{N + 2\alpha} \lesssim \frac{\delta_{q+1}^{1/2} \delta_{q}^{1/2}\lambda_q^{1/3}}{\lambda_{q+1}^{1/3}} \lambda_{q+1}^N.
    \end{eqnarray*}
Fir the other term, thanks to sufficiently small $\alpha > 0$, Lemma \ref{a_cor} and Proposition \ref{NewIter}, we have for $N \geq 1$, 
    \begin{align*}
       & \|\Delta^{-1} \div \div \big( R_{q, n} + \sum_{k \in \mathbb Z_{q,n}} \sum_{\xi \in\Lambda_R} g_{\xi, k, n+1}^2  A^1_{\xi, k, n} + \sum_{k \in \mathbb Z_{q,n}} \sum_{\zeta \in\Lambda_T}  h_{\zeta, k, n+1}^2  A^2_{\zeta, k, n} \big) \|_{N+ \alpha} \\
       & \qquad \lesssim \|R_{q,n}\|_{N+\alpha} + \|A_{\sigma, k, n}\|_{N+\alpha} 
      \lesssim \frac{\delta_{q+1}^{1/2}\delta_q^{1/2}\lambda_q^{1/3}}{\lambda_{q+1}^{1/3}}.
        \end{align*}
Finally, the estimates for $|w_{q+1}^{(t)}|^2$ and $\langle \bar v_q - v_q, w_{q+1} \rangle$ are the same as in Lemma \ref{Newt_err_estim} and Lemma~\ref{spatial_moli_estim_err_fin}, respectively. This finishes the proof.
\end{proof}

\begin{cor} \label{press_corol_fin}
The following bounds are valid: 
\begin{equation}
    \|p_{q+1}\|_N \leq M^2 \delta_{q+1} \lambda_{q+1}^N, \,\,\, \forall N\in\{1,2,...,L\}.
\end{equation}
\end{cor}

\begin{proof}
The proof follows from \cite[Lemma 4.18]{vikram}.
\end{proof}

\subsection{Verification of inductive estimates at level $q+1$} Note that Corollary~\ref{velo_corrol_fin} and Corollary~\ref{press_corol_fin} confirm the inductive estimates regarding the velocity field, temperature and the pressure. For the desired estimates on the Reynolds stresses, we have the following result.
\begin{cor}
The following bounds are valid: 
    \begin{equation} \label{Spatial_estim_qplus1}
        \|R_{q+1}\|_N, \, \|T_{q+1}\|_N\leq \delta_{q+2} \lambda_{q+1}^{N-2\alpha}, \hspace{2.5cm}\,\, \forall N \in \{0,1,...,L\},
    \end{equation}
    \begin{equation} \label{Material_estim_qplus1}
        \|D_{t, q+1} R_{q+1}\|_N, \, \|D_{t, q+1} T_{q+1}\|_N \leq \delta_{q+2} \delta_{q+1}^{1/2} \lambda_{q+1}^{N+1-2\alpha}, \,\,\, \forall N \in \{0,1,...,L-1\}.
    \end{equation}
\end{cor}
\begin{proof}
We have all the required estimates to furnish the proof. In fact, in view of the estimates obtained in the previous sections, we have
\begin{equation*}
        \|R_{q+1}\|_N, \, \|T_{q+1}\|_N \lesssim \bigg(\frac{\delta_q^{1/2}\delta_{q+1}^{1/2}\lambda_q}{\lambda_{q+1}} + \frac{\delta_{q+1}\lambda_q^{2/3}}{\lambda_{q+1}^{2/3}} + \frac{\delta_{q+1}\lambda_q}{\lambda_{q+1}} + \frac{\delta_{q+1}^{3/2} \lambda_q^{1/3}}{\delta_q^{1/2}\lambda_{q+1}^{1/3}}\bigg) \lambda_{q+1}^{5\alpha} \lambda_{q+1}^N.
    \end{equation*}
It is easy to see, thanks to $1 < b < \frac{1 + 3 \beta}{6 \beta}$ and $\beta<1/3$, that the last term of the above expression is dominant and we have (for sufficiently large $a_0$, sufficiently small $\alpha > 0$ and for all $N \in \{0,1,..., L\}$):
\begin{equation*}
        \|R_{q+1}\|_N, \, \|T_{q+1}\|_N \lesssim \frac{\delta_{q+1}^{3/2} \lambda_q^{1/3}}{\delta_q^{1/2}\lambda_{q+1}^{1/3}} \lambda_{q+1}^{5\alpha} \lambda_{q+1}^N \leq \delta_{q+2} \lambda_{q+1}^{N-2\alpha}.
    \end{equation*}
To obtain material derivative estimates associated with $v_{q+1}$, we first write
    \begin{align*}
        \|D_{t,q+1} R_{q+1}\|_N &\lesssim \|\bar D_{t,\Ga} R_{q+1}\|_N + \|(v_q - \bar v_q) \cdot\na R_{q+1}\|_N + \|w_{q+1}^{(s)} \cdot \nabla R_{q+1}\|_{N}, \\
        \|D_{t,q+1} T_{q+1}\|_N &\lesssim \|\bar D_{t,\Ga} T_{q+1}\|_N + \|(v_q - \bar v_q) \cdot\na T_{q+1}\|_N + \|w_{q+1}^{(s)} \cdot \nabla T_{q+1}\|_{N}
    \end{align*}
Again, collecting all the estimates derived in the previous sections, for all $N \in \{0,1,..., L-1\}$, we conclude that
    \begin{align*}
        \|\bar D_{t, \Gamma} R_{q+1}\|_N, \, \|\bar D_{t, \Gamma} T_{q+1}\|_N &\lesssim \bigg(\frac{\delta_{q+1}\delta_q^{1/2}\lambda_q^{5/3}}{\lambda_{q+1}^{2/3}} + \delta_{q+1}\delta_q^{1/2} \lambda_q + \delta_{q+1}^{3/2}\lambda_q^{2/3}\lambda_{q+1}^{1/3} + \frac{\delta_{q+1}^{3/2}\lambda_q^{5/3}}{\lambda_{q+1}^{2/3}}  \\ 
        & \qquad + \frac{\delta_{q+1}\delta_q^{1/2}\lambda_q^2}{\lambda_{q+1}} + \frac{\delta_{q+1}^{3/2}\lambda_q^{4/3}}{\lambda_{q+1}^{1/3}} + \frac{\delta_{q+1}\delta_q^{1/2}\lambda_q^{4/3}}{\lambda_{q+1}^{1/3}} \bigg) \lambda_{q+1}^{5\alpha}\lambda_{q+1}^N.
    \end{align*}
Therefore, thanks to $b < \frac{1 + 3 \beta}{6 \beta} < \frac{1}{3\beta}$, we have for sufficiently small $\alpha > 0$,
\begin{equation*}
        \|\bar D_{t, \Gamma} R_{q+1}\|_N, \,  \|\bar D_{t, \Gamma} T_{q+1}\|_N \leq \delta_{q+2} \delta_{q+1}^{1/2} \lambda_{q+1}^{N+1-2\alpha}, \,\,\, \forall N \in \{0,1,...,L-1\}.
\end{equation*}
For the rest of the terms, thanks to Lemma \ref{velo_estim_fin}, we have    
\begin{align*}
        \|(\bar v_q - v_q) \cdot \nabla R_{q+1}\|_N, \,  \|(\bar v_q - v_q) \cdot \nabla T_{q+1}\|_N & \lesssim \delta_{q+2} \delta_{q+1}^{1/2} \lambda_{q+1}^{N+1-2\alpha}, \\
        \|w_{q+1}^{(s)} \cdot \nabla R_{q+1}\|_N, \,  \|w_{q+1}^{(s)} \cdot \nabla T_{q+1}\|_N &\lesssim \delta_{q+2}\delta_{q+1}^{1/2}\lambda_{q+1}^{N +1 - 2\alpha}.
    \end{align*}
This finishes the proof of the lemma.
\end{proof}

\appendix

\section{Well-posedness Results for the Linearized Euler-Boussinesq Equations} 
\label{wl-psd}

Here we discuss, for any fixed $T>0$ and dimension $d \geq 2$, the well-posedness results of the linearized Euler-Boussinesq system of equations on the domain $[-T, T] \times \mathbb T^d$. In what follows, we consider the following linearized Euler-Boussinesq
equations:
\begin{equation} \label{LinEul}
    \begin{cases}
        \partial_t w + v \cdot \nabla w + w \cdot \nabla v + \nabla p + \Theta e_2= F \\ 
         \partial_t \Theta + v \cdot \nabla \Theta + w \cdot \nabla \theta  = G\\
        \div w = 0, \\ 
        w \big|_{t = 0} = w_0, \quad \Theta \big|_{t = 0} = \Theta_0,
    \end{cases}
\end{equation}
Here the unknowns are the scalar temperature $\Theta:[-T,T] \times \mathbb T^d \rightarrow \mathbb R$, the vector-field $w:[-T,T] \times \mathbb T^d \rightarrow \mathbb R^d$, and the pressure $p:[-T,T] \times \mathbb T^d \rightarrow \mathbb R$. Moreover, $u$ is a given divergence-free vector field. Furthermore, $\theta: [-T,T] \times \mathbb T^d \rightarrow \mathbb R$, $F:[-T,T] \times \mathbb T^d \rightarrow \mathbb R^d$, $G:[-T,T] \times \mathbb T^d \rightarrow \mathbb R$, $w_0:\mathbb T^d \rightarrow \mathbb R^d$ and $\Theta_0:\mathbb T^d \rightarrow \mathbb R^d$ are all known smooth functions, the source term $F$ is divergence-free, and the source term $G$ has mean-zero. Note that, we can calculate the pressure, up to a constant by 
\begin{equation} \label{press_LinEul}
    - \Delta p = \div \div (v \otimes w + w \otimes v) + \div(\Theta e_2)= 2 \div (w \cdot \nabla v) + \div(\Theta e_2).
\end{equation}
Hence, we can recast the first equation of the above system \eqref{LinEul} as follows: 
\begin{equation} \label{LinEul2}
    \begin{cases}
        \partial_t w + v \cdot \nabla w + (\I - 2 \nabla \Delta^{-1}\div)(w \cdot \nabla v) - \nabla \Delta^{-1}\div (\Theta e_2) + \Theta e_2= F \\ 
        w \big|_{t = 0} = w_0.
    \end{cases}
\end{equation}

Now we are in a position to state and prove the global well-posedness results for the system of equations given by \eqref{LinEul}. In fact, we will assume that the pressure $p$ has mean zero to make sure it is uniquely determined by the pressure equation \eqref{press_LinEul}.

\begin{prop}
Let us fix $T > 0$, $N \in \mathbb N \setminus \{0\}$, $\alpha \in (0,1)$. Moreover, assume that the velocity and the temperature $v, \theta \in C([-T,T]; C^{N + 1+\alpha}(\mathbb T^d))$, the source terms $F, G \in C([-T, T]; C^{N+\alpha}(\mathbb T^d))$ and $w_0, \Theta_0 \in C^{N+\alpha}(\mathbb T^d)$ such that $w_0$ is divergence-free. Then, there exists a unique solution $(w, \Theta, p)$ of the system \eqref{LinEul} in the following regularity class:
$$
w, \Theta \in C([-T, T]; C^{N+\alpha}(\mathbb T^d)) \cap C^{1}([-T,T]; C^{N -1 + \alpha }(\mathbb T^d)), \,\, \text{and}\,\,
p \in C([-T,T]; C^{N+\alpha}(\mathbb T^d)).
$$
\end{prop}

\begin{proof}
We first prove the local existence of solutions to \eqref{LinEul}. In fact, we will establish well-posedness results for $w$ and $\Theta$ only, in the required regularity class, since the estimate for the pressure $p$ will follow from the pressure equation \eqref{press_LinEul}. For the local existence of solutions, we first let $q \in \mathbb N \setminus \{0\}$ and inductively define $w_q, \Theta_q \in C_tC_x^{N+\alpha} \cap C_t^1C_x^{N-1+\alpha}$ as solutions of the following system:
    \begin{equation*}
        \begin{cases}
            \partial_t w_q + v\cdot \nabla w_q + \mathcal{T}(w_{q-1}\cdot \nabla v) - \mathcal{S}(\Theta_{q-1} e_2) + \Theta_{q-1} e_2= F \\ 
            \partial_t \Theta_q + v\cdot \nabla \Theta_q + w_{q-1} \cdot \nabla \theta = G\\
            w_q \big|_{t=0} = w_0, \quad \Theta_q \big|_{t=0} = \Theta_0.
        \end{cases}
    \end{equation*}
Here $\mathcal{T}$ and $\mathcal{S}$ are the Calder\'on-Zygmund type operators $(\I - 2 \nabla \Delta^{-1}\div)$ and $ \nabla \Delta^{-1} \div $ respectively. We can now apply the standard Cauchy-Lipschitz theory, for ordinary differential equations, to prove the existence and uniqueness of the solutions (for all time) $w_q$, and $\Theta_q$. Moreover, notice that 
\begin{align*}
        w_q(x,t) &= w_0(\Psi(x,t)) - \int_0^t \mathcal{T}(w_{q-1}\cdot \nabla v)(X(\Psi(x, t), s),s)ds 
        + \int_0^t F(X(\Psi(x,t),s),s)ds, \\
        & \qquad + \int_0^t \mathcal{S}(\Theta_{q-1} e_2)(X(\Psi(x, t), s),s)ds - \int_0^t \Theta_{q-1} e_2(X(\Psi(x,t),s),s)ds \\
        \Theta_q(x,t) &= \Theta_0(\Psi(x,t)) - \int_0^t (w_{q-1}\cdot \nabla \theta)(X(\Psi(x, t), s),s)ds + \int_0^t G(X(\Psi(x,t),s),s)ds,
    \end{align*}
where, as before, we denote by $X$ and $\Psi$, the forward and backwards flows of $v$, respectively. Notice that, in view of the above expressions and induction on $q$, we can conclude $w_q, \Theta_q \in C_tC_x^{N+\alpha}$. Moreover, we conclude that $w_q, \Theta_q \in C^1_tC_x^{N-1+\alpha}$, since 
    \begin{align*}
       \partial_t w_q &= - v\cdot \nabla w_q - \mathcal{T}(w_{q-1} \cdot \nabla v) 
       - \mathcal{S}(\Theta_{q-1} e_2) + \Theta_{q-1} e_2 + F, \\
       \partial_t \Theta_q &= - v\cdot \nabla \Theta_q - w_{q-1} \cdot \nabla \theta + G.
    \end{align*}
For a sufficiently small $\tau > 0$, it is also easy to see that both the sequences $\{w_m\}$ and $\{\Theta_m\}$ are Cauchy in the space $C([-\tau, \tau]; C^{N+\alpha}(\mathbb T^d)) \cap C^{1}([-\tau,\tau]; C^{N -1 + \alpha }(\mathbb T^d))$. Indeed, to see this, let us denote by $v_q = w_{q+1} - w_q$ and $z_q = \Theta_{q+1} - \Theta_q$. Then, they satisfy 
    \begin{equation*}
        \begin{cases}
            \partial_t v_q + v \cdot \nabla v_q + \mathcal{T}(v_{q-1}\cdot\nabla v) - \mathcal{S}(z_{q-1} e_2) + z_{q-1} e_2= 0, \\ 
            \partial_t z_q + v \cdot \nabla z_q - v_{q-1} \cdot \nabla \theta =0, \\
            v_q \big|_{t = 0} = 0, \quad z_q \big|_{t = 0} = 0.
        \end{cases}
    \end{equation*}
Then, thanks to Proposition \ref{transport_estim} and under the assumption $\tau \|v\|_{N+\alpha} < 1$, we have    \begin{equation*}
        \|v_q\|_\alpha \lesssim \tau \|\mathcal{T}(v_{q-1}\cdot \nabla v) - \mathcal{S}(z_{q-1} e_2) + z_{q-1} e_2\|_\alpha \lesssim \tau \left( \|v_{q-1}\|_\alpha + \|z_{q-1}\|_\alpha \right),
    \end{equation*}
Moreover, for all $1 \leq k \leq N$, we have
    \begin{equation*}
        [v_q]_{k+\alpha} \lesssim \tau \|\mathcal{T}(v_{q-1} \cdot \nabla v) - \mathcal{S}(z_{q-1} e_2) + z_{q-1} e_2 \|_{k+\alpha} \lesssim \tau \left( \|v_{q-1}\|_{k+\alpha} + \|z_{q-1}\|_{k+\alpha} \right).
    \end{equation*}
Therefore, we conclude that 
\begin{equation*}
        \|v_q\|_{N+\alpha} \leq C \tau \left( \|v_{q-1}\|_{N+\alpha} + \|z_{q-1}\|_{N+\alpha} \right),
\end{equation*}
where the constant $C$ depends on $N$, $\alpha$ and $\|v\|_{N+1+\alpha}$. Similarly, under the same assumption $\tau \|v\|_{N+\alpha} < 1$, Proposition \ref{transport_estim} implies
     \begin{equation*}
        \|z_q\|_{N+\alpha} \leq C \tau \|v_{q-1}\|_{N+\alpha}.
    \end{equation*}
If we choose $\tau$ such that $C\tau < 1/2$, then, combining above two estimates, we obtain 
    \begin{equation*}
     \left(\|v_q \|_{N+\alpha} + \|z_q \|_{N+\alpha} \right) \leq \frac{1}{2^q} \left( \|v_0\|_{N+\alpha} + \|z_0\|_{N+\alpha} \right).
    \end{equation*}
In other words, both the sequences $\{w_q\}$ and $\{\Theta_q\}$ are Cauchy in $C([-\tau, \tau]; C^{N+\alpha}(\mathbb T^d))$. To conclude that both the sequences are also Cauchy in $C^1([-\tau, \tau]; C^{N-1+\alpha}(\mathbb T^d))$, we see that for $q, q' \in \mathbb N \setminus \{0\}$, we have 
\begin{align*}
\|\partial_t (w_q - w_{q'})\|_{N-1-\alpha} &+ \|\partial_t (\Theta_q - \Theta_{q'})\|_{N-1-\alpha} 
\lesssim \|w_q - w_{q'}\|_{N+\alpha} + \|\Theta_{q} - \Theta_{q'}\|_{N+\alpha} \\
& \quad +  \|w_{q-1} - w_{q' -1}\|_{N-1+\alpha} + \|\Theta_{q-1} - \Theta_{q' -1}\|_{N-1+\alpha}. 
\end{align*}
Then, by defining,  
    \begin{align*}
        w = \lim_{q \rightarrow \infty} w_q \in C([-\tau, \tau]; C^{N+\alpha}(\mathbb T^d)) \cap C^{1}([-\tau,\tau]; C^{N -1 + \alpha }(\mathbb T^d)), \\
        \Theta = \lim_{q \rightarrow \infty} \Theta_q \in C([-\tau, \tau]; C^{N+\alpha}(\mathbb T^d)) \cap C^{1}([-\tau,\tau]; C^{N -1 + \alpha }(\mathbb T^d)),
    \end{align*}
we can conclude that $w$ and $\Theta$ solves the system \eqref{LinEul} on the time interval $[-\tau, \tau]$. This proves the existence of local solutions to \eqref{LinEul}.

It is also clear that such local solutions are unique. Indeed, by denoting $w_1$ and $w_2$, two solutions of the first equation of \eqref{LinEul}, and $\Theta_1$ and $\Theta_2$, two solutions of the second equation of \eqref{LinEul}, we see that $y := w_2 - w_1$ and $z := \Theta_2 - \Theta_1$ satisfy
    \begin{equation*}
    \begin{cases}
        \partial_t y + v \cdot \nabla y + \mathcal{T}(y \cdot \nabla v) - \mathcal{S}(z e_2) + z e_2= 0, \\ 
            \partial_t z + v\cdot \nabla z - y \cdot \nabla \theta =0,  \\
        y \big|_{t = 0} = 0, \quad z \big|_{t = 0} = 0.
    \end{cases} 
    \end{equation*}
Again, a simple application of Proposition~\ref{transport_estim}, under the assumption $|t| \|v\|_1 \leq 1$, reveals that
    \begin{align*}
        \|y(\cdot, t)\|_{\alpha} &\lesssim \int_0^t\|\mathcal{T}(y\cdot \nabla v)(\cdot, s) - \mathcal{S}(z e_2) + z e_2\|_\alpha ds \lesssim \int_0^t \left( \|y(\cdot, s)\|_\alpha + \|z(\cdot, s)\|_\alpha \right) ds, \\
     \|z(\cdot, t)\|_{\alpha} &\lesssim \int_0^t\| (y\cdot \nabla \theta)(\cdot, s)\|_\alpha ds \lesssim \int_0^t  \|y(\cdot, s)\|_\alpha ds, 
    \end{align*}
Therefore, a standard Gr\"onwall's argument reveals that $y = z= 0$, inside the time interval $(-\|v\|_1^{-1}, \|v\|_1^{-1})$. To obtain the global uniqueness, we simply cover the interval $[-T,T]$ by intervals of length $\|v\|_1^{-1}$.

To prove the existence of global-in-time solutions, we first denote by $t_p = \frac{1}{2} p \tau$, for $p \in \mathbb Z$. Moreover, we denote the unique solutions by $w^0$ and $\Theta^0$ on the interval $(-\tau, \tau)$, with initial data $w^0\big|_{t=0}=w_0$ and $\Theta^0\big|_{t=0}=\Theta_0$ respectively. Furthermore, we define $w^p$, and $\Theta^p$ to be the solutions on $(t_p - \tau, t_p + \tau) \cap [-T, T]$ with initial conditions 
    \begin{equation*}
        w^p \big|_{t = t_p} = 
        \begin{cases}
            w^{p-1}(t_p), \,\,\, \text{if } p > 0, \\ 
            w^{p+1}(t_p), \,\,\, \text{if } p < 0,
        \end{cases} \quad
        \text{and} \quad \Theta^p \big|_{t = t_p} = 
        \begin{cases}
            \Theta^{p-1}(t_p), \,\,\, \text{if } p > 0, \\ 
            \Theta^{p+1}(t_p), \,\,\, \text{if } p < 0,
        \end{cases}
    \end{equation*}
respectively. Thanks to uniqueness result, we can now define global-in-time solutions $w(\cdot, t) = w^{p}(\cdot, t)$, for $t \in (t_p - \tau, t_p + \tau)$, and $\Theta(\cdot, t) = \Theta^{p}(\cdot, t)$, for $t \in (t_p - \tau, t_p + \tau)$. This finishes the proof of the proposition. 
\end{proof}

\begin{rem}
If we are given that $v, \theta, F, G \in C^\infty ([-T,T] \times \mathbb T^d)$, and $w_0, \Theta_0 \in C^\infty(\mathbb T^d)$, then we conclude that the unique solutions $w$ and $\Theta$ belongs to the regularity class $C^\infty ([-T,T] \times \mathbb T^d)$. In fact, the conclusion $w, \Theta \in C_t C^\infty_x$ is clear, and for the regularity in time, observe that
\begin{align*}
    \partial_t w &= - v \cdot \nabla w - (\I - 2\nabla \Delta^{-1} \div)(w\cdot \nabla v) - \nabla \Delta^{-1}\div (\Theta e_2) + \Theta e_2 + F, \\
    \partial_t \Theta &= - v \cdot \nabla \Theta - w \cdot \nabla \theta + G
\end{align*}
and, therefore, we conclude
\begin{equation*}
    w, \Theta \in C_t^p C_x^\infty \implies \partial_t w, \partial_t \Theta \in C_t^p C_x^\infty \implies w, \Theta \in C_t^{p+1} C_x^\infty.
\end{equation*}
\end{rem}

\medskip\noindent
{\bf Acknowledgments:} {The author wishes to thank Vikram Giri for enlightening discussions on this topic, and also acknowledges the support of the Department of Atomic Energy, Government of India, under the project no.$12$-R$\&$D-TFR-$5.01$-$0520$, and the DST-SERB SJF grant DST/SJF/MS/$2021$/$44$. } 

\end{document}